\documentclass[oneside,english]{amsart}
\usepackage[T1]{fontenc}
\usepackage[latin9]{inputenc}
\usepackage{amstext}
\usepackage{amsthm}
\usepackage{amssymb}

\makeatletter
\numberwithin{equation}{section}
\numberwithin{figure}{section}
\theoremstyle{plain}
\newtheorem{thm}{\protect\theoremname}
\theoremstyle{plain}
\newtheorem{lem}[thm]{\protect\lemmaname}
\theoremstyle{plain}
\newtheorem{prop}[thm]{\protect\propositionname}
\theoremstyle{plain}
\newtheorem{cor}[thm]{\protect\corollaryname}
\theoremstyle{remark}
\newtheorem{rem}[thm]{\protect\remarkname}
\theoremstyle{definition}
\newtheorem{example}[thm]{\protect\examplename}

\numberwithin{thm}{section}
\usepackage{lineno,xcolor}
\usepackage{chngcntr}

\makeatother

\usepackage{babel}
\providecommand{\corollaryname}{Corollary}
\providecommand{\examplename}{Example}
\providecommand{\lemmaname}{Lemma}
\providecommand{\propositionname}{Proposition}
\providecommand{\remarkname}{Remark}
\providecommand{\theoremname}{Theorem}

\begin{document}

\title{Superconvergence and regularity of densities in free probability}

\author{Hari Bercovici, Jiun-Chau Wang, and Ping Zhong}

\date{March 15, 2021}
\begin{abstract}
The phenomenon of superconvergence, first observed in the central
limit theorem of free probability, was subsequently extended to arbitrary
limit laws for free additive convolution. We show that the same phenomenon
occurs for the multiplicative versions of free convolution on the
positive line and on the unit circle. We also show that a certain
H\"older regularity, first demonstrated by Biane for the density
of a free additive convolution with a semicircular law, extends to
free (additive and multiplicative) convolutions with arbitrary freely
infinitely divisible distributions.
\end{abstract}



\address{Hari Bercovici: Department of Mathematics, Indiana University, Bloomington,
IN 47405, United States }

\address{Jiun-Chau Wang: Department of Mathematics and Statistics, University
of Saskatchewan, Saskatoon, S7N 5E6, Canada}

\address{Ping Zhong: Department of Mathematics and Statistics, University
of Wyoming, Laramie, WY 82071-3036, United States}

\maketitle

\section{Introduction}

Suppose that $\{X_{n}\}_{n\in\mathbb{N}}$ is a free, identically
distributed sequence of bounded random variables with zero mean and
unit variance. It is known from \cite{vo-sym} that the distributions
$\mu_{n}$ of the central limit averages
\[
\frac{X_{1}+\cdots+X_{n}}{\sqrt{n}}
\]
converge weakly to a standard semicircular distribution. Unlike the
classical central limit theorem, it was shown in \cite{BV-super-c}
that the distribution $\mu_{n}$ is absolutely continuous relative
to Lebesgue measure on $\mathbb{R}$ for sufficiently large $n$,
and that the densities $d\mu_{n}/dt$ converge uniformly to $\sqrt{4-t^{2}}\text{\ensuremath{\chi_{[-2,2]}}}/2\pi$.
This unexpected convergence of densities (along with the fact that
the support $[a_{n},b_{n}]$ of $\mu_{n}$ converges to $[-2,2]$
and the density is analytic on $(a_{n},b_{n})$) was called \emph{superconvergence}.
The uniform convergence of densities was later proved to hold even
when the variables $X_{n}$ are not bounded \cite{JC-local}. The
phenomenon of superconvergence was extended to other limit laws and
applied to limit theorems for eigenvalue densities of random matrices
(see, for instance, \cite{Ka-super,Bao-E-S}). Eventually, the present
authors proved in \cite{BWZ-super+} that uniform convergence of densities
holds in the general context of limit laws for triangular arrays with
free, identically distributed rows. That is, suppose that $k_{1}<k_{2}<\cdots$
is a sequence of positive integers, and for each $n$ the variables
$\{X_{n,j}:j=1,\dots,k_{n}\}$ are free and identically distributed.
Suppose also that the distribution $\mu_{n}$ of 
\[
X_{n,1}+\cdots+X_{n,k_{n}}
\]
converges weakly to some nondegenerate distribution $\mu$. The measure
$\mu$ is $\boxplus$-infinitely divisible \cite{BP-hincin} and it
is absolutely continuous everywhere, except on a set $D_{\mu}$ that is either
empty or a singleton \cite[Proposition 5.1]{BWZ-super+}. Let $V\supset D_{\mu}$
be an arbitrary open set in $\mathbb{R};$ $V$ can be taken to be
empty if $D_{\mu}=\varnothing$. Then the result of \cite{BWZ-super+}
states that $\mu_{n}$ is absolutely continuous on $\mathbb{R}\backslash V$
and the density of $\mu_{n}$ converges uniformly to the density of
$\mu$ as $n\to\infty$. 

Of course, the results mentioned above can be formulated just as easily
in terms of free \emph{additive }convolution of measures. One purpose
of the present note is to prove completely analogous results for free
\emph{multiplicative} convolution of probability measures on $\mathbb{R_{+}}=[0,+\infty)$
and on the unit circle $\mathbb{T}=\{e^{it}:t\in\mathbb{R}\}$. Our results here supersede those in \cite{AWZ}, since the uniform convergence of densities in \cite{AWZ} was only proved for compact intervals on which the limiting density is nonzero. The
multiplicative results are not simply consequences of the additive
ones. In fact, each of the three convolutions has its own analytic
apparatus, and in each case an important fact is that the respective
Voiculescu transform of an infinitely divisible measure has an analytic
extension to a certain domain $D$ (that depends on the type of convolution).
In each case, the proof is done first for convolutions of infinitely
divisible measures. The general case is then obtained via an approximation
of infinitesimal measures by infinitely divisible ones, somewhat analogous
to the \emph{associated laws} used in the classical treatment of limit
laws for sums of independent random variables \cite{GK}. These infinitely
divisible laws are obtained from the subordination properties that
hold for free convolutions.

The methods we develop for superconvergence are useful in other contexts
as well. We illustrate this by extending results of Biane \cite{Bi-cusp}
concerning the density of a free convolution of the form $\mu\boxplus\gamma$,
where $\gamma$ is a semicircular distribution. Such a convolution
is always absolutely continuous, its density $h$ is continuous and,
in fact, locally analytic wherever it is positive. If $h(t)=0$ for
some $t$ and $h(x)\ne0$ in some interval with an endpoint at $t$,
it is shown in \cite{Bi-cusp} that $h(x)=O(|x-t|^{1/3})$ for $x$
close to $t$ in that interval. We show that this result holds if
$\gamma$ is replaced by an arbitrary nondegenerate $\boxplus$-infinitely
divisible distribution. Of course, in this general context, it may
happen that $\mu\boxplus\gamma$ has a finite number of atoms and
points at which the density is unbounded. The result holds for all
other points where the density vanishes. Analogous results are also
proved for the two multiplicative free convolutions. 

The remainder of this paper is organized as follows. Sections \ref{sec:Free-multiplicative-convolution}--\ref{sec:cusps-in R_+}
deal with free multiplicative convolution on $\mathbb{R}_{+}$. A section
presents preliminaries about this operation, including a new observation
analogous to the Schwarz lemma, the next section demonstrates superconvergence,
and the last section deals with the possible cusps of the free convolution
with an infinitely divisible law. Sections \ref{sec:Free-mutiplicative-convolution on T}--\ref{sec:Cusp-behavior-in T}
follow the same program for multiplicative free convolution on the
unit circle $\mathbb{T}.$ Finally, Sections \ref{sec:Free-additive-convolution}
and \ref{sec:Cusp-behavior-in R} deal with additive free convolution;
there is no additive analog of Sections \ref{sec:Superconvergence-in pos line}
and \ref{sec:Superconvergence-in T} in the additive case because
the corresponding result was already proved in \cite{BWZ-super+}.
(The reader may however note that the arguments of \cite{BWZ-super+}
can be simplified using the present methods.)  Appendix A provides applications of the cusp results
to measures in a free convolution semigroup.  Finally, Appendix B provides
examples that show that the cusp estimates are often sharp.

\section{Free multiplicative convolution on $\mathbb{R}_{+}$\label{sec:Free-multiplicative-convolution}}

We denote by $\mathcal{P}_{\mathbb{R}_{+}}$ the collection of all
probability measures on $\mathbb{R}_{+}.$ The free multiplicative
convolution $\boxtimes$ is a binary operation on $\mathcal{P}_{\mathbb{R}_{+}}$.
The mechanics of its calculation involves analytic functions defined
on the domains $\mathbb{C}\backslash\mathbb{R}_{+}$, 
\[
\mathbb{H}=\{x+iy:x,y\in\mathbb{R},y>0\},
\]
and $-\mathbb{H}$. The first of these is the \emph{moment generating
function} $\psi_{\mu}$ of a measure $\mu\in\mathcal{P}_{\mathbb{R}_{+}}$
defined by
\[
\psi_{\mu}(z)=\int_{\mathbb{R}_{+}}\frac{tz}{1-tz}\,d\mu(t),\quad z\in\mathbb{C}\backslash\mathbb{R}_{+}.
\]
This function satisfies $\psi_{\mu}(\mathbb{H})\subset\mathbb{H}$
and $\psi_{\mu}((-\infty,0))\subset(-1,0)$ unless $\mu$ is the unit
point mass at $0$, denoted $\delta_{0}$, for which $\psi_{\delta_{0}}=0$.
A closely related function is the $\eta$-\emph{transform} of $\mu$
given by
\[
\eta_{\mu}(z)=\frac{\psi_{\mu}(z)}{1+\psi_{\mu}(z)},\quad z\in\mathbb{C}\backslash\mathbb{R}_{+}.
\]
 We have $\eta_{\mu}(\mathbb{H})\subset\mathbb{H}$ and $\eta_{\mu}((-\infty,0))\subset(-\infty,0)$
when $\mu\ne\delta_{0}$. These transforms are related to the \emph{Cauchy
transform} defined by
\[
G_{\mu}(z)=\int_{\mathbb{R}_{+}}\frac{d\mu(t)}{z-t},\quad z\in\mathbb{C}\backslash\mathbb{R}_{+},
\]
by the identity
\begin{equation}
\frac{1}{z}G_{\mu}\left(\frac{1}{z}\right)=\frac{1}{1-\eta_{\mu}(z)},\quad z\in\mathbb{C}\backslash\mathbb{R}_{+}.\label{eq:G versus eta}
\end{equation}
The Stieltjes inversion formula shows that any of these functions
can be used to recover the measure $\mu$. More precisely, the measures
\[
-\frac{1}{\pi}(\Im G_{\mu}(x+iy))dx,\quad y>0,
\]
converge weakly to $\mu$ as $y\downarrow0$. The boundary values
\[
G_{\mu}(x)=\lim_{y\downarrow0}G_{\mu}(x+iy),\quad x\in\mathbb{R}_{+},
\]
exists almost everywhere (with respect to Lebesgue measure) on $\mathbb{R}_{+}$,
and the density $d\mu/dt$ of $\mu$ is equal almost everywhere to
$(-1/\pi)\Im G_{\mu}$ (cf. \cite{SteinWeiss}). 

In terms of the $\eta$-transform, the relation (\ref{eq:G versus eta})
shows that
\begin{equation}
\frac{1}{x}\frac{d\mu}{dt}\left(\frac{1}{x}\right)=\frac{1}{\pi}\Im\frac{1}{1-\eta_{\mu}(x)}\label{eq:density vs eta}
\end{equation}
almost everywhere on $\mathbb{R}_{+}$, where $\eta_{\mu}(x)$ is
defined almost everywhere as
\[
\eta_{\mu}(x)=\lim_{y\downarrow0}\eta_{\mu}(x+iy).
\]

The collection of functions $\{\eta_{\mu}:\mu\in\mathcal{P}_{\mathbb{R}_{+}}\backslash\{\delta_{0}\}\}$
is described as follows.
\begin{lem}
\label{lem:description of all etas on the line}\cite{B-B-IMRN} Let
$f:\mathbb{C}\backslash\mathbb{R}_{+}\to\mathbb{C}$ be an analytic
function. Then there exists $\mu\in\mathcal{P}_{\mathbb{R}_{+}}$
such that $f=\eta_{\mu}$ if and only if the following conditions
are satisfied:
\begin{enumerate}
\item $f(\overline{z})=\overline{f(z)}$ for every $z\in\mathbb{C}\setminus\mathbb{R}_{+}$,
\item $\lim_{x\uparrow0}f(x)=0$, and
\item $\arg f(z)\ge\arg z$, $z\in\mathbb{H},$ where the arguments are
in $(0,\pi)$.
\end{enumerate}
Equality occurs in \emph{(3)} for some $z$ precisely when $\mu=\delta_{a}$
for some $a>0$, in which case $f(z)=\eta_{\mu}(z)=az$.
\end{lem}

In fact, condition (3) above can be replaced by $f(\mathbb{H})\subset\mathbb{H}$, as can be seen from Lemma
\ref{lem:Schwarz analog}, which we may view as an analog of the Schwarz
lemma for analytic functions in the unit disk. (This version of Lemma \ref{lem:description of all etas on the line} is useful in Lemma \ref{lem:trade a free convolution for another}.)
\begin{lem}
\label{lem:Structure of functions on omega}Suppose that $F:\mathbb{C}\backslash\mathbb{R}_{+}\to\mathbb{C}$
is analytic, $F(\mathbb{H})\subset\mathbb{H}$, $F((-\infty,0))\subset(-\infty,0)$,
and 
\[
F(\overline{z})=\overline{F(z)},\qquad z\in\mathbb{C}\backslash\mathbb{R}_{+}.
\]
Then there exist constants $\alpha,\beta\in[0,+\infty)$ and a finite
Borel measure $\rho$ on $(0,+\infty)$ such that $\int_{(0,+\infty)}d\rho(t)/t<+\infty$
and
\[
F(z)=-\alpha+\beta z+\int_{(0,+\infty)}\frac{z(1+t^{2})}{t(t-z)}\,d\rho(t),\qquad z\in\mathbb{C}\backslash\mathbb{R}_{+}.
\]
\end{lem}

\begin{proof}
Since $F(\mathbb{H})\subset\mathbb{H}$, $F$ has a Nevanlinna representation
of the form
\[
F(z)=\alpha_{0}+\beta z+\int_{\mathbb{R}}\frac{1+zt}{t-z}\,d\rho(t),\qquad z\in\mathbb{H},
\]
with $\alpha_{0}\in\mathbb{R},$ $\beta\in[0,+\infty)$, and a finite
Borel measure $\rho$ on $\mathbb{R}$ (cf. \cite{Akhiezer}). Because
$F$ is analytic and real-valued on $(-\infty,0)$, the measure $\rho$
is supported on $[0,+\infty).$ The formula
\[
F'(z)=\beta+\int_{[0,+\infty)}\frac{1+t^{2}}{(t-z)^{2}}\,d\rho(t)
\]
shows that $F$ is increasing on $(-\infty,0)$. Now, $F((-\infty,0))\subset(-\infty,0)$,
so $\lim_{z\uparrow0}F(z)\le0.$ The monotone convergence theorem
yields now
\[
\alpha_{0}+\int_{[0,+\infty)}\frac{1}{t}\,d\rho(t)=\lim_{z\uparrow0}F(z)\le0.
\]
In particular, $\rho(\{0\})=0$ and $\rho$ satisfies the condition
in the statement. We set
\[
\alpha=-\alpha_{0}-\int_{(0,+\infty)}\frac{1}{t}\,d\rho(t),
\]
and obtain the formula
\[
F(z)=-\alpha+\beta z+\int_{(0,+\infty)}\left[\frac{1+zt}{t-z}-\frac{1}{t}\right]\,d\rho(t),
\]
valid in the entire region $\mathbb{C}\backslash\mathbb{R}_{+}$ by
reflection. This is easily seen to be precisely the formula in the
statement.
\end{proof}
Notation: $\Omega_{\alpha}=\{z\in\mathbb{C}\backslash\mathbb{R}_{+}:|\arg z|>\alpha\}$.
Here $\alpha\in(0,\pi)$ and the argument takes values in $(-\pi,\pi)$.

\begin{lem}
\label{lem:Schwarz analog}Under the conditions of \textup{Lemma}
\emph{\ref{lem:Structure of functions on omega}}, we have $F(\Omega_{\alpha})\subset\Omega_{\alpha}$
for every $\alpha\in(0,\pi)$.
\end{lem}

\begin{proof}
It suffices to prove that $F(\Omega_{\alpha}\cap\mathbb{H})\subset\Omega_{\alpha}\cap\mathbb{H}$.
Since $\Omega_{\alpha}\cap\mathbb{H}$ is a convex cone, the representation
formula in Lemma \ref{lem:Structure of functions on omega} reduces
the proof to the following three cases:
\begin{enumerate}
\item $F(z)=-1$,
\item $F(z)=z$,
\item $F(z)=z/(t-z)$ for some $t>0$.
\end{enumerate}
The result is trivial in the first two cases. In the third case one
observes that $F$ maps $\Omega_{\alpha}\cap\mathbb{H}$ conformally
onto a region $D_{\alpha}$ bounded by the 
interval $(-1,0)$ and by
a circular arc $C$ joining $-1$ and $0$. Moreover, since $F'(0)>0$,
the tangent to $C$ at $0$ is the line $\{\arg z=\alpha\}$. It follows
immediately that $D_{\alpha}\cap\mathbb{H}\subset\Omega_{\alpha}\cap\mathbb{H}$.
\end{proof}
Mapping $\mathbb{C}\backslash\mathbb{R}_{+}$ conformally to a strip
by the logarithm, we obtain another version of the Schwarz lemma as follows.
We set $\mathcal{S}_{t}=\{z\in\mathbb{C}:|\Im z|<t\}$ for $t>0$. 
\begin{prop}
Let $F:\mathcal{S}_{1}\to\mathcal{S}_{1}$ be an analytic function
such $F(\mathcal{S}_{1}\cap\mathbb{H})\subset\mathcal{S}_{1}\cap\mathbb{H}$
and
\[
F(\overline{z})=\overline{F(z)},\qquad z\in\mathcal{S}_{1}.
\]
Then $F(\mathcal{S}_{t})\subset\mathcal{S}_{t}$ for every $t\in(0,1)$.
\end{prop}

Given a measure $\mu\ne\delta_{0}$ in $\mathcal{P}_{\mathbb{R}_{+}}$,
the function $\eta_{\mu}$ is conformal in an open set $U$ containing
some interval $(\beta,0)$ with $\beta<0$, and the restriction $\eta_{\mu}|U$
has an inverse $\eta_{\mu}^{\langle-1\rangle}$ defined in an open
set containing an interval of the form $(\alpha,0)$ with $\alpha<0$.
The free multiplicative convolution $\mu_{1}\boxtimes\mu_{2}$ of
two measures $\mu_{1},\mu_{2}\in\mathcal{P}_{\mathbb{R}_{+}}\backslash\{\delta_{0}\}$
is the unique measure $\mu\in\mathcal{P}_{\mathbb{R}_{+}}\backslash\{\delta_{0}\}$
that satisfies the identity
\begin{equation}
z\eta_{\mu}^{\langle-1\rangle}(z)=\eta_{\mu_{1}}^{\langle-1\rangle}(z)\eta_{\mu_{2}}^{\langle-1\rangle}(z)\label{eq:defining boxtimes}
\end{equation}
for $z$ in some open set containing an interval $(\alpha,0)$ with
$\alpha<0$ (see \cite{BV-unbounded}). (We also have $\delta_{0}\boxtimes\mu=\delta_{0}$
for every $\mu\in\mathcal{P}_{\mathbb{R}_{+}}.$) Based on the characterization
of $\eta$-transform, another approach to free convolution is given
by the following reformulation of the subordination results in \cite{Bi-free inc}. 
\begin{thm}
\label{thm:subordination on the line (mult)}For every $\mu_{1},\mu_{2}\in\mathcal{P}_{\mathbb{R}_{+}}\backslash\{\delta_{0}\},$
there exist unique $\rho_{1},\rho_{2}\in\mathcal{P}_{\mathbb{R}_{+}}\backslash\{\delta_{0}\}$
such that
\[
\eta_{\mu_{1}}(\eta_{\rho_{1}}(z))=\eta_{\mu_{2}}(\eta_{\rho_{2}}(z))=\frac{\eta_{\rho_{1}}(z)\eta_{\rho_{2}}(z)}{z},\quad z\in\mathbb{C}\backslash\mathbb{R}_{+}.
\]
Moreover, we have $\eta_{\mu_{1}\boxtimes\mu_{2}}=\eta_{\mu_{1}}\circ\eta_{\rho_{1}}$.
If $\mu_{1}$ and $\mu_{2}$ are nondegenerate \emph{(}that is, $\mu_{1}$ and $\mu_{2}$ are not point masses\emph{)}, then so are $\rho_{1}$
and $\rho_{2}$. 
\end{thm}

We recall that a measure $\mu\in\mathcal{P}_{\mathbb{R}_{+}}$ is
said to be $\boxtimes$-infinitely divisible if there exist measures
$\{\mu_{n}\}_{n\in\mathbb{N}}\subset\mathcal{P}_{\mathbb{R}_{+}}$
satisfying the identities
\[
\underbrace{\mu_{n}\boxtimes\cdots\boxtimes\mu_{n}}_{n\text{ times}}=\mu,\quad n\in\mathbb{N}.
\]
 Obviously, $\delta_{0}$ is $\boxtimes$-infinitely divisible; one
can take $\mu_{n}=\delta_{0}$. It was shown in \cite{vo-mul,BV-unbounded}
that a measure $\mu\in\mathcal{P}_{\mathbb{R}_{+}}\backslash\{\delta_{0}\}$
is $\boxtimes$-infinitely divisible precisely when the inverse $\eta_{\mu}^{\langle-1\rangle}$
continues analytically to $\mathbb{C}\backslash\mathbb{R}_{+}$ and
this analytic continuation has the special form
\begin{equation}
\Phi(z)=\gamma z\exp\left[\int_{[0,+\infty]}\frac{1+tz}{z-t}\,d\sigma(t)\right],\quad z\in\mathbb{C}\backslash\mathbb{R}_{+},\label{eq:extension of eta inverse (line)}
\end{equation}
 for some $\gamma>0$ and some finite Borel measure $\sigma$ on the
one point compactification of $\mathbb{R}_{+}$. The fraction in the
above formula must be interpreted as $-z$ when $t=+\infty$. This
is, of course, an analog of the classical L\'evy-Hin\v cin formula.
The pair $(\gamma,\sigma)$ is uniquely determined by $\mu$, and
every such pair corresponds with a unique $\boxtimes$-infinitely
divisible measure, sometimes denoted $\nu_{\boxtimes}^{\gamma,\sigma}$.
Another description of the class of functions defined by (\ref{eq:extension of eta inverse (line)})
is as follows: 
\[
\Phi(z)=z\exp(u(z)),
\]
where $u:\mathbb{C}\backslash\mathbb{R}_{+}\rightarrow\mathbb{C}$
is an analytic function such that $u(\mathbb{H})\subset-\mathbb{H}$
and $u(\overline{z})=\overline{u(z)}$ for all $z\in\mathbb{C}\setminus\mathbb{R}_{+}$.
This equivalent description is used in Lemma \ref{lem:omega is inf-div}.

Suppose now that $\mu\in\mathcal{P}_{\mathbb{R}_{+}}$ is a nondegenerate
$\boxtimes$-infinitely divisible measure and that $\eta_{\mu}^{\langle-1\rangle}$
has the analytic continuation given in (\ref{eq:extension of eta inverse (line)}).
The equation $\Phi(\eta_{\mu}(z))=z$ holds in some open set and therefore
it holds on the entire $\mathbb{C}\backslash\mathbb{R}_{+}$ by analytic
continuation. In particular, $\eta_{\mu}$ maps $\mathbb{C}\backslash\mathbb{R}_{+}$
conformally onto a domain $\Omega_{\mu}\subset\mathbb{C}\backslash\mathbb{R}_{+}$
that is symmetric relative to the real line. The domain $\Omega_{\mu}$
is easily identified as the connected component of the set $\{z\in\mathbb{C}\backslash\mathbb{R}_{+}:\Phi(z)\in\mathbb{C}\backslash\mathbb{R}_{+}\}$
containing $(-\infty,0)$. This set and its boundary were thoroughly
investigated in \cite{huang-zhong,huang-wang}, and the results are important
in the sequel. Because of the symmetry of $\Omega_{\mu}$, we consider
only the upper half of $\Omega_{\mu}$, namely, $\Omega_{\mu}\cap\mathbb{H}$.
A simple calculation shows that
\begin{equation}
\Phi(re^{i\theta})=\gamma\exp[u(re^{i\theta})+iv(re^{i\theta})],\label{eq:u and v}
\end{equation}
where the real and imaginary parts $u$ and $v$ are given by
\begin{equation}
u(re^{i\theta})=\log r+\int_{[0,+\infty]}\frac{(1-t^{2})r\cos\theta+t(r^{2}-1)}{|re^{i\theta}-t|^{2}}\,d\sigma(t),\label{eq:u(z) in polar terms}
\end{equation}
and
\[
v(re^{i\theta})=\theta\left[1-\frac{r\sin\theta}{\theta}\int_{[0,+\infty]}\frac{1+t^{2}}{|re^{i\theta}-t|^{2}}\,d\sigma(t)\right],
\]
for $r>0$ and $\theta\in(0,\pi)$. As noted in \cite{huang-zhong,huang-wang},
a remarkable situation occurs: for fixed $r>0$, the function
\begin{equation}
I_{r}(\theta)=\frac{r\sin\theta}{\theta}\int_{[0,+\infty]}\frac{1+t^{2}}{|re^{i\theta}-t|^{2}}\,d\sigma(t),\quad\theta\in(0,\pi],\label{eq:I_r except 0}
\end{equation}
is continuous, strictly decreasing, and $I_{r}(\pi)=0$. Thus, the
set $\{\theta\in(0,\pi):I_{r}(\theta)<1\}$ is an interval, say
\begin{equation}
\{\theta\in(0,\pi):I_{r}(\theta)<1\}=(f(r),\pi).\label{eq:def of f half-line}
\end{equation}
 The value $f(r)$ is $0$ precisely when the limit
\begin{equation}
I_{r}(0)=\lim_{\theta\downarrow0}I_{r}(\theta)=r\int_{[0,+\infty]}\frac{1+t^{2}}{(r-t)^{2}}\,d\sigma(t)\label{eq:I_r at zero}
\end{equation}
 is at most $1$. Otherwise, we have $I_{r}(f(r))=1$. The following
statement summarizes results from \cite{huang-zhong,huang-wang}.
\begin{thm}
\label{thm:inversion for half-line} Let $\mu\in\mathcal{P}_{\mathbb{R}_{+}}$
be a nondegenerate $\boxtimes$-infinitely divisible measure, let
$\Phi$ defined by \emph{(\ref{eq:extension of eta inverse (line)})}
be the analytic continuation of $\eta_{\mu}^{\langle-1\rangle}$,
let $I_{r}:[0,\pi]\to(0,+\infty]$ be defined by \emph{(\ref{eq:I_r except 0})
}and \emph{(\ref{eq:I_r at zero})}, and let $f:(0,+\infty)\to[0,\pi)$
be defined by \emph{(\ref{eq:def of f half-line})}. Then\emph{:}
\begin{enumerate}
\item $\eta_{\mu}$ maps $\mathbb{H}$ conformally onto 
\[
\Omega_{\mu}\cap\mathbb{H}=\{re^{i\theta}:r>0,\theta\in(f(r),\pi)\}.
\]
\item The function $f$ is continuous on $(0,+\infty)$ and continuously differentiable on the open set $\{r:f(r)>0\}$. 
\item The topological boundary of the set $\Omega_{\mu}\cap\mathbb{H}$
is $(-\infty,0]\cup\{re^{if(r)}:r>0\}$.
\item $\eta_{\mu}$ extends continuously to the closure $\overline{\mathbb{H}}$,
$\Phi$ extends continuously to the closure $\overline{\Omega_{\mu}\cap\mathbb{H}}$,
and these extensions are homeomorphisms, inverse to each other. In
particular, the function $h:(0,+\infty)\to(0,+\infty)$ defined by
\[
h(r)=\Phi(re^{if(r)}),\quad r>0,
\]
is an increasing homeomorphism from $(0,+\infty)$ onto $(0,+\infty)$
and the image $\eta_{\mu}((0,+\infty))$ is parametrized implicitly
as
\begin{equation}
\eta_{\mu}(h(r))=re^{if(r)},\quad r>0.\label{eq:param of eta_mu(pos line)}
\end{equation}
\end{enumerate}
\end{thm}

It is known that $\mu(\{0\})=0$ for every $\boxtimes$-infinitely
divisible measure $\mu\in\mathcal{P}_{\mathbb{R}_{+}}\backslash\{\delta_{0}\}$.
For such a measure $\mu,$ we can define a measure $\mu_{*}\in\mathcal{P}_{\mathbb{R}_{+}}\backslash\{\delta_{0}\}$
such that $d\mu_{*}(t)=d\mu(1/t)$. An easy calculation yields the
identities
\[
\psi_{\mu_{*}}(z)=-1-\psi_{\mu}(1/z),\quad\eta_{\mu_{*}}(z)=\frac{1}{\eta_{\mu}(1/z)},\quad z\in\mathbb{C}\backslash\mathbb{R}_{+},
\]
and therefore
\[
\eta_{\mu_{*}}^{\langle-1\rangle}(z)=\frac{1}{\eta_{\mu}^{\langle-1\rangle}(1/z)}
\]
for $z$ in some open set containing $(-\infty,0)$. It follows that
$\eta_{\mu_{*}}^{\langle-1\rangle}$ has an analytic continuation
to $\mathbb{C}\backslash\mathbb{R}_{+}$. In fact, if $\Phi$ is the
continuation of $\eta_{\mu}^{\langle-1\rangle}$ given by (\ref{eq:extension of eta inverse (line)}),
then the function
\[
\Phi_{*}(z)=\frac{1}{\Phi(1/z)}=\frac{1}{\gamma}z\exp\left[\int_{[0,+\infty]}\frac{1+tz}{z-t}\,d\sigma_{*}(t)\right],\quad z\in\mathbb{C}\backslash\mathbb{R}_{+},
\]
extends $\eta_{\mu_{*}}^{\langle-1\rangle}$, where $d\sigma_{*}(t)=d\sigma(1/t)$
with the convention that $1/0=+\infty$ and $1/+\infty=0$. Thus,
$\mu_{*}$ is also infinitely divisible, and the boundary of $\Omega_{\mu_{*}}\cap\mathbb{H}$
is described as above using a continuous function $f_{*}:(0,+\infty)\to[0,\pi)$.
This function and the associated homeomorphism $h_{*}(r)=\Phi_{*}(re^{if_{*}(r)})$
are easily seen to satisfy the identities
\[
f_{*}(r)=f(1/r),\quad h_{*}(r)=\frac{1}{h(1/r)},\quad r\in(0,+\infty).
\]

The following result gives estimates for the growth of $h$ at $0$
and $+\infty$.
\begin{prop}
\label{prop:endpoint estimates for h} Let $\mu,\Phi$, and $h$ be
as in \textup{Theorem}\emph{ \ref{thm:inversion for half-line}}.
Then
\[
h(r)\le\gamma r\exp(\sigma([0,+\infty])+2),\quad r\in(0,1/4),
\]
and
\[
h(r)\ge\gamma r\exp(-\sigma([0,+\infty])-2),\quad r\in(4,+\infty).
\]
In particular, $\lim_{r\downarrow0}h(r)=0$. 
\end{prop}

\begin{proof}
Suppose for the moment that the first inequality was proved. Applying
the result to the measure $\mu_{*}$, we see that
\begin{align*}
\frac{1}{h(1/r)}=h_{*}(r) & \le\frac{1}{\gamma}r\exp(\sigma_{*}([0,+\infty])+2)\\
 & =\frac{1}{\gamma}r\exp(\sigma([0,+\infty])+2)
\end{align*}
for $r<1/4$. The second inequality follows after replacing $r$ by
$1/r$. 

Fix now $r\in(0,1/4)$, and use relations (\ref{eq:u and v}), (\ref{eq:u(z) in polar terms}),
and the fact that $r^{2}-1\le0$ to deduce the inequality
\[
|\Phi(re^{i\theta})|\le\gamma r\exp\left[r\cos\theta\int_{[0,+\infty]}\frac{1-t^{2}}{|re^{i\theta}-t|^{2}}\,d\sigma(t)\right],\quad\theta\in(0,\pi).
\]
We distinguish two cases, according to whether $f(r)<\pi/2$ or $f(r)\ge\pi/2$.
In the first case, we have $I_{r}(f(r))\le1$ and hence
\begin{align*}
\left|r\cos\theta\int_{[0,+\infty]}\frac{1-t^{2}}{|re^{i\theta}-t|^{2}}\,d\sigma(t)\right| & \le r\int_{[0,+\infty]}\frac{1+t^{2}}{|re^{i\theta}-t|^{2}}\,d\sigma(t)\\
 & =\frac{\theta}{\sin\theta}I_{r}(\theta)\le\frac{\pi}{2}I_{r}(f(r))<2
\end{align*} for $\theta\in(f(r),\pi/2)$.
It follows that $|h(r)|=\lim_{\theta\downarrow f(r)}|\Phi(re^{i\theta})|\le\gamma re^{2}$, verifying the first inequality in this case. In the second case,
we observe, for $\theta=f(r)$, that 
\[
r\cos\theta\int_{[0,1]}\frac{1-t^{2}}{|re^{i\theta}-t|^{2}}\,d\sigma(t)\le0,
\]
and thus
\begin{align*}
r\cos\theta\int_{[0,+\infty]}\frac{1-t^{2}}{|re^{i\theta}-t|^{2}}\,d\sigma(t) & \le r\int_{[1,+\infty]}\frac{t^{2}-1}{|re^{i\theta}-t|^{2}}\,d\sigma(t)\\
 & \le r\int_{[1,+\infty]}\frac{t^{2}-1}{|t-\frac{1}{2}|^{2}}\,d\sigma(t)\\
 & \le r\int_{[1,+\infty]}2\,d\sigma(t)\le\sigma([0,+\infty]).
\end{align*}
This verifies the inequality in the second case and concludes the
proof.
\end{proof}
The continuity of the function $\Phi$ on some parts of $\Omega_{\mu}$
can be established as follows. 
\begin{lem}
\label{lem:continuity of u (half line)}Let $\mu\in\mathcal{P}_{\mathbb{R}_{+}}$
be a nondegenerate $\boxtimes$-infinitely divisible measure, and
let $\Phi$ defined by \emph{(\ref{eq:extension of eta inverse (line)})}
be the analytic continuation of $\eta_{\mu}^{\langle-1\rangle}$.
Set
\[
u(z)=\int_{[0,+\infty]}\frac{1+tz}{z-t}\,d\sigma(t),\quad z\in\mathbb{C}\backslash\mathbb{R}_{+}.
\]
Suppose that $z_{j}=r_{j}e^{i\theta_{j}}\in\Omega_{\mu}\cap\mathbb{H}$
and that $\theta_{j}\leq\pi/2$ for $j=1,2$. Then
\[
|u(z_{1})-u(z_{2})|\le\frac{\pi}{2}\frac{|z_{1}-z_{2}|}{\sqrt{|z_{1}z_{2}|}}.
\]
\end{lem}

\begin{proof}
We have $\theta_{j}\in(f(r_{j}),\pi/2]$, $j=1,2$. In particular,
\[
I_{r_{j}}(\theta_{j})\le I_{r_{j}}(f(r_{j}))\le1,\quad j=1,2.
\]
Then
\begin{align*}
|u(z_{1})-u(z_{2})| & =\left|(z_{1}-z_{2})\int_{[0,+\infty]}\frac{(1+t^{2})^{1/2}}{z_{1}-t}\frac{(1+t^{2})^{1/2}}{z_{2}-t}\,d\sigma(t)\right|\\
 & \le|z_{1}-z_{2}|\left[\int_{[0,+\infty]}\frac{1+t^{2}}{|z_{1}-t|^{2}}\,d\sigma(t)\right]^{1/2}\left[\int_{[0,+\infty]}\frac{1+t^{2}}{|z_{2}-t|^{2}}\,d\sigma(t)\right]^{1/2}\\
 & =|z_{1}-z_{2}|\left[\frac{\theta_{1}}{r_{1}\sin\theta_{1}}I_{r_{1}}(\theta_{1})\right]^{1/2}\left[\frac{\theta_{2}}{r_{2}\sin\theta_{2}}I_{r_{2}}(\theta_{2})\right]^{1/2}\le\frac{\pi}{2}\frac{|z_{1}-z_{2}|}{\sqrt{r_{1}r_{2}}},
\end{align*}
where we used the Schwarz inequality.
\end{proof}
We conclude this section with a few known facts about convolution
powers. Given a measure $\nu\in\mathcal{P}_{\mathbb{R}_{+}}\backslash\{\delta_{0}\}$
and $k\in\mathbb{N},$ we use the notation
\[
\nu^{\boxtimes k}=\underbrace{\nu\boxtimes\cdots\boxtimes\nu}_{k\text{ times}}
\]
for the free multiplicative convolution of $k$ copies of $\nu$.
By Theorem \ref{thm:subordination on the line (mult)}, there exists
a measure $\mu\in\mathcal{P}_{\mathbb{R}_{+}}$ such that $\eta_{\nu^{\boxtimes k}}=\eta_{\nu}\circ\eta_{\mu}$.
It is shown in \cite{B-B-IMRN} that 
\[
\Phi(\eta_{\mu}(z))=z,\quad z\in\mathbb{C}\backslash\mathbb{R_{+}},
\]
where 
\[
\Phi(z)=\frac{z^{k}}{\eta_{\nu}(z)^{k-1}},\quad z\in\mathbb{C}\backslash\mathbb{R_{+}},
\]
is easily seen to have the form (\ref{eq:extension of eta inverse (line)}).
As seen earlier, this means that $\mu$ is in fact $\boxtimes$-infinitely
divisible, and therefore $\eta_{\mu}$ has a continuous (and injective)
extension to $(0,+\infty)$. The relation between $\eta_{\mu}$ and
$\nu^{\boxtimes k}$ can also be written as
\begin{equation}
\eta_{\mu}(z)^{k}=z\eta_{\nu^{\boxtimes k}}(z)^{k-1},\quad z\in\mathbb{C}\backslash\mathbb{R}_{+}.\label{eq:eta of subordination, halflin}
\end{equation}
As observed in \cite{B-B-IMRN, Ch-Gotze}, this equality implies that $\eta_{\nu^{\boxtimes k}}$
also has a continuous extension to $(0,+\infty)$ and (\ref{eq:eta of subordination, halflin})
remains true for real values of $z$. We use below this identity under
the equivalent form
\[
x\eta_{\mu}(1/x)^{k}=\eta_{\nu^{\boxtimes k}}(1/x)^{k-1},\quad x\in(0,+\infty).
\]
This is of interest because it allows us to calculate the density
$d\nu^{\boxtimes k}/dt$ in terms of the density of $\mu$. Indeed,
rewriting the above identity as
\[
\eta_{\mu}(1/x)\left[x\eta_{\mu}(1/x)\right]^{1/(k-1)}=\eta_{\nu^{\boxtimes k}}(1/x),\quad x\in(0,+\infty),
\]
one may be able to argue (as we do in Section \ref{sec:Superconvergence-in pos line})
that $\eta_{\nu^{\boxtimes k}}(1/x)$ is very close to $\eta_{\mu}(1/x)$
if $k$ is large, and then (\ref{eq:density vs eta}) allows us to
conclude that these two measures have close densities.

\section{Superconvergence in $\mathcal{P}_{\mathbb{R}_{+}}$\label{sec:Superconvergence-in pos line}}

We begin by studying the weak convergence of a sequence of nondegenerate
$\boxtimes$-infinitely divisible measures in $\mathcal{P_{\mathbb{R}_{+}}}$.
Thus, suppose that $\text{\ensuremath{\gamma} and }\{\gamma_{n}\}_{n\in\mathbb{N}}$
are positive numbers, $\sigma$ and $\{\sigma_{n}\}_{n\in\mathbb{N}}$
are finite, nonzero Borel measures on $[0,+\infty]$, $\mu$ and $\{\mu_{n}\}_{n\in\mathbb{N}}$
are nondegenerate $\boxtimes$-infinitely divisible measures in $\mathcal{P_{\mathbb{R}_{+}}}$,
and the inverses $\eta_{\mu}^{\langle-1\rangle},\{\eta_{\mu_{n}}^{\langle-1\rangle}\}_{n\in\mathbb{N}}$
have analytic continuations $\Phi,\{\Phi_{n}\}_{n\in\mathbb{N}}$
given by (\ref{eq:extension of eta inverse (line)}) for $\mu$ and
by analogous formulas for $\mu_{n}$ (with $\gamma_{n}$ and $\sigma_{n}$
in place of $\gamma$ and $\sigma$). The sequence $\{\mu_{n}\}_{n\in\mathbb{N}}$
converges weakly to $\mu$ if and only if $\{\sigma_{n}\}_{n\in\mathbb{N}}$
converges weakly to $\sigma$ and $\lim_{n\to\infty}\gamma_{n}=\gamma$.
(This fact is implicit in the proof of Theorem 4.3 in \cite{BPata-mult-laws}.)
When these conditions are satisfied, it is also true that the sequences
$\{\eta_{\mu_{n}}\}_{n\in\mathbb{N}}$ and $\{\Phi_{\mu_{n}}\}_{n\in\mathbb{N}}$
converge to $\eta_{\mu}$ and $\Phi_{\mu}$, respectively, and the
convergence is uniform on compact subsets of $\mathbb{C}\backslash\mathbb{R}_{+}$.

In order to show that superconvergence occurs, we need to understand
the behavior of the functions $f$ and $h$ defined in Section \ref{sec:Free-multiplicative-convolution}
in relation  to $\mu$ and that of the functions $f_{n}$ and $h_{n}$
associated to $\mu_{n}$. By Proposition \ref{prop:endpoint estimates for h}, it is understood that $h_{n}$ and $h$
are extended to $\mathbb{R}_{+}$ so that $h(0)=h_{n}(0)=0$.
\begin{lem}
\label{lem:f_n tends to f etc}With the above notation, suppose that
the sequence $\{\mu_{n}\}_{n\in\mathbb{N}}$ converges weakly to $\mu$.
Then\emph{:}
\begin{enumerate}
\item The sequence $\{f_{n}\}_{n\in\mathbb{N}}$ converges to $f$ uniformly
on compact subsets of $(0,+\infty)$.
\item The sequence $\{h_{n}\}_{n\in\mathbb{N}}$ converges to $h$ uniformly
on compact subsets of $\mathbb{R}_{+}.$
\item The sequence of inverses $\{h_{n}^{\langle-1\rangle}\}_{n\in\mathbb{N}}$
converges to $h^{\langle-1\rangle}$ uniformly on compact subsets
of $(0,+\infty)$.
\item The sequence $\{f_{n}\circ h_{n}^{\langle-1\rangle}\}_{n\in\mathbb{N}}$
converges to $f\circ h^{\langle-1\rangle}$ uniformly on compact subsets
of $(0,+\infty)$.
\item The sequence $\{\eta_{\mu_{n}}\}_{n\in\mathbb{N}}$ converges to
$\eta_{\mu}$ uniformly on compact subsets of $(0,+\infty)$.
\end{enumerate}
\end{lem}

\begin{proof}
Fix $r>0$, let $\varepsilon\in(0,\pi-f(r))$, and let $J=[r-\delta,r+\delta]$
be such that $|f(s)-f(r)|<\varepsilon$ for every $s\in J$. Observe
that the compact set
\[
C=\{se^{i\theta}:\theta\in[f(r)+\varepsilon,\pi],s\in J\}
\]
has the property that $\Phi(C)\subset\mathbb{H}$. Since $\Phi_{n}$
converges to $\Phi$ uniformly on $C$, it follows that $\Phi_{n}(C)\subset\mathbb{H}$
for sufficiently large $n$, and thus $f_{n}(s)<f(r)+\varepsilon<f(s)+2\varepsilon$,
$s\in J$, for such $n$. This proves (1) in case $f(r)=0$. If $f(r)>0$,
there exists a positive angle $\theta_{0}\in(f(r)-\varepsilon,f(r))$
such that $\Phi(re^{i\theta_{0}})\in-\mathbb{H}.$ Shrink the number
$\delta$ such that $\Phi(se^{i\theta_{0}})\in-\mathbb{H}$ for every
$s\in J$. It follows from uniform convergence that $\Phi_{n}(se^{i\theta_{0}})\in-\mathbb{H},\ s\in J$,
for sufficiently large $n$, and thus $f_{n}(s)>\theta_{0}>f(r)-\varepsilon>f(s)-2\varepsilon,\ s\in J$,
thus completing the proof of (1).

For (2) and (3), it suffices to prove pointwise convergence because
pointwise convergence of continuous increasing functions to a continuous limit
 is automatically
locally uniform. Since convergence obviously holds at $0$, fix $r>0$.
Suppose first that $f(r)>0$. In this case, $se^{if(s)}\in\mathbb{H}$
for $s$ in some compact neighborhood of $r$, and hence $\Phi_{n}$
converges uniformly to $\Phi$ in a neighborhood of $re^{if(r)}$.
By (1), $\lim_{n\to\infty}f_{n}(r)=f(r)$, and the local uniform convergence
of $\Phi_{n}$ yields
\[
h(r)=\Phi(re^{if(r)})=\lim_{n\to\infty}\Phi_{n}(re^{if_{n}(r)})=\lim_{n\to\infty}h_{n}(r),
\]
thus proving (1) in this case. Suppose now that $f(r)=0$, and thus
$\lim_{n\to\infty}f_{n}(r)=0.$ Assume, for simplicity, that $f_{n}(r)<1$
for every $n\in\mathbb{N},$ and define functions $\Psi_{n}:(0,\pi/2-1]\to\mathbb{C}$
by setting 
\[
\Psi_{n}(\theta)=\Phi_{n}(re^{i\theta+f_{n}(r)}),\qquad0<\theta\leq\frac{\pi}{2}-1,\;n\in\mathbb{N}.
\]
It follows from Lemma \ref{lem:continuity of u (half line)} that
the functions $\Psi_{n}$ are uniformly equicontinuous. The local
uniform convergence of $\Phi_{n}$ to $\Phi$ shows that $\Psi_{n}$
converges pointwise to $\Phi(re^{i\theta})$. Now, both $\Psi_{n}$
and $\Phi(re^{i\theta})$ extend continuously to $\theta=0$ with
\[
\Psi_{n}(0)=h_{n}(r),\;\Phi(r)=h(r).
\]
The uniform equicontinuity of $\Psi_{n}$ implies that the convergence
also holds (even uniformly) for these continuous extensions, and at
$\theta=0$ this yields the desired equality $\lim_{n\to\infty}h_{n}(r)=h(r)$.

The pointwise convergence of $h_{n}^{\langle-1\rangle}$ to $h^{\langle-1\rangle}$
follows directly from (2). Indeed, suppose that $t_{0}>0$, $s_{0}=h(t_{0})$,
and $0<\varepsilon<t_{0}$. We have $\lim_{n\to\infty}h_{n}(t_{0}-\varepsilon)=h(t_{0}-\varepsilon),$
$\lim_{n\to\infty}h_{n}(t_{0}+\varepsilon)=h(t_{0}+\varepsilon)$,
and the open interval $(h(t_{0}-\varepsilon),h(t_{0}+\varepsilon))$
contains $s_{0}$. It follows that the interval $(h_{n}(t_{0}-\varepsilon),h_{n}(t_{0}+\varepsilon))$
also contains $s_{0}$ for sufficiently large $n$, and thus $h_{n}^{\langle-1\rangle}(s_{0})\in(t_{0}-\varepsilon,t_{0}+\varepsilon)$
for such $n$. Since $\varepsilon$ is arbitrarily small, we have
$\lim_{n\to\infty}h_{n}^{\langle-1\rangle}(s_{0})=t_{0}=h^{\langle-1\rangle}(s_{0})$.

Finally, (4) and (5) follow from (1) and (3) (see \cite[Theorem XII.2.2]{dug}). 
\end{proof}
We are now ready to show that the weak convergence of infinitely divisible
measures implies the convergence of the densities of these measures,
locally uniformly outside a singleton. We first identify the density
of an infinitely divisible $\mu$, for which $\eta_{\mu}^{\langle-1\rangle}$
has the continuation $\Phi$ in (\ref{eq:extension of eta inverse (line)}),
in terms of the functions $f$ and $h$. The fact that the extension
of $\eta_{\mu}$ to $(0,+\infty)$ is continuous and injective shows
that $A_{\mu}=\{t\in(0,+\infty)\mathbb{:}\:\eta_{\mu}(t)=1\}$ is
either empty or a singleton. It is clear from the definition of $f$
that the set $A_{\mu}$ is nonempty precisely when $I_{1}(0)\le1$.
If this condition is satisfied, the set $A_{\mu}$ consists of $h(1)$
and $\mu(\{1/h(1)\})=1-I_{1}(0)$. Accordingly,
we denote $D_{\mu}=\{1/h(1)\}$ if $I_{1}(0)\le1$, and $D_{\mu}=\phi$
otherwise. It follows that $\mu$ is absolutely continuous with a
continuous density $p_{\mu}=d\mu/dt$ on $(0,+\infty)\backslash D_{\mu}$.
Equations (\ref{eq:param of eta_mu(pos line)}) and (\ref{eq:density vs eta})
give the implicit formula 
\begin{equation}
\frac{1}{h(r)}p_{\mu}\left(\frac{1}{h(r)}\right)=\frac{1}{\pi}\frac{r\sin f(r)}{|1-re^{if(r)}|^{2}},\qquad r>0,\;h(r)\notin A_{\mu}.\label{eq:the density in terms of h and r}
\end{equation}

We record for further use a simple consequence of (\ref{eq:the density in terms of h and r}).
For fixed $r$, the function $|1-re^{i\theta}|,$ $\theta\in\mathbb{R}$,
achieves its minimum at $\theta=0$, and thus
\[
\frac{1}{h(r)}p_{\mu}\left(\frac{1}{h(r)}\right)\le\frac{r}{\pi(1-r)^{2}},\quad r\in(0,+\infty)\setminus\left\{ 1\right\} ,
\]
or, equivalently,
\begin{equation}
tp_{\mu}(t)\le\frac{h^{\langle-1\rangle}(1/t)}{\pi(1-h^{\langle-1\rangle}(1/t))^{2}},\quad t\in(0,+\infty)\setminus D_{\mu}.\label{eq:xp(x) estimated (for endpoints)}
\end{equation}

\begin{prop}
\label{prop:unif convergence of inf div densities pos line} Let $\mu$
and $\{\mu_{n}\}_{n\in\mathbb{N}}$ be nondegenerate $\boxtimes$-infinitely
divisible measures in $\mathcal{P}_{\mathbb{R}_{+}}$ and let $U$
be an arbitrary open neighborhood of the set $D_{\mu}$; if $D_{\mu}=\varnothing$,
take $U=\varnothing$. Then $D_{\mu_{n}}\subset U$ for sufficiently
large $n$, and the functions $tp_{\mu_{n}}(t)$ converge to $tp_{\mu}(t)$
uniformly for $t\in(0,+\infty)\backslash U$.
\end{prop}

\begin{proof}
We use the notation established above: $\eta_{\mu_{n}}^{\langle-1\rangle}$
has the analytic continuation $\Phi_{n}$ determined by the parameters
$\gamma_{n}$ and $\sigma_{n}$, and $f_{n},h_{n}$ play the roles
of $f,h$ for the measure $\mu_{n}$. The relation $D_{\mu_{n}}=\left\{ 1/h_{n}(1)\right\} \subset U$
for large $n$ follows directly from Lemma 3.1(2). We focus on the
proof of uniform convergence. We show first that it suffices to prove
that $tp_{\mu_{n}}(t)$ converges to $tp_{\mu}(t)$ locally uniformly
on $(0,+\infty)\backslash U$. For this purpose, fix $\varepsilon>0$
and choose $\alpha,\beta\in(0,+\infty)$ such that
\[
\frac{x}{\pi(1-x)^{2}}<\varepsilon,\quad x\in(0,+\infty)\backslash[\alpha,\beta].
\]
Since $h^{\langle-1\rangle}$ is an increasing homeomorphism of $(0,+\infty)$,
there exist $a,b\in(0,+\infty)$ such that $h^{\langle-1\rangle}(1/b)<\alpha$
and $h^{\langle-1\rangle}(1/a)>\beta$. Lemma \ref{lem:f_n tends to f etc}
shows that there exists $N\in\mathbb{N}$ such that $h_{n}^{\langle-1\rangle}(1/b)<\alpha$
and $h_{n}^{\langle-1\rangle}(1/a)>\beta$ for $n\ge N,$ and hence
\[
tp_{\mu_{n}}(t),tp_{\mu}(t)<\varepsilon,\quad t\in(0,+\infty)\backslash[a,b],
\]
by (\ref{eq:xp(x) estimated (for endpoints)}). It suffices therefore
to prove uniform convergence on $[a,b]\backslash U$, and this would
follow from local uniform convergence on $(0,+\infty)\backslash D_{\mu}$.
For this purpose, it is convenient to write (\ref{eq:the density in terms of h and r})
in the explicit form
\begin{equation}
tp_{\mu}(t)=\frac{1}{\pi}\frac{h^{\langle-1\rangle}(1/t)\sin f(h^{\langle-1\rangle}(1/t))}{|1-h^{\langle-1\rangle}(1/t)e^{if(h^{\langle-1\rangle}(1/t))}|^{2}},\quad t\notin D_{\mu}.\label{eq: explicit 3.1}
\end{equation}
Suppose that $t_{0}\notin D_{\mu}$, and choose a compact neighborhood
$W$ of $t_{0}$ such that 
\[
1-h^{\langle-1\rangle}(1/t)e^{if(h^{\langle-1\rangle}(1/t))}\ne0,\quad t\in W.
\]
 Lemma \ref{lem:f_n tends to f etc} shows that there exists an integer
$N$ such that
\[
1-h_{n}^{\langle-1\rangle}(1/t)e^{if_{n}(h_{n}^{\langle-1\rangle}(1/t))}\ne0,\quad t\in W,\;n\ge N,
\]
and then we conclude from (\ref{eq: explicit 3.1}) (applied to $\mu_{n}$),
and from Lemma \ref{lem:f_n tends to f etc}, that $tp_{n}(t)$ converges
to $tp_{\mu}(t)$ uniformly on $W$.
\end{proof}
An immediate consequence is as follows.
\begin{cor}
\label{cor:uniform conv from xp(x)} Under the conditions of \textup{Proposition}
\emph{\ref{prop:unif convergence of inf div densities pos line},}
the sequence $\{p_{\mu_{n}}\}_{n\in\mathbb{N}}$ converges to $p$
locally uniformly on $(0,+\infty)\backslash D_{\mu}$.
\end{cor}

We can now prove a general version of superconvergence.
\begin{thm}
\label{thm:superconvergence Rplus}Let $k_{1}<k_{2}<\cdots$ be positive
integers, and let $\mu$ and $\{\nu_{n}\}_{n\in\mathbb{N}}$ be nondegenerate
measures in $\mathcal{P}_{\mathbb{R}_{+}}$ such that $\mu$ is $\boxtimes$-infinitely
divisible. Suppose that the sequence $\{\nu_{n}^{\boxtimes k_{n}}\}_{n\in\mathbb{N}}$
converges weakly to $\mu$. Let $K\subset(0,+\infty)\backslash D_{\mu}$
be an arbitrary compact set. Then $\nu_{n}^{\boxtimes k_{n}}$ is
absolutely continuous on $K$ for sufficiently large $n$, and the
sequence $\{d\nu_{n}^{\boxtimes k_{n}}/dt\}_{n\in\mathbb{N}}$ converges
to $d\mu/dt$ uniformly on $K$.
\end{thm}

\begin{proof}
As noted at the end of Section 2, there exist nondegenerate $\boxtimes$-infinitely
divisible measures $\mu_{n}\in\mathcal{P}_{\mathbb{R}_{+}}$ such
that $\eta_{\nu_{n}^{\boxtimes k_{n}}}=\eta_{\nu_{n}}\circ\eta_{\mu_{n}}$
and
\begin{equation}
\eta_{\mu_{n}}(1/x)\left[x\eta_{\mu_{n}}(1/x)\right]^{1/(k_{n}-1)}=\eta_{\nu_{n}^{\boxtimes k_{n}}}(1/x),\quad x\in(0,+\infty),\;n\in\mathbb{N}.\label{eq:even more eta vs eta}
\end{equation}
It is known from \cite{BPata-mult-laws} that the sequence $\{\nu_{n}\}_{n\in\mathbb{N}}$
converges weakly to $\delta_{1}$, and therefore the functions $\eta_{\nu_{n}}(z)$
converge to $z$ uniformly on compact subsets of $\mathbb{C}\backslash\mathbb{R}_{+}$.
Similarly, the functions $\eta_{\nu_{n}}^{\langle-1\rangle}(z)$ converge
uniformly to $z$ on compact subsets of $(-\infty,0).$ Since $\eta_{\nu_{n}^{\boxtimes k_{n}}}$
converges to $\eta_{\mu}$ uniformly on compact subsets of $\mathbb{C}\backslash\mathbb{R}_{+}$,
we deduce that $\eta_{\mu_{n}}(z)=\eta_{\nu_{n}}^{\langle-1\rangle}(\eta_{\nu_{n}^{\boxtimes k_{n}}}(z))$
converge to $\eta_{\mu}$ uniformly on compact subsets of $(-\infty,0)$.
It follows that the sequence $\{\mu_{n}\}_{n\in\mathbb{N}}$ converges
weakly to $\mu$. By Lemma \ref{lem:f_n tends to f etc}(5), $\eta_{\mu_{n}}(x)$
tends to $\eta_{\mu}(x)$ uniformly on compact subsets of $(0,+\infty)$,
and therefore 
\[
\left[x\eta_{\mu_{n}}(1/x)\right]^{1/(k_{n}-1)}
\]
converges to $1$ uniformly on compact subsets of $(0,+\infty)$ since
$k_{n}\to\infty$. Then (\ref{eq:even more eta vs eta}) shows that
$\eta_{\nu_{n}^{\boxtimes k_{n}}}(1/x)$ converges to $\eta_{\mu}(1/x)$
uniformly on compact subsets of $(0,+\infty)$. The conclusion of
the theorem follows now from (\ref{eq:density vs eta}) applied to
these measures, as in the proof of Proposition \ref{prop:unif convergence of inf div densities pos line}. 
\end{proof}

\section{Cusp behavior in $\mathcal{P}_{\mathbb{R}_{+}}$\label{sec:cusps-in R_+} }

In this section, we describe the qualitative behavior of a convolution
$\mu_{1}\boxtimes\mu_{2}$, where $\mu_{1},\mu_{2}\in\mathcal{P}_{\mathbb{R}_{+}}$
are nondegenerate measures and $\mu_{2}$ is $\boxtimes$-infinitely
divisible, subject to a mild additional condition. It was shown in
\cite{Bi-cusp} how an analytic function argument provides examples
in which the density of $\mu\boxplus\nu$, with $\nu$ a semicircle
law, can have a cusp behavior at some points. More precisely, if $h$
is the density, then, at some of its zeros $t_{0}\in\mathbb{R}$, the ratio
$h(t)/|t-t_{0}|^{1/3}$ is bounded away from zero and infinity. Then it is shown in \cite{Bi-cusp}
 that this is the worst possible cusp behavior that such a density
can have. Arguments, similar to those in \cite{Bi-cusp} , show that
the density of $\mu_{1}\boxtimes\mu_{2}$ can also be bounded by a
cubic root near a zero if $\mu_{2}$ is the multiplicative analog
of the semicircular law. Our purpose in this section is to show that
this is the worst possible behavior for such densities if $\mu_{2}$
is an almost arbitrary $\boxtimes$-infinitely divisible measure.
The argument proceeds in two steps. First, we work with the case in
which $\mu_{2}$ is the multiplicative analog of a semicircular measure,
thus producing a multiplicative analog of Proposition
4 and Corollary 5 in \cite{Bi-cusp}. For general $\mu_{2}$, we show that the density
of $\mu_{1}\boxtimes\mu_{2}$ can be estimated using a different convolution
$\nu_{1}\boxtimes\nu_{2}$, where $\nu_{2}$ is one of these multiplicative
analogs of the semicircular measure, chosen with appropriate parameters.

We recall an observation first made in \cite{Bi-cusp} in the free
additive case. (The simple proof is provided for convenience as well
as for establishing notation.)
\begin{lem}
\label{lem:omega is inf-div} Let $\mu_{1},\mu_{2}\in\mathcal{P}_{\mathbb{R}_{+}}$
be such that $\mu_{2}$ is $\boxtimes$-infinitely divisible, and
let $\rho_{1},\rho_{2}\in\mathcal{P}_{\mathbb{R}_{+}}$ be given by
\textup{Theorem} \emph{\ref{thm:subordination on the line (mult)}}.
Then $\rho_{1}$ is $\boxtimes$-infinitely divisible.
\end{lem}

\begin{proof}
Let $\Phi$ given by (\ref{eq:extension of eta inverse (line)}) be
the analytic continuation of $\eta_{\mu_{2}}^{\langle-1\rangle}$.
Then (\ref{eq:defining boxtimes}) can be written as 
\[
\frac{\Phi(z)}{z}\eta_{\mu_{1}}^{-1}(z)=\eta_{\mu_{1}\boxtimes\mu_{2}}^{\langle-1\rangle}(z),
\]
and applying this equality with $\eta_{\mu_{1}}(z)$ in place of $z$,
we obtain
\[
\frac{\Phi(\eta_{\mu_{1}}(z))}{\eta_{\mu_{1}}(z)}z=\eta_{\mu_{1}\boxtimes\mu_{2}}^{\langle-1\rangle}(\eta_{\mu_{1}}(z))=\eta_{\rho_{1}}^{\langle-1\rangle}(z)
\]
for $z\in(\beta,0)$ and $\beta<0$. The lemma follows because the
function
\begin{equation}
\Psi(z)=\frac{\Phi(\eta_{\mu_{1}}(z))}{\eta_{\mu_{1}}(z)}z,\quad z\in\mathbb{C}\backslash\mathbb{R}_{+},\label{eq:big psi line}
\end{equation}
 is of the form $z\exp(v(z)),$ where 
\[
v(z)=\log\gamma+\int_{[0,+\infty]}\frac{1+t\eta_{\mu_{1}}(z)}{\eta_{\mu_{1}}(z)-t}\,d\sigma(t),\quad z\in\mathbb{C}\setminus\mathbb{R}_{+},
\]
is an analytic function satisfying $v(\mathbb{H})\subset-\mathbb{H}$
(since $\eta_{\mu_{1}}(\mathbb{H})\subset\mathbb{H})$ and $v(\overline{z})=\overline{v(z)}$
for $z\in\mathbb{C}\setminus\mathbb{R}_{+}$.
\end{proof}
With the notation of the preceding lemma, we recall that the domain
\[
\eta_{\rho_{1}}(\mathbb{H})=\Omega_{\rho_{1}}\cap\mathbb{H}
\]
can be described as
\[
\eta_{\rho_{1}}(\mathbb{H})=\{re^{i\theta}:r>0,f(r)<\theta<\pi\}
\]
for some continuous function $f:(0,+\infty)\to[0,\pi)$, and that
$\eta_{\rho_{1}}$ extends to a homeomorphism of $\overline{\mathbb{H}}$
onto $\overline{\eta_{\rho_{1}}(\mathbb{H})}$. It was shown in \cite{huang-wang}
that $\eta_{\mu_{1}}$ extends continuously to $\overline{\eta_{\rho_{1}}(\mathbb{H})}$
provided that we allow $\infty$ as a possible value. Using, as before,
the increasing homeomorphism
\[
h(r)=\Psi(re^{if(r)}),\quad r\in(0,+\infty),
\]
the density $q_{\mu_{1}\boxtimes\mu_{2}}$ of $\mu_{1}\boxtimes\mu_{2}$,
relative to the Haar measure $dx/x$ on $(0,+\infty)$, is calculated
using the formula
\begin{equation}
q_{\mu_{1}\boxtimes\mu_{2}}(1/x)=\begin{cases}
\frac{1}{\pi}\Im\frac{1}{1-\eta_{\mu_{1}}(re^{if(r)})}, & x=h(r)\text{ and }f(r)>0,\\
0, & x=h(r)\text{ and }f(r)=0.
\end{cases}\label{eq:density boxtimes line}
\end{equation}

The following proposition examines the density of $\mu_{1}\boxtimes\mu_{2}$
when $\mu_{2}$ is analogous to the semicircular measure, that is,
when $\sigma$ is a point mass at $t=1$. (The equation (\ref{eq:density, semicircle xtimes})
regarding this density also appeared in \cite{Zhong}.)
\begin{prop}
\label{prop:boxtimes convo with semisemi} Suppose that $\beta,\gamma\in(0,+\infty)$,
and that $\mu_{2}\in\mathcal{P}_{\mathbb{R}_{+}}$ is such that
\[
\gamma z\exp\left[\beta\frac{z+1}{z-1}\right],\quad z\in\mathbb{C}\backslash\mathbb{R}_{+}
\]
is an analytic continuation of $\eta_{\mu_{2}}^{\langle-1\rangle}$.
Let $q_{\mu_{1}\boxtimes\mu_{2}}$ be the density of $\mu_{1}\boxtimes\mu_{2}$
relative to the Haar measure $dx/x$ and define $k(x)=q_{\mu_{1}\boxtimes\mu_{2}}(1/x).$
Then\emph{:}
\begin{enumerate}
\item $\left|k'(x)\right|k(x)^{2}\le (4\pi^{3}\beta^{2}x)^{-1}$ for every $x\in(0,+\infty)$
such that $k(x)\ne0$.
\item If $I\subset\mathbb{R}_{+}$ is an interval with one endpoint $x_{0}>0$,
$k(x)>0$ for $x\in I$, and $k(x_{0})=0$, then
\[
k(x)^{3}\le\frac{3}{4\pi^{3}\beta^{2}}|\log x-\log x_{0}|,\quad x\in I.
\]
In particular, $k(x)/|x-x_{0}|^{1/3}$ and $k(x)/|x^{-1}-x_{0}^{-1}|^{1/3}$
remain bounded for $x\in I$ close to $x_{0}$.
\end{enumerate}
\end{prop}

\begin{proof}
Part (2) follows from (1) because
\[
k(x)^{3}=\left|\int_{x_{0}}^{x}3k(s)^{2}k'(s)\,ds\right|.
\]
 By (\ref{eq:big psi line}), we have 
\begin{align*}
\Psi(z) & =\gamma z\exp\left[\beta\frac{\eta_{\mu_{1}}(z)+1}{\eta_{\mu_{1}}(z)-1}\right]\\
 & =\gamma z\exp\beta\left[1-\frac{2}{1-\eta_{\mu_{1}}(z)}\right]\\
 & =e^{-\beta}\gamma z\exp\left[-2\beta\psi_{\mu_{1}}(z)\right].
\end{align*}
The fact that $\arg\Psi\left(re^{if(r)}\right)=0$ leads to the identity
\begin{equation}
f(r)=2\beta\Im\frac{1}{1-\eta_{\mu_{1}}(re^{if(r)})}=2\pi\beta k(\Psi(re^{if(r)})).\label{eq:density, semicircle xtimes}
\end{equation}
We note for further use that
\begin{equation}
\Im\frac{1}{1-\eta_{\mu_{1}}(re^{if(r)})}=\Im(1+\psi_{\mu_{1}}(re^{if(r)}))=\int_{\mathbb{R}_{+}}\frac{tr\sin(f(r))}{|1-tre^{if(r)}|^{2}}\,d\mu_{1}(t).\label{eq:useful soon}
\end{equation}
Of course, our estimate applies to points $x=x(r)=\Psi(re^{if(r)})$
such that $f(r)>0$, and $f$ is continuously differentiable at such $r$. By
the chain rule and \eqref{eq:density, semicircle xtimes},
\begin{equation}
k'(x(r))=\frac{(d/dr)k(x(r))}{(d/dr)x(r)}=\frac{(1/2\pi\beta)f'(r)}{[x'(r)/x(r)]x(r)},\label{eq:chain rule}
\end{equation}
and thus we must find lower estimates for the logarithmic derivative
$x'(r)/x(r)$. We have
\begin{align*}
\left|\frac{(d/dr)\Psi(re^{if(r)})}{\Psi(re^{if(r)})}\right| & =\left|\frac{\Psi'(re^{if(r)})}{\Psi(re^{if(r)})}\right|\left|\frac{d(re^{if(r)})}{dr}\right|\\
 & =\left|\frac{1}{re^{if(r)}}-2\beta\psi'_{\mu_{1}}(re^{if(r)})\right|\left|e^{if(r)}(1+irf'(r))\right|\\
 & =\frac{1}{r}\left|1-2\beta re^{if(r)}\psi'_{\mu_{1}}(re^{if(r)})\right|\sqrt{1+r^{2}f'(r)^{2}}.
\end{align*}
Observe that 
\[
\psi_{\mu_{1}}'(re^{if(r)})=\int_{\mathbb{R}_{+}}\frac{t}{(1-tre^{if(r)})^{2}}\,d\mu_{1}(t),
\]
and use relations (\ref{eq:density, semicircle xtimes}) and (\ref{eq:useful soon})
to see that
\begin{align*}
1-2\beta re^{if(r)}\psi'_{\mu_{1}}(re^{if(r)}) & =\frac{2\beta}{f(r)}\Im\frac{1}{1-\eta_{\mu_{1}}(re^{if(r)})}-2\beta re^{if(r)}\psi'_{\mu_{1}}(re^{if(r)})\\
 & =2\beta\int_{\mathbb{R}_{+}}\left[\frac{1}{f(r)}\frac{tr\sin(f(r))}{|1-tre^{if(r)}|^{2}}-\frac{tre^{if(r)}}{(1-tre^{if(r)})^{2}}\right]\,d\mu_{1}(t).
\end{align*}
We now calculate 
\begin{align*}
\Re & \left[\frac{1}{f(r)}\frac{tr\sin(f(r))}{|1-tre^{if(r)}|^{2}}-\frac{tre^{if(r)}}{(1-tre^{if(r)})^{2}}\right]\\
 & =tr\frac{\sin(f(r))|1-tre^{if(r)}|^{2}-f(r)\Re\left[tre^{if(r)}(1-tre^{-if(r)})^{2}\right]}{f(r)|1-tre^{if(r)}|^{4}}\\
 & =tr\frac{(1+t^{2}r^{2})\left[\sin(f(r))-f(r)\cos(f(r))\right]+tr\left[2f(r)-\sin(2f(r))\right]}{f(r)|1-tre^{if(r)}|^{4}}\\
 & \ge\frac{t^{2}r^{2}\left[2f(r)-\sin(2f(r))\right]}{f(r)|1-tre^{if(r)}|^{4}},
\end{align*}
where we used the fact that $\sin f-f\cos f\ge0$ for $f\in(0,\pi).$
Thus,
\begin{align*}
|1-2\beta re^{if(r)}\psi'_{\mu_{1}}(re^{if(r)})| & \ge2\beta\int_{\mathbb{R}_{+}}\Re\left[\frac{1}{f(r)}\frac{tr\sin(f(r))}{|1-tre^{if(r)}|^{2}}-\frac{tre^{if(r)}}{(1-tre^{if(r)})^{2}}\right]\,d\mu_{1}(t).\\
 & \ge2\beta\frac{2f(r)-\sin(2f(r))}{f(r)}\int_{\mathbb{R}_{+}}\frac{t^{2}r^{2}}{|1-tre^{if(r)}|^{4}}\,d\mu_{1}(t)\\
\text{(Schwarz inequality)} & \ge2\beta\frac{2f(r)-\sin(2f(r))}{f(r)}\left[\int_{\mathbb{R}_{+}}\frac{tr}{|1-tre^{if(r)}|^{2}}\,d\mu_{1}(t)\right]^{2}\\
\text{(by (\ref{eq:density, semicircle xtimes}) and (\ref{eq:useful soon})) } & =2\beta\frac{2f(r)-\sin(2f(r))}{f(r)}\left[\frac{f(r)}{2\beta\sin(f(r))}\right]^{2}\\
 & =\frac{2f(r)-\sin(2f(r))}{f(r)\sin^{2}(f(r))}\left[\frac{f(r)^{2}}{2\beta}\right].
\end{align*}
A further lower bound is obtained using the inequality $2f-\sin(2f)\ge f\sin^{2}f$,
valid for $f\in(0,\pi)$. We obtain
\begin{align*}
\left|\frac{x'(r)}{x(r)}\right| & =\left|\frac{(d/dr)\Psi(re^{if(r)})}{\Psi(re^{if(r)})}\right|\\
 & \ge\frac{f(r)^{2}}{2\beta}\frac{\sqrt{1+r^{2}f'(r)^{2}}}{r},
\end{align*}
and finally from (\ref{eq:chain rule}),
\begin{align*}
|k'(x(r))| & =\left|\frac{(1/2\pi\beta)f'(r)}{(x'(r)/x(r))x(r)}\right|\\
 & \le\frac{1}{\pi f(r)^{2}x(r)}\frac{r|f'(r)|}{\sqrt{1+r^{2}f'(r)^{2}}}\\
 & \le\frac{1}{\pi f(r)^{2}x(r)}.
\end{align*}
By (\ref{eq:density, semicircle xtimes}), this is precisely the inequality in (1).
\end{proof}
\begin{rem}
With the notation of the preceding lemma, it is easy to verify that
the inequality in (1) is equivalent to
\[
\left|q'_{\mu_{1}\boxtimes\mu_{2}}(x)\right|q_{\mu_{1}\boxtimes\mu_{2}}(x)^{2}\le \frac{1}{4\pi^{3}\beta^{2}x},\quad \text{where} \quad x\in (0,+\infty),\,\,q_{\mu_{1}\boxtimes\mu_{2}}(x)\ne0.
\]
\end{rem}

One essential observation that allows us to extend the preceding result
to more general $\boxtimes$-infinitely divisible measures $\mu_{2}$
is as follows. The density of $\mu_{1}\boxtimes\mu_{2}$ depends largely,
via (\ref{eq:density boxtimes line}) on the function $f$, and thus
on the $\boxtimes$-infinitely divisible measure $\rho_{1}$. In many
cases, it is possible to find another convolution $\nu_{1}\boxtimes\nu_{2}$,
such that $\eta_{\nu_{1}\boxtimes\nu_{2}}=\eta_{\nu_{1}}\circ\eta_{\rho_{1}}$
(with the same measure $\rho_{1}),$ and such that $\nu_{2}$ is a
multiplicative analog of the semicircular measure. The verification
of the following result is a simple calculation. The details are left
to the reader. Note that the existence of the measure $\nu_{1}$ below
follows from Lemmas 2.1 and 2.3.
\begin{lem}
\label{lem:trade a free convolution for another}Let $\mu_{1},\mu_{2}\in\mathcal{P}_{\mathbb{R}_{+}}$
be such that $\mu_{2}$ is $\boxtimes$-infinitely divisible, and
let
\[
\Phi(z)=\gamma z\exp\left[\int_{[0,+\infty]}\frac{1+tz}{z-t}\,d\sigma(t)\right],\quad z\in\mathbb{C}\backslash\mathbb{R}_{+},
\]
be an analytic continuation of $\eta_{\mu_{2}}^{\langle-1\rangle}$.
Denote by $\rho_{1}\in\mathcal{P}_{\mathbb{R}_{+}}$ the $\boxtimes$-infinitely
divisible measure such that $\eta_{\rho_{1}}^{\langle-1\rangle}$
has the analytic continuation
\[
\Psi(z)=\gamma z\exp\left[\int_{[0,+\infty]}\frac{1+t\eta_{\mu_{1}}(z)}{\eta_{\mu_{1}}(z)-t}\,d\sigma(t)\right],\quad z\in\mathbb{C}\backslash\mathbb{R}_{+}.
\]
 Suppose that
\[
\beta=\frac{1}{2}\int_{[0,+\infty]}\left(\frac{1}{t}+t\right)\,d\sigma(1/t)
\]
is finite and nonzero. Denote by $\nu_{1}\in\mathcal{P}_{\mathbb{R}_{+}}$
the measure satisfying
\[
\psi_{\nu_{1}}(z)=\frac{1}{2\beta}\int_{[0,+\infty]}\frac{t\eta_{\mu_{1}}(z)}{1-t\eta_{\mu_{1}}(z)}\,\left(\frac{1}{t}+t\right)\,d\sigma(1/t),\quad z\in\mathbb{C}\backslash\mathbb{R}_{+},
\]
and denote by $\nu_{2}\in\mathcal{P}_{\mathbb{R}_{+}}$ the $\boxtimes$-infinitely
divisible measure such that $\eta_{\nu_{2}}^{\langle-1\rangle}$ has
the analytic continuation
\[
\gamma^{\prime}z\exp\left[\beta\frac{z+1}{z-1}\right],\quad z\in\mathbb{C}\backslash\mathbb{R}_{+},
\]
where 
\[
\gamma^{\prime}=\gamma\exp\left[\frac{1}{2}\int_{[0,+\infty]}\left(\frac{1}{t}-t\right)\,d\sigma(1/t)\right].
\]
Then $\eta_{\mu_{1}\boxtimes\mu_{2}}=\eta_{\mu_{1}}\circ\eta_{\rho_{1}}$
and $\eta_{\nu_{1}\boxtimes\nu_{2}}=\eta_{\nu_{1}}\circ\eta_{\rho_{1}}$.
\end{lem}

For the final proof in this section, we need some results from \cite{huang-wang},
which we formulate using the notation established in Lemma \ref{lem:omega is inf-div}.
According to \cite[Theorem 4.16]{huang-wang}, the zero set $\{\alpha\in(0,+\infty):f(\alpha)=0\}$
can be partitioned into three sets $A,B,C$ defined as follows.
\begin{enumerate}
\item The set $A$ consists of those $\alpha\in(0,+\infty)$ such that $\mu_{1}(\{1/\alpha\})>0$
and
\[
\int_{[0,+\infty]}\frac{1+t^{2}}{(1-t)^{2}}\,d\sigma(t)\le\mu_{1}(\{1/\alpha\}).
\]
\item The set $B$ consists of those $\alpha\in(0,+\infty)$ for which $\eta_{\mu_{1}}(\alpha)\in\mathbb{R}\backslash\{1\}$
and
\[
\left[\int_{\mathbb{R}_{+}}\frac{\alpha t}{(1-\alpha t)^{2}}\,d\mu_{1}(t)\right]\left[\int_{[0,+\infty]}\frac{1+t^{2}}{(\eta_{\mu_{1}}(\alpha)-t)^{2}}\,d\sigma(t)\right]\le\frac{1}{(1-\eta_{\mu_{1}}(\alpha))^{2}}.
\]
\item Finally, $\alpha\in C$ provided that $\eta_{\mu_{1}}(\alpha)=\infty$
and
\[
\left[\int_{\mathbb{R}_{+}}\frac{d\mu_{1}(t)}{(1-\alpha t)^{2}}\right]\left[\int_{[0,+\infty]}(1+t^{2})\,d\sigma(t)\right]\le1.
\]
 
\end{enumerate}
The proof of this result relies on Proposition 4.10 of \cite{B-B-IMRN} which states that $f(\alpha)=0$ if and only if the map $\Psi$ has a finite Julia-Carath\'{e}odory derivative at the point $\alpha$, so that the preceding inequalities are a consequence of the chain rule for Julia-Carath\'{e}odory derivative. (See \cite{Shapiro} for the basics of Julia-Carath\'{e}odory derivative.) The density of $\mu_{1}\boxtimes\mu_{2}$ is continuous everywhere,
except on the finite set 
\[
\{1/\Psi(\alpha):\alpha\in A\}.
\]
 If $x\in(0,+\infty)$ is an atom of $\mu_{1}\boxtimes\mu_{2}$ then
$\eta_{\rho_{1}}(1/x)\in A$.
\begin{thm}
\label{thm:cusp on R+}Let $\mu_{1},\mu_{2}\in\mathcal{P}_{\mathbb{R}_{+}}$
be two nondegenerate measures such that $\mu_{2}$ is $\boxtimes$-infinitely
divisible, and let
\[
\Phi(z)=\gamma z\exp\left[\int_{[0,+\infty]}\frac{1+tz}{z-t}\,d\sigma(t)\right],\quad z\in\mathbb{C}\backslash\mathbb{R}_{+},
\]
be an analytic continuation of $\eta_{\mu_{2}}^{\langle-1\rangle}$.
Suppose that $\sigma(0,+\infty)>0$. If $I\subset(0,+\infty)$ is
an open interval with an endpoint $x_{0}>0$ such that $1/\eta_{\rho_{1}}(1/x_{0})$
is not an atom of $\mu_{1}$, and $q_{\mu_{1}\boxtimes\mu_{2}}(x_{0})=0<q_{\mu_{1}\boxtimes\mu_{2}}(x)$
for every $x\in I$, then $q_{\mu_{1}\boxtimes\mu_{2}}(x)/|x-x_{0}|^{1/3}$
is bounded for $x\in I$ close to $x_{0}$.
\end{thm}

\begin{proof}
We can always find finite measures $\sigma'$ and $\sigma''$ on $[0,+\infty]$
such that $\sigma=\sigma'+\sigma'',$ $\sigma''\ne0$ and $\sigma''$
has compact support contained in $(0,+\infty)$. The $\boxtimes$-infinitely
divisible measures $\mu'_{2},\mu_{2}''\in\mathcal{P}_{\mathbb{R_{+}}}$,
defined by the fact that $\eta_{\mu_{2}'}^{\langle-1\rangle}$ and
$\eta_{\mu_{2}^{\prime\prime}}^{\langle-1\rangle}$ have analytic
continuations
\[
\gamma z\exp\left[\int_{[0,+\infty]}\frac{1+tz}{z-t}\,d\sigma'(t)\right]\text{ and }z\exp\left[\int_{[0,+\infty]}\frac{1+tz}{z-t}\,d\sigma''(t)\right],\quad z\in\mathbb{C}\backslash\mathbb{R}_{+},
\]
respectively, satisfy the relation $\mu_{2}'\boxtimes\mu_{2}''=\mu_{2}$,
and thus $\mu_{1}\boxtimes\mu_{2}=\mu_{1}''\boxtimes\mu_{2}''$, where
$\mu_{1}''=\mu_{1}\boxtimes\mu_{2}'$. There exist additional $\boxtimes$-infinitely
divisible measures $\rho_{1}',\rho_{1}''\in\mathcal{P}_{\mathbb{R_{+}}}$
such that $\eta_{\mu_{1}''}=\eta_{\mu_{1}}\circ\eta_{\rho_{1}'}$
and $\eta_{\mu_{1}\boxtimes\mu_{2}}=\eta_{\mu_{1}''}\circ\eta_{\rho_{1}''}$.
Clearly, $\eta_{\rho_{1}}=\eta_{\rho_{1}'}\circ\eta_{\rho_{1}''}$,
and we argue that $1/\eta_{\rho_{1}''}(1/x_{0})$ is a real number
but not an atom of $\mu_{1}''$. Indeed, letting $z\rightarrow1/x_{0}$
in the inequality
\[
\arg\eta_{\rho_{1}}(z)=\arg(\eta_{\rho_{1}'}(\eta_{\rho_{1}''}(z))\ge\arg\eta_{\rho_{1}''}(z),\quad z\in\mathbb{H},
\]
the hypothesis $\eta_{\rho_{1}}(1/x_{0})\in(0,+\infty)$ implies that
$\eta_{\rho_{1}''}(1/x_{0})\in(0,+\infty)$. Suppose, to get a contradiction,
that $1/\eta_{\rho_{1}''}(1/x_{0})$ is an atom of $\mu_{1}''$. Then,
as seen in \cite{bel-atomX}, 
\[
1/\eta_{\rho_{1}}(1/x_{0})=1/\eta_{\rho_{1}'}(\eta_{\rho_{1}''}(1/x_{0}))
\]
 is necessarily an atom of $\mu_{1}$, contrary to the hypothesis.

The above construction shows that the hypothesis of the theorem also
holds with $\mu_{1}'',\mu_{2}''$, and $\rho_{1}''$ in place of $\mu_{1},\mu_{2},$
and $\rho_{1}$, respectively. Moreover, it is obvious that $\int_{[0,+\infty]}((t^{2}+1)/t)\,d\sigma''(t)<+\infty$.
Therefore we may, and do, assume that the additional hypothesis $\int_{[0,+\infty]}((t^{2}+1)/t)\,d\sigma(t)<+\infty$
is satisfied. In particular, the hypothesis of Lemma \ref{lem:trade a free convolution for another}
is satisfied. With the notation of that lemma, Proposition \ref{prop:boxtimes convo with semisemi}
shows that it suffices to prove that $q_{v_{1}\boxtimes\nu_{2}}(x)/q_{\mu_{1}\boxtimes\mu_{2}}(x)$
is bounded away from zero for $x\in I$ close to $x_{0}$. For this
purpose, we write points $x\in(0,+\infty)$ as $x=1/\Psi(re^{if(r)})$.
In particular, $x_{0}=1/\Psi(r_{0}e^{if(r_{0})})$ and $f(r_{0})=0$.
The fact that $1/\eta_{\rho_{1}}(1/x_{0})$ is not an atom for $\mu_{1}$
implies that $r_{0}\in B\cup C$. The formula (\ref{eq:density boxtimes line})
and the definition of $\nu_{1}$ yield
\begin{align*}
q_{\nu_{1}\boxtimes\nu_{2}}(x) & =\frac{1}{\pi}\Im\frac{1}{1-\eta_{\nu_{1}}(re^{if(r)})}=\frac{1}{\pi}\Im\psi_{\nu_{1}}(re^{if(r)})\\
 & =\frac{1}{2\pi\beta}\Im\left[\int_{[0,+\infty]}\frac{t\eta_{\mu_{1}}(re^{if(r)})}{1-t\eta_{\mu_{1}}(re^{if(r)})}\left(t+\frac{1}{t}\right)\,d\sigma(1/t)\right]\\
 & =\frac{\Im\eta_{\mu_{1}}(re^{if(r)})}{2\pi\beta}\int_{[0,+\infty]}\frac{1+t^{2}}{|t-\eta_{\mu_{1}}(re^{if(r)})|^{2}}\,d\sigma(t).
\end{align*}
Since we also have
\[
q_{\mu_{1}\boxtimes\mu_{2}}(x)=\frac{1}{\pi}\Im\frac{1}{1-\eta_{\mu_{1}}(re^{if(r)})}=\frac{1}{\pi}\frac{\Im\eta_{\mu_{1}}(re^{if(r)})}{|1-\eta_{\mu_{1}}(re^{if(r)})|^{2}},
\]
we deduce that
\begin{equation}
\frac{q_{\nu_{1}\boxtimes\nu_{2}}(x)}{q_{\mu_{1}\boxtimes\mu_{2}}(x)}=\frac{|1-\eta_{\mu_{1}}(re^{if(r)})|^{2}}{2\beta}\int_{[0,+\infty]}\frac{1+t^{2}}{|t-\eta_{\mu_{1}}(re^{if(r)})|^{2}}\,d\sigma(t).\label{eq:use in remark}
\end{equation}
 Letting $x\to x_{0}$, so $r\to r_{0}$, we see that
\[
\liminf_{x\to x_{0},x\in I}\frac{q_{\nu_{1}\boxtimes\nu_{2}}(x)}{q_{\mu_{1}\boxtimes\mu_{2}}(x)}\ge\frac{|1-\eta_{\mu_{1}}(r_{0})|^{2}}{2\beta}\int_{[0,+\infty]}\frac{1+t^{2}}{|t-\eta_{\mu_{1}}(r_{0})|^{2}}\,d\sigma(t)
\]
 if $r_{0}\in B$, and 
\[
\liminf_{x\to x_{0},x\in I}\frac{q_{\nu_{1}\boxtimes\nu_{2}}(x)}{q_{\mu_{1}\boxtimes\mu_{2}}(x)}\ge\frac{1}{2\beta}\int_{[0,+\infty]}(1+t^{2})\,d\sigma(t)
\]
 if $r_{0}\in C$. In either case, the lower estimate is strictly
positive.
\end{proof}
\begin{rem}
\label{rem:universality, reverse inequality} In the above proof,
we show that $q_{\mu_{1}\boxtimes\mu_{2}}(x)=O(q_{\nu_{1}\boxtimes\nu_{2}}(x))$
as $x\to x_{0},x\in I$. It is also true that $q_{\nu_{1}\boxtimes\nu_{2}}(x)=O(q_{\mu_{1}\boxtimes\mu_{2}}(x))$
as $x\to x_{0},x\in I$. To see this, we observe that the definition
of $f$ implies the equality
\[
f(r)=\Im\eta_{\mu_{1}}(re^{if(r)})\int_{[0,+\infty]}\frac{1+t^{2}}{|t-\eta_{\mu_{1}}(re^{if(r)})|^{2}}\,d\sigma(t).
\]
Thus, the reciprocal of the fraction in (\ref{eq:use in remark})
can be rewritten as
\begin{align*}
\frac{q_{\mu_{1}\boxtimes\mu_{2}}(x)}{q_{\nu_{1}\boxtimes\nu_{2}}(x)} & =\frac{2\beta}{|1-\eta_{\mu_{1}}(re^{if(r)})|^{2}}\frac{\Im\eta_{\mu_{1}}(re^{if(r)})}{f(r)}\\
 & =\frac{2\beta r\sin(f(r))}{f(r)}\frac{\Im\eta_{\mu_{1}}(re^{if(r)})}{\Im(re^{if(r)})|1-\eta_{\mu_{1}}(re^{if(r)})|^{2}}\\
 & =\frac{2\beta r\sin(f(r))}{f(r)}\frac{\Im\psi_{\mu_{1}}(re^{if(r)})}{\Im(re^{if(r)})}\\
 & =\frac{2\beta r\sin(f(r))}{f(r)}\int_{\mathbb{R}_{+}}\frac{t}{\left|1-tre^{if(r)}\right|^{2}}\,d\mu_{1}(t).
\end{align*}
Letting $x\to x_{0}$ yields
\[
\liminf_{x\to x_{0},x\in I}\frac{q_{\mu_{1}\boxtimes\mu_{2}}(x)}{q_{\nu_{1}\boxtimes\nu_{2}}(x)}\ge2\beta\int_{\mathbb{R}_{+}}\frac{tr_{0}}{(1-tr_{0})^{2}}\,d\mu_{1}(t)>0.
\]
Note that the quantity on the right hand side is in fact finite. This
is immediate if $r_{0}\in B$, and it follows from the identity
\[
\int_{\mathbb{R}_{+}}\frac{tr_{0}}{(1-tr_{0})^{2}}\,d\mu_{1}(t)=\int_{\mathbb{R}_{+}}\frac{1}{(1-tr_{0})^{2}}\,d\mu_{1}(t)-\int_{\mathbb{R}_{+}}\frac{1}{1-tr_{0}}\,d\mu_{1}(t)
\]
 if $r_{0}\in C$.
\begin{rem}
There are cases, other than those of Proposition \ref{prop:boxtimes convo with semisemi},
in which the set $\{x_{0}>0:\eta_{\rho_{1}}(1/x_{0})\in A\}$ is
empty, and thus the conclusion of Theorem \ref{thm:cusp on R+} holds at every zero of $q_{\mu_1\boxtimes \mu_{2}}$; in particular, the result holds at cusps and at the edge of every connected component of the set $\{x:p_{\mu_{1}\boxtimes\mu_{2}}(x)>0\}$.
See Remark \ref{rem:no exceptions} for a brief discussion in the
context of additive free convolution.
\end{rem}

\end{rem}

\section{Free multiplicative convolution on $\mathbb{T}$\label{sec:Free-mutiplicative-convolution on T}}

We denote by $\mathcal{P}_{\mathbb{T}}$ the collection of probability measures on $\mathbb{T}$. The definition of the moment
generating function for a measure $\mu\in\mathcal{P}_{\mathbb{T}}$
is analogous to the one used for $\mathcal{P}_{\mathbb{R}_{+}}$,
but the domain is now the unit disk $\mathbb{D}=\{z\in\mathbb{C}:|z|<1\}$:
\[
\psi_{\mu}(z)=\int_{\mathbb{T}}\frac{tz}{1-tz}\,d\mu(t),\quad z\in\mathbb{D}.
\]
The $\eta$-transform of $\mu$ is the function
\[
\eta_{\mu}(z)=\frac{\psi_{\mu}(z)}{1+\psi_{\mu}(z)},\quad z\in\mathbb{D}.
\]
 The collection $\{\eta_{\mu}:\mu\in\mathcal{P}_{\mathbb{T}}\}$ is
simply the set of all analytic functions $f:\mathbb{D}\to\mathbb{D}$
that satisfy $f(0)=0$. If we denote by 
\[
H_{\mu}(z)=\int_{\mathbb{T}}\frac{t+z}{t-z}\,d\mu(t),\quad z\in\mathbb{D},
\]
the Herglotz integral of $\mu$, and if we define $\mu_{*}\in\mathcal{P}_{\mathbb{T}}$
by $d\mu_{*}(t)=d\mu(1/t)$, then
\[
H_{\mu_{*}}(z)=1+2\psi_{\mu}(z)=\frac{1+\eta_{\mu}(z)}{1-\eta_{\mu}(z)},\quad z\in\mathbb{D}.
\]
Since $\Re H_{\mu_{*}}$ is the Poisson integral of $\mu_{*}$, we
deduce that the measures
\[
\frac{1}{2\pi}\Re\frac{1+\eta_{\mu}(re^{-i\theta})}{1-\eta_{\mu}(re^{-i\theta})}\,d\theta,\quad \theta\in[0,2\pi),\,r\in(0,1),
\]
 converge weakly to $d\mu(e^{i\theta})$ as $r\uparrow1$. In particular,
the density of $\mu$ relative to arclength measure $d\theta$ on $\mathbb{T}$
is given almost everywhere by
\begin{equation}
p_{\mu}(\xi)=\frac{1}{2\pi}\Re\frac{1+\eta_{\mu}(\overline{\xi})}{1-\eta_{\mu}(\overline{\xi})},\quad \xi\in\mathbb{T},\label{eq:density from eta, circle}
\end{equation}
where 
\[
\eta_{\mu}(\xi)=\lim_{r\uparrow1}\eta_{\mu}(r\xi),\quad \xi\in\mathbb{T},
\]
exists almost everywhere as shown by Fatou \cite{Fatou}. In many
cases of interest, the function $\eta_{\mu}$ extends continuously
to $\mathbb{T}$, and thus $\mu$ is absolutely continuous on the
set $\{\xi\in\mathbb{T}:\eta_{\mu}(\overline{\xi})\ne1\}$.

The $\eta$-transform is used in the description of free multiplicative
convolution on the subset $\mathcal{P}_{\mathbb{T}}^{*}$ of $\mathcal{P}_{\mathbb{T}}$, consisting of those measures $\mu$ with the property
that $\int_{\mathbb{T}}t\,d\mu(t)\ne0$. If $\mu\in\mathcal{P}_{\mathbb{T}}^{*}$,
we have $\eta_{\mu}'(0)=\int_{\mathbb{T}}t\,d\mu(t)\ne0$, and thus $\eta_{\mu}$ has an inverse
$\eta_{\mu}^{\langle-1\rangle}$ that is a convergent power series
in a neighborhood of zero. The free multiplicative convolution of
two measures $\mu_{1},\mu_{2}\in\mathcal{P}_{\mathbb{T}}^{*}$ is
characterized by the identity (\ref{eq:defining boxtimes}) that is
now true in some neighborhood of zero. The following theorem is a
reformulation of the analytic subordination from \cite{Bi-free inc}.
\begin{thm}
\label{thm:subordination by Bi}For every $\mu_{1},\mu_{2}\in\mathcal{P}_{\mathbb{T}}^{*}$,
there exist unique $\rho_{1},\rho_{2}\in\mathcal{P}_{\mathbb{T}}^{*}$
such that
\[
z\eta_{\mu_{1}}(\eta_{\rho_{1}}(z))=z\eta_{\mu_{2}}(\eta_{\rho_{2}}(z))=\eta_{\rho_{1}}(z)\eta_{\rho_{2}}(z),\quad z\in\mathbb{D}.
\]
Moreover, we have $\eta_{\mu_{1}\boxtimes\mu_{2}}=\eta_{\mu_{1}}\circ\eta_{\rho_{1}}$. If $\mu_{1},\mu_{2}$ are nondegenerate, then so are $\rho_{1},\rho_{2}$.
\end{thm}

The concept of $\boxtimes$-infinite divisibility for measures in
$\mathcal{P}_{\mathbb{T}}$ is defined as for $\mathcal{P}_{\mathbb{R}_{+}}$.
The normalized arclength measure $m=d\theta/2\pi$ is the only $\boxtimes$-infinitely
divisible measure in $\mathcal{P}_{\mathbb{T}}\backslash\mathcal{P}_{\mathbb{T}}^{*}$.
All other $\boxtimes$-infinitely divisible measures are described
by results of \cite{vo-mul,B-V-levy}. Suppose that $\mu\in\mathcal{P}_{\mathbb{T}}^{*}$
is $\boxtimes$-infinitely divisible. Then the function $\eta_{\mu}^{\langle-1\rangle}$
has an analytic continuation $\Phi$ to $\mathbb{D}$ satisfying
\begin{equation}
\Phi(0)=0,\;|\Phi(z)|\ge|z|,\quad z\in\mathbb{D}.\label{eq:Phi on D}
\end{equation}
Conversely, every analytic function $\Phi:\mathbb{D}\to\mathbb{C}$
that satisfies (\ref{eq:Phi on D}) is the analytic continuation of
$\eta_{\mu}^{\langle-1\rangle}$ for some $\boxtimes$-infinitely
divisible measure $\mu\in\mathcal{P}_{\mathbb{T}}^{*}$. Of course,
the identity
\[
\Phi(\eta_{\mu}(z))=z
\]
extends by analytic continuation to arbitrary $z\in\mathbb{D},$ and
thus $\eta_{\mu}$ is a conformal map if $\mu$ is $\boxtimes$-infinitely
divisible. Some further information about this case is summarized
below (see \cite{B-B-IMRN}).
\begin{prop}
\label{prop:starlike stuff} Let $\mu\in\mathcal{P}_{\mathbb{T}}^{*}$
be $\boxtimes$-infinitely divisible, and let $\Phi:\mathbb{D}\to\mathbb{C}$
be the analytic continuation of $\eta_{\mu}^{\langle-1\rangle}.$
Then\emph{:}
\begin{enumerate}
\item The domain $\Omega_{\mu}=\eta_{\mu}(\mathbb{D})$ is starlike relative
to the origin.
\item The function $\eta_{\mu}$ extends to a homeomorphism of $\overline{\mathbb{D}}$
onto $\overline{\Omega_{\mu}}.$
\item We have $\Omega_{\mu}=\{z\in\mathbb{D}:|\Phi(z)|<1\}$.
\item If $|\eta_{\mu}(t)|<1$ for some $t\in\mathbb{T}$ then $\eta_{\mu}$
continues analytically to a neighborhood of $t$. 
\end{enumerate}
\end{prop}

The functions $\Phi$ that satisfy (\ref{eq:Phi on D}) can be written
as
\begin{equation}
\Phi(z)=\gamma z\exp H_{\sigma}(z),\quad z\in\mathbb{D},\label{eq:Levi-Hincin T}
\end{equation}
where $\gamma\in\mathbb{T}$ and $\sigma$ is a finite, positive Borel
measure on $\mathbb{T}$. The parameters $(\gamma,\sigma)$ are uniquely
determined by $\Phi$ (or by $\mu$) and (\ref{eq:Levi-Hincin T})
is an analog of the L\'evy-Hin\v cin formula in classical probability.
(Recall that $H_{\sigma}$ denotes the Herglotz integral of $\sigma$.)
This representation of $\Phi$, along with part (3) of the above lemma,
allow us to give an alternative description of $\eta_{\mu}|\mathbb{T}$.
We have
\[
|\Phi(r\zeta)|=r\exp\Re H_{\sigma}(r\zeta)=r\exp\left[\int_{\mathbb{T}}\frac{1-r^{2}}{|t-r\zeta|^{2}}\,d\sigma(t)\right],\quad r\in(0,1),\zeta\in\mathbb{T},
\]
and thus
\[
\log|\Phi(r\zeta)|=[1-T(r\zeta)]\log r,
\]
where 
\[
T(r\zeta)=\frac{r^{2}-1}{\log r}\int_{\mathbb{T}}\frac{d\sigma(t)}{|t-r\zeta|^{2}},\quad r\in(0,1),\zeta\in\mathbb{T}.
\]
The map $T(r\zeta)$ is an increasing, continuous function of $r$ for
fixed $\zeta$ (see \cite[Lemma 3.1]{Zhong}). We also set
\[
T(\zeta)=\lim_{r\uparrow1}T(r\zeta)=2\int_{\mathbb{T}}\frac{d\sigma(t)}{|t-\zeta|^{2}}.
\]
We conclude that $r\zeta\in\Omega_{\mu}$ precisely when $T(r\zeta)<1.$
Since $\Omega_{\mu}$ is starlike relative to $0$, we conclude that, for
each fixed $\zeta\in\mathbb{T}$, the set
\[
\{r\in(0,1):T(r\zeta)<1\}
\]
is an interval $(0,R(\zeta))$. We summarize some of the properties
of the function $R$ below. 
\begin{lem}
\label{lem:rho is cont}\cite{huang-zhong, huang-wang, Zhong} Suppose
that $\mu\in\mathcal{P}_{\mathbb{T}}^{*}$ is $\boxtimes$-infinitely
divisible. With the notation introduced above, we have\emph{:}
\begin{enumerate}
\item The function $R$ is continuous and $R(e^{i\theta})$ is continuously differentiable on $\{\theta\in \mathbb{R}:R(e^{i\theta})<1\}$.
\item $\Omega_{\mu}=\{r\zeta:\zeta\in\mathbb{T},0\le r<R(\zeta)\}$ and
$\partial\Omega_{\mu}=\{R(\zeta)\zeta:\zeta\in\mathbb{T}\}$.
\item $R(\zeta)<1$ if and only if $T(\zeta)>1$, in which case $T(R(\zeta)\zeta)=1$.
The inequality $T(R(\zeta)\zeta)\le1$ holds for every $\zeta\in\mathbb{T}$.
\end{enumerate}
\end{lem}

The following result is analogous to Lemma \ref{lem:continuity of u (half line)}.
A similar estimate could be derived from \cite[(4.20)]{B-B-IMRN}.
\begin{lem}
\label{lem:equicontinuity circle} Suppose that $\mu\in\mathcal{P}_{\mathbb{T}}^{*}$
is $\boxtimes$-infinitely divisible. With the notation introduced
above, we have 
\[
|dH_{\sigma}/dz|\le8\sigma(\mathbb{T})+2,\quad z\in\Omega_{\mu}.
\]
\end{lem}

\begin{proof}
Direct calculation yields
\begin{align*}
|dH_{\sigma}/dz| & =\left|\int_{\mathbb{T}}\frac{2t}{(t-z)^{2}}\,d\sigma(t)\right|\le2\int_{\mathbb{T}}\frac{d\sigma(t)}{|t-z|^{2}}.
\end{align*}
Since $T(z)\le1$ for $z\in\Omega_{\mu}$, we have 
\[
\int_{\mathbb{T}}\frac{d\sigma(t)}{|t-z|^{2}}\le\frac{\log|z_{j}|}{|z|^{2}-1}<\frac{1}{2|z|}\leq1,
\]
if $|z|\ge1/2$. If $|z|<1/2$, we have $|t-z|\ge1/2$ for $t\in\mathbb{T}$,
and the estimate
\[
\int_{\mathbb{T}}\frac{d\sigma(t)}{|t-z|^{2}}\le4\sigma(\mathbb{T})
\]
yields the desired result.
\end{proof}
The discussion of convolution powers in $\mathcal{P}_{\mathbb{T}}^{*}$
is best carried out for real exponents rather than just integer ones. We review this construction from \cite{B-B-IMRN} as follows.
Suppose that $\nu\in\mathcal{P}_{\mathbb{T}}^{*}$ satisfies $\int_{\mathbb{T}}t\,d\nu(t)>0$
and $\eta_{\nu}$ has no zeros in $\mathbb{D}\backslash\{0\}$. Fix
$k\in(1,+\infty)$ and set 
\[
\Phi(z)=z\left(\frac{z}{\eta_{\nu}(z)}\right)^{k-1},\quad z\in\mathbb{D}.
\]
We have 
\[
\eta'_{\nu}(0)=\int_{\mathbb{T}}t\,d\nu(t)>0,
\]
and the power above is chosen such that $\Phi'(0)>0$. The Schwarz
lemma shows that $|\Phi(z)|\ge|z|$ for $z\in\mathbb{D},$ and therefore
there exists a $\boxtimes$-infinitely divisible measure $\mu\in\mathcal{P}_{\mathbb{T}}^{*}$
such that $\Phi$ is an analytic continuation of $\eta_{\mu}^{\langle-1\rangle}$.
We can then \emph{define} the convolution power $\nu^{\boxtimes k}$
by setting
\[
\eta_{\nu^{\boxtimes k}}=\eta_{\nu}\circ\eta_{\mu}.
\]
If $k$ is an integer, the measure $\nu^{\boxtimes k}$ is in fact
equal to the free multiplicative convolution of $k$ copies of $\nu$.
The analog of (\ref{eq:eta of subordination, halflin}) also holds
in this context, but it must be written so that the powers make sense:
\[
\left(\frac{\eta_{\mu}(z)}{z}\right)^{k}=\left(\frac{\eta_{\nu^{\boxtimes k}}(z)}{z}\right)^{k-1},\quad z\in\mathbb{D};
\]
equivalently,
\[
\eta_{\nu^{\boxtimes k}}(z)=\eta_{\mu}(z)\left(\frac{\eta_{\mu}(z)}{z}\right)^{1/(k-1)},\quad z\in\mathbb{D}.
\]
As in the real case, the function $\eta_{\nu^{\boxtimes k}}$ extends
continuously to the closure $\overline{\mathbb{D}}$ \cite{B-B-IMRN}.
This construction of real powers fails if $\eta_{\nu}(z)=0$ for some
$z\in\mathbb{D}\backslash\{0\}$. Suppose however that $\int_{\mathbb{T}}t\,d\nu(t)>0$.
The $\eta$-transform of the measure $\nu^{\boxtimes2}=\nu\boxtimes\nu$
has no zeros other than $0$, and therefore one can define
\[
\nu^{\boxtimes k}=(\nu\boxtimes\nu)^{k/2}
\]
provided that $k>2$. These considerations can be carried out for
arbitrary measures in $\mathcal{P}_{\mathbb{T}}^{*}$ by choosing
an arbitrary determination of the power $(z/\eta_{\nu}(z))^{k-1}$.
If $k$ is not an integer, there may be infinitely many versions of
$\nu^{\boxtimes k}$, but each of them can be obtained from the others
by appropriate rotations. 

\section{Superconvergence in $\mathcal{P}_{\mathbb{T}}$\label{sec:Superconvergence-in T}}

The weak convergence of $\boxtimes$-infinitely divisible measures
is equivalent to certain convergence properties of the $\eta$-transforms
and of their inverses. We record the result from \cite[Proposition 2.9]{B-V-levy}.
The equivalence between (1) and (5) below is implicit in the proof
of Theorem 4.3 from \cite{BW-mult-laws}.
\begin{prop}
\cite{B-V-levy,BW-mult-laws} Suppose that $\mu$ and $\{\mu_{n}\}_{n\in\mathbb{N}}$
are nondegenerate $\boxtimes$-infinitely divisible measures in $\mathcal{P}_{\mathbb{T}}^{*}.$
Denote by $\Phi$ and $\{\Phi_{n}\}_{n\in\mathbb{N}}$ the analytic
continuations to $\mathbb{D}$ of the functions $\eta_{\mu}^{\langle-1\rangle}$
and $\{\eta_{\mu_{n}}^{\langle-1\rangle}\}_{n\in\mathbb{N}}$, and
represent these functions as in \emph{(\ref{eq:Levi-Hincin T}), }using
$(\gamma_{n},\sigma_{n})$ for the parameters corresponding to $\mu_{n}$.
The following conditions are equivalent\emph{:}
\begin{enumerate}
\item The sequence $\{\mu_{n}\}_{n\in\mathbb{N}}$ converges weakly to $\mu$.
\item The sequence $\{\eta_{\mu_{n}}\}_{n\in\mathbb{N}}$ converges pointwise
to $\eta_{\mu}$ on $\mathbb{D}$.
\item The sequence $\{\eta_{\mu_{n}}\}_{n\in\mathbb{N}}$ converges to $\eta_{\mu}$
uniformly on the compact subsets of $\mathbb{D}.$
\item The sequence $\{\Phi_{n}\}_{n\in\mathbb{N}}$ converges to $\Phi$
uniformly on the compact subsets of $\mathbb{D}$.
\item The sequence $\{\gamma_{n}\}_{n\in\mathbb{N}}$ converges to $\gamma$
and the sequence $\{\sigma_{n}\}_{n\in\mathbb{N}}$ converges weakly
to $\sigma$. 
\end{enumerate}
\end{prop}

In preparation for the analog of Lemma \ref{lem:f_n tends to f etc},
we suppose that $\mu,\mu_{n},\Phi,\Phi_{n}$ are as in the preceding
result, and we consider the continuous functions $R,R_{n}:\mathbb{T}\to(0,1]$
such that
\[
\Omega_{\mu}=\{rt:t\in\mathbb{T},r\in[0,R(t))\},\ \Omega_{\mu_{n}}=\{rt:t\in\mathbb{T},r\in[0,R_{n}(t))\}.
\]
 We also consider the homeomorphisms $h,h_{n}:\mathbb{T}\to\mathbb{T}$
defined by
\[
h(t)=\frac{\eta_{\mu}(t)}{|\eta_{\mu}(t)|},\ h_{n}(t)=\frac{\eta_{\mu_{n}}(t)}{|\eta_{\mu_{n}}(t)|},\quad t\in\mathbb{T},n\in\mathbb{N}.
\]
The existence of these (orientation preserving) homeomorphisms is
a consequence of the fact that $\Omega_{\mu_{n}}$ is starlike with
respect to $0$, and of the fact that $\eta_{\mu_{n}}$ extends to
a homeomorphism of $\overline{\mathbb{D}}$ onto $\overline{\Omega_{\mu_{n}}}$.
Observe that we have
\[
\eta_{\mu}(t)=R(h(t))h(t),\ \eta_{\mu_{n}}(t)=R_{n}(h_{n}(t))h_{n}(t)\quad t\in\mathbb{T},n\in\mathbb{N}.
\]

\begin{lem}
\label{lem:convergence of quantities for T}With the above notation,
suppose that the sequence $\{\mu_{n}\}_{n\in\mathbb{N}}$ converges
weakly to $\mu$. Then\emph{:}
\begin{enumerate}
\item The sequence $\{R_{n}\}_{n\in\mathbb{N}}$ converges to $R$ uniformly
on $\mathbb{T}$.
\item The sequence $\{h_{n}\}_{n\in\mathbb{N}}$ converges to $h$ uniformly
on $\mathbb{T}$.
\item The sequence of inverses $\{h_{n}^{\langle-1\rangle}\}_{n\in\mathbb{N}}$
converges to $h^{\langle-1\rangle}$ uniformly on $\mathbb{T}$.
\item The sequence $\{R_{n}\circ h_{n}\}_{n\in\mathbb{N}}$ converges to
$R\circ h$ uniformly on $\mathbb{T}$.
\item The sequence $\{\eta_{\mu_{n}}(t)\}_{n\in\mathbb{N}}$ converges to
$\eta_{\mu}(t)$ uniformly on $\mathbb{T}$.
\end{enumerate}
\end{lem}

\begin{proof}
(1) Since $\mathbb{T}$ is compact, it suffices to show that, for
every $t_{0}\in\mathbb{T}$ and for every $\varepsilon>0$ there exist
$N\in\mathbb{N}$ and an arc $V\subset\mathbb{T}$ containing $t_{0}$
in its interior such that 
\[
R(t)-\varepsilon<R_{n}(t)<R(t)+\varepsilon,\quad t\in V,n\ge N.
\]
Fix $t_{0}$ and $\varepsilon$ and chose a compact neighborhood $V$
of $t_{0}$ such that |$R(t)-R(t_{0})|<\varepsilon/2$ for $t\in V$.
Thus, 
\[
|\Phi((R(t_{0})-\varepsilon/2)t)|<1,\quad t\in V.
\]
 The uniform convergence of $\Phi_{n}$ to $\Phi$ on the set $\{(R(t_{0})-\varepsilon/2)t:t\in V\}$
shows that there exists $N_{1}$ such that
\[
|\Phi_{n}((R(t_{0})-\varepsilon/2)t)|<1,\quad t\in V,n\ge N_{1},
\]
 and thus 
\[
R_{n}(t)>R(t_{0})-\frac{\varepsilon}{2}>R(t)-\varepsilon,\quad t\in V,n\ge N_{1}.
\]
 If $R(t_{0})+\varepsilon/2\ge1,$ the inequality $R_{n}(t)<R(t)+\varepsilon$
is automatically satisfied for $t\in V$. If $R(t_{0})+\varepsilon/2<1,$
we observe that
\[
|\Phi((R(t_{0})+\varepsilon/2)t)|>1,\quad t\in V,
\]
and we choose $N_{2}$ such that
\[
|\Phi_{n}((R(t_{0})+\varepsilon/2)t)|>1,\quad t\in V,n\ge N_{2}.
\]
 Thus,
\[
R_{n}(t)<R(t_{0})+\frac{\varepsilon}{2}<R(t)+\varepsilon,\quad t\in V,n\ge N_{2},
\]
so it suffices to choose $N=\max\{N_{1},N_{2}\}$.

(3) It suffices to prove pointwise convergence. We observe that $h_{n}^{\langle-1\rangle}(t)=\Phi_{n}(R_{n}(t)t)$.
Since the measures $\sigma_{n}$ converge weakly, the sequence $\{\sigma_{n}(\mathbb{T})\}_{n\in\mathbb{N}}$
is bounded. Lemma \ref{lem:equicontinuity circle} shows that the
restrictions $\Phi_{n}|\overline{\Omega_{\mu_{n}}}$ are equicontinuous.
These facts, along with (1), imply the desired pointwise convergence. 

(2) This follows directly from (3). Then (4) and (5) follow as in
the proof of Lemma \ref{lem:f_n tends to f etc}. 
\end{proof}
As in the case of $\mathbb{R}_{+}$, the $\eta$-transform of an $\boxtimes$-infinitely
divisible measure $\mu\in\mathcal{P}_{\mathbb{T}}^{*}$ may take the
value $1$ at most once on $\mathbb{T}.$ If $\eta_{\mu}(t)=1,$ we
write $D_{\mu}=\{\overline{t}\},$ otherwise $D_{\mu}=\varnothing$.
The measure $\mu$ is absolutely continuous relative to arclength
measure on $\mathbb{T}\backslash D_{\mu}$.

We can now use the preceding result and (\ref{eq:density from eta, circle})
to prove the analog of Proposition \ref{prop:unif convergence of inf div densities pos line}
for the circle. The details are left to the interested reader.
\begin{prop}
\label{prop:uniform conv circle}Let $\mu$ and $\{\mu_{n}\}_{n\in\mathbb{N}}$
be $\boxtimes$-infinitely divisible measures in $\mathcal{P}_{\mathbb{T}}^{\mathbb{*}}$
such that $\mu_{n}$ converges weakly to $\mu$. Let $K\subset\mathbb{T}\backslash D_{\mu}$
be an arbitrary compact set. Then $D_{\mu_{n}}\subset\mathbb{T}\backslash K$
for sufficiently large $n$, and the densities $p_{\mu_{n}}$ of $\mu_{n}$
relative to arclength measure converge to $p_{\mu}$ uniformly on
$K$. If $D_{\mu}=\varnothing$, we can take $K=\mathbb{T}$. 
\end{prop}

Finally, we derive a superconvergence result.
\begin{thm}
\label{thm:super T}Let $\{k_{n}\}_{n\in\mathbb{N}}\subset[2,+\infty)$
be a sequence with limit $+\infty$, and let $\mu$ and $\{\nu_{n}\}_{n\in\mathbb{N}}$
be measures in $\mathcal{P}_{\mathbb{T}}^{*}$ such that $\mu$ is
$\boxtimes$-infinitely divisible and $\int_{\mathbb{T}}t\,d\nu_{n}(t)>0$
for every $n\in\mathbb{N}$. Suppose that the sequence $\{\nu_{n}^{\boxtimes k_{n}}\}_{n\in\mathbb{N}}$
converges weakly to $\mu$. Let $K\subset\mathbb{T}\backslash D_{\mu}$
be an arbitrary compact set. Then $\nu_{n}^{\boxtimes k_{n}}$ is
absolutely continuous on $K$ for sufficiently large $n$, and the
densities $p_{n}$ of $\nu_{n}^{\boxtimes k_{n}}$ relative to arclength
measure converge to $p_{\mu}$ uniformly on $K$. If $D_{\mu}=\varnothing$,
we can take $K=\mathbb{T}$.
\end{thm}

\begin{proof}
We first replace $\nu_{n}$ by $\nu_{n}\boxtimes\nu_{n}$ and $k_{n}$
by $k_{n}/2$. After this substitution we may assume that $\eta_{\nu_{n}}$
does not vanish on $\mathbb{D}$ and the convolution powers can be
calculated as in Section \ref{sec:Free-mutiplicative-convolution on T},
using analytic subordination. Thus, there exist $\boxtimes$-infinitely
divisible measures $\mu_{n}\in\mathcal{P}_{\mathbb{T}}^{*}$ satisfying
the equations
\begin{equation}
\eta_{\nu_{n}^{\boxtimes k_{n}}}(z)=\eta_{\mu_{n}}(z)\left(\frac{\eta_{\mu_{n}}(z)}{z}\right)^{1/(k_{n}-1)},\quad z\in\mathbb{\overline{D}},\;n\in\mathbb{N},\label{eq:formula}
\end{equation}
and
\[
\eta_{\nu_{n}^{\boxtimes k_{n}}}=\eta_{\nu_{n}}\circ\eta_{\mu_{n}},\quad n\in\mathbb{N}.
\]
 As in the case of $\mathbb{R}_{+}$, the measures $\nu_{n}$ necessarily
converge to $\delta_{1}$ as $n\to\infty,$ and thus $\eta_{\nu_{n}}(z)$
converges to $z$ uniformly for $z$ in a compact subset of $\mathbb{D}.$
The inverses $\eta_{\nu_{n}}^{\langle-1\rangle}$ converge uniformly
to the identity function for $z$ in a neighborhood of $0$, and therefore
\[
\eta_{\mu_{n}}=\eta_{\nu_{n}}^{\langle-1\rangle}\circ\eta_{\nu_{n}^{\boxtimes k_{n}}}
\]
converge uniformly on a neighborhood of $0$ to $\eta_{\mu}.$ We
conclude that the sequence $\{\mu_{n}\}_{n\in\mathbb{N}}$ converges
weakly to $\mu$. Lemma \ref{lem:convergence of quantities for T}
implies now that the functions $\eta_{\mu_{n}}$ converge to $\eta_{\mu}$
uniformly on $\overline{\mathbb{D}},$ and therefore the functions
\[
\left(\frac{\eta_{\mu_{n}}(z)}{z}\right)^{1/(k_{n}-1)},\quad z\in\overline{\mathbb{D}},
\]
converge uniformly to $1$. Formula (\ref{eq:formula}) implies now
that the sequence $\{\eta_{\nu_{n}^{\boxtimes k_{n}}}\}_{n\in\mathbb{N}}$
converges to $\eta_{\mu}$ uniformly on $\overline{\mathbb{D}}.$
The desired conclusion is now obtained easily by applying (\ref{eq:density from eta, circle})
to these measures.
\end{proof}

\section{Cusp behavior in $\mathcal{P}_{\mathbb{T}}$\label{sec:Cusp-behavior-in T}}

This section is the counterpart of Section \ref{sec:cusps-in R_+}
for $\mathbb{T}$. Thus, we consider the qualitative behavior of a
convolution $\mu_{1}\boxtimes\mu_{2}$, where $\mu_{1},\mu_{2}\in\mathcal{P}_{\mathbb{T}}^{*}$
are nondegenerate measures and $\mu_{2}$ is $\boxtimes$-infinitely
divisible. Of course, all $\boxtimes$-infinitely divisible measures
in $\mathcal{P}_{\mathbb{T}}$ belong to $\mathcal{P}_{\mathbb{T}}^{*}$,
with the exception of the normalized arclength measure $m$. For this
measure, we have $\mu\boxtimes m=m$, $\mu\in\mathcal{P}_{\mathbb{T}}$,
so $m$ is the analog of the measure $\delta_{0}\in\mathcal{P}_{\mathbb{R}_{+}}$,
and indeed it has the same moment sequence.

We start with the analog of Lemma \ref{lem:omega is inf-div}.
\begin{lem}
\label{lem:omega is inf-div-1} Let $\mu_{1},\mu_{2}\in\mathcal{P}_{\mathbb{T}}^{*}$
be such that $\mu_{2}$ is $\boxtimes$-infinitely divisible, and
let $\rho_{1},\rho_{2}\in\mathcal{P}_{\mathbb{T}}^{*}$ be given by
\textup{Theorem} \emph{\ref{thm:subordination by Bi}}. Then $\rho_{1}$
is $\boxtimes$-infinitely divisible. 
\end{lem}

\begin{proof}
Let $\Phi$ given by (\ref{eq:Levi-Hincin T}) be the analytic continuation
of $\eta_{\mu_{2}}^{\langle-1\rangle}$ to $\mathbb{D}$. Thus,
\[
\Phi(z)=zF(z),\quad z\in\mathbb{D},
\]
where $F$ satisfies $|F(z)|\ge1$ for $z\in\mathbb{D}.$ Then the
analog of (\ref{eq:defining boxtimes}) for $\mathcal{P}_{\mathbb{T}}^{*}$
can be written as 
\[
F(z)\eta_{\mu_{1}}^{-1}(z)=\eta_{\mu_{1}\boxtimes\mu_{2}}^{\langle-1\rangle}(z),
\]
and applying this equality with $\eta_{\mu_{1}}(z)$ in place of $z$,
we obtain
\[
F(\eta_{\mu_{1}}(z))z=\eta_{\mu_{1}\boxtimes\mu_{2}}^{\langle-1\rangle}(\eta_{\mu_{1}}(z))=\eta_{\rho_{1}}^{\langle-1\rangle}(z)
\]
for $z$ in some neighborhood of zero. The lemma follows because the
function
\[
G(z)=F(\eta_{\mu_{1}}(z)),\quad z\in\mathbb{D},
\]
also satisfies the inequality $|G(z)|\ge1$ for $z\in\mathbb{D}$.
\end{proof}
With the notation of the preceding lemma, we recall that the domain
\[
\Omega_{\rho_{1}}=\eta_{\rho_{1}}(\mathbb{D})
\]
 can be described as
\[
\Omega_{\rho_{1}}=\{rt:t\in\mathbb{T},0\le r<R(t)\}
\]
for some continuous function $R:\mathbb{T}\to(0,1]$, and that $\eta_{\rho_{1}}$
extends to a homeomorphism of $\overline{\mathbb{D}}$ onto $\overline{\Omega_{\rho_{1}}}$.
Using, the analytic continuation
\[
\Psi(z)=zG(z),\quad z\in\mathbb{D},
\]
of $\eta_{\rho_{1}}^{\langle-1\rangle}$, the map
\[
\Psi|\partial\Omega_{\rho_{1}},\quad \partial\Omega_{\rho_{1}}=\{R(t)t:t\in\mathbb{T}\},
\]
is a homeomorphism from $\partial\Omega_{\rho_{1}}$ onto $\mathbb{T}$.
The density $p_{\mu_{1}\boxtimes\mu_{2}}$ of $\mu_{1}\boxtimes\mu_{2}$,
relative to arclength measure $2\pi dm$ on $\mathbb{T}$, is calculated
using the formula
\begin{equation}
p_{\mu_{1}\boxtimes\mu_{2}}(\xi)=\begin{cases}
\frac{1}{2\pi}\Re\frac{1+\eta_{\mu_{1}}(R(t)t)}{1-\eta_{\mu_{1}}(R(t)t)}, & \text{if \ensuremath{\xi=\frac{1}{\Psi(R(t)t)}\text{ and }}}R(t)<1,\\
0, & \text{if \ensuremath{\xi=\frac{1}{\Psi(R(t)t)}\text{ and }}}R(t)=1.
\end{cases}\label{eq:density boxtimes T}
\end{equation}
As noted earlier, this density is real analytic at all points where
it is nonzero. Using the Herglotz formula for analytic functions with
a positive real part, we write the function $F$ above as
\[
F(z)=\gamma\exp(H_{\sigma}(z)),\quad z\in\mathbb{D},
\]
where $|\gamma|=1$ and $\sigma$ is a finite, positive Borel measure
on $\mathbb{T}$. The appropriate analog of the semicircular measure
is obtained when $\sigma$ is a point mass at $1\in\mathbb{T},$ that
is,
\[
F(z)=\gamma\exp\left[\beta\frac{1+z}{1-z}\right],\quad z\in\mathbb{D},
\]
for some $\gamma\in\mathbb{T}$ and $\beta>0$. The following proposition
examines the density of $\mu_{1}\boxtimes\mu_{2}$ when $\mu_{2}$
is one of these measures. (The formula (\ref{eq:p vs R}) below also
appeared in \cite{Zhong}.) We use the notation $p'$ for the derivative
$dp(e^{i\theta})/d\theta$ if $p$ is a differentiable function defined
on some open subset of $\mathbb{T}$. 
\begin{prop}
\label{prop:boxtimes convo with semisemi-1} Suppose that $\mu_{1},\mu_{2}\in\mathcal{P}_{\mathbb{T}}^{*}$
are nondegenerate measures, and that $\mu_{2}$ is such that
\[
\gamma z\exp\left[\beta\frac{1+z}{1-z}\right],\quad z\in\mathbb{D},
\]
is an analytic continuation of $\eta_{\mu_{2}}^{\langle-1\rangle}$
for some $\gamma\in\mathbb{T}$ and $\beta\in(0,+\infty)$. Let $p_{\mu_{1}\boxtimes\mu_{2}}$
denote the density of $\mu_{1}\boxtimes\mu_{2}$ relative to the arclength
measure $2\pi\,dm$. Then\emph{:}
\begin{enumerate}
\item $\left|p_{\mu_{1}\boxtimes\mu_{2}}'(\xi)\right|p_{\mu_{1}\boxtimes\mu_{2}}(\xi)^{2}\le 7/(8\pi^{3}\beta^{3})$
for every $\xi\in\mathbb{T}$ such that $0<p_{\mu_{1}\boxtimes\mu_{2}}(\xi)\le \log2/(2\pi\beta)$.
\item $\left|p_{\mu_{1}\boxtimes\mu_{2}}'(\xi)\right|\le7/\pi\beta$ for
every $\xi\in\mathbb{T}$ such that $p_{\mu_{1}\boxtimes\mu_{2}}(\xi)\ge\log2/(2\pi\beta)$.
\item If $I\subset\mathbb{T}$ is an arc with one endpoint $\xi_{0}$, $p_{\mu_{1}\boxtimes\mu_{2}}(\xi)>0$
for $\xi\in I$, and $p_{\mu_{1}\boxtimes\mu_{2}}(\xi_{0})=0$, then
\[
p_{\mu_{1}\boxtimes\mu_{2}}(\xi)\le\frac{2}{\pi\beta}|\xi-\xi_{0}|^{1/3}
\]
 for $\xi\in I$ close to $\xi_{0}$.
\end{enumerate}
\end{prop}

\begin{proof}
Part (3) follows from (1) by integration since $\sqrt[3]{21/4}<2$
and $\ell(\xi,\xi_{0})<2|\xi-\xi_{0}|$ if $\xi$ is close to $\xi_{0}$;
here $\ell(\xi,\xi_{0})$ denotes the length of the (short) arc joining
$\xi$ and $\xi_{0}$.

As seen in the preceding lemma, $\eta_{\rho_{1}}^{\langle-1\rangle}$
has the analytic continuation
\[
\Psi(z)=zG(z)=F(\eta_{\mu_{1}}(z))=\gamma z\exp\left[\beta u(z)\right],\quad z\in\mathbb{D},
\]
where
\begin{align*}
u(z)=\frac{1+\eta_{\mu_{1}}(z)}{1-\eta_{\mu_{1}}(z)} & =1+2\psi_{\mu_{1}}(z)\\
 & =\int_{\mathbb{T}}\left[1+\frac{2\xi z}{1-\xi z}\right]\,d\mu_{1}(\xi)\\
 & =\int_{\mathbb{T}}\left[\frac{1+\xi z}{1-\xi z}\right]\,d\mu_{1}(\xi)\\
 & =\int_{\mathbb{T}}\left[\frac{\xi+z}{\xi-z}\right]\,d\mu_{1}(1/\xi)
\end{align*}
is precisely the Herglotz integral of the measure $d\mu_{1}(1/\xi)=d\mu_{1}(\overline{\xi})$.
Thus, when the boundary of the domain $\Omega_{\rho_{1}}$ is parametrized
as $z(t)=R(t)t$, $t\in\mathbb{T},$ we have
\begin{equation}
\beta\int_{\mathbb{T}}\frac{d\mu_{1}(\overline{\xi})}{|\xi-z(t)|^{2}}=\frac{\log R(t)}{R(t)^{2}-1}\label{eq:boundary of omega (circle)}
\end{equation}
 whenever $R(t)<1$. (We will use implicitly the easily established
inequality
\[
\frac{2r\log r}{r^{2}-1}<1,
\]
valid for $r\in(0,1)$. In fact the function
\[
\frac{2r\log r}{r^{2}-1}
\]
is increasing for $r\in(0,1)$ and it tends to $1$ at $r=1$.) Setting
\[
f(t)=\frac{1}{\Psi(z(t))},
\]
for $R(t)<1$, we see from (\ref{eq:density boxtimes T}) and (\ref{eq:boundary of omega (circle)})
that
\begin{align}
p_{\mu_{1}\boxtimes\mu_{2}}(f(t)) & =\frac{1}{2\pi}\Re\frac{1+\eta_{\mu_{1}}(z(t))}{1-\eta_{\mu_{1}}(z(t))}\nonumber \\
 & =\frac{1}{2\pi}\int_{\mathbb{T}}\frac{1-|z(t)|^{2}}{|\xi-z(t)|^{2}}\,d\mu_{1}(\overline{\xi})\label{eq:p vs R}\\
 & =\frac{1}{2\pi\beta}\beta\int_{\mathbb{T}}\frac{1-|z(t)|^{2}}{|\xi-z(t)|^{2}}\,d\mu_{1}(\overline{\xi})=\frac{-\log R(t)}{2\pi\beta}.\nonumber 
\end{align}
As in the case of $\mathbb{R}_{+}$, this allows us to use the chain
rule for our estimates. We begin with the derivative of $f$ that
can be estimated as
\begin{align*}
|f'(t)| & =\left|\frac{f'(t)}{f(t)}\right|=\left|\frac{\Psi'(z(t))}{\Psi(z(t))}\right||z'(t)|.
\end{align*}
Here, $\Psi'$ is the usual complex derivative of $\Psi$,
\[
\left|\frac{\Psi'(z)}{\Psi(z)}\right|=\left|\frac{1}{z}+\beta u'(z)\right|=\left|\frac{1}{z}+\beta\int_{\mathbb{T}}\frac{2\xi}{(\xi-z)^{2}}d\mu_{1}(\overline{\xi})\right|,
\]
so using (\ref{eq:boundary of omega (circle)}) we obtain
\begin{align*}
\left|\frac{\Psi'(z(t))}{\Psi(z(t))}\right| & =\left|\frac{1}{R(t)t}+\beta\int_{\mathbb{T}}\frac{2\xi}{(\xi-R(t)t)^{2}}d\mu_{1}(\overline{\xi})\right|\\
 & =\frac{1}{R(t)}\left|1+\beta\int_{\mathbb{T}}\frac{2\xi R(t)t}{(\xi-R(t)t)^{2}}d\mu_{1}(\overline{\xi})\right|\\
 & \ge\frac{1}{R(t)}\left[1-\frac{2R(t)\log R(t)}{R(t)^{2}-1}\right].
\end{align*}
For the second factor $|z'(t)|$, we have
\[
z'(e^{i\theta})=\frac{d}{d\theta}R(e^{i\theta})e^{i\theta}=[R'(e^{i\theta})+iR(e^{i\theta})]e^{i\theta},
\]
and thus
\[
|z'(t)|=\sqrt{R(t)^{2}+R'(t)^{2}.}
\]
Putting these together, we see that
\begin{align*}
|f'(t)| & \ge\sqrt{1+\left(\frac{R'(t)}{R(t)}\right)^{2}}\left[1-\frac{2R(t)\log R(t)}{R(t)^{2}-1}\right]\\
 & \ge\left|\frac{R'(t)}{R(t)}\right|\left[1-\frac{2R(t)\log R(t)}{R(t)^{2}-1}\right].
\end{align*}
Since 
\[
\left|\frac{R'(t)}{R(t)}\right|=|(\log R)'(t)|,
\]
formula (\ref{eq:p vs R}) yields the estimate
\begin{align*}
\left|p_{\mu_{1}\boxtimes\mu_{2}}'(f(t))\right| & =\frac{|(\log R)'(t)|}{2\pi\beta|f'(t)|}\\
 & \le\frac{1}{2\pi\beta}\frac{1}{1-\frac{2R(t)\log R(t)}{R(t)^{2}-1}}.
\end{align*}
The inequality $p_{\mu_{1}\boxtimes\mu_{2}}(f(t))>(\log2)/2\pi\beta$
amounts to $R(t)<1/2$, and the preceding estimate yields
\[
\left|p_{\mu_{1}\boxtimes\mu_{2}}'(f(t))\right|\le\frac{1}{2\pi\beta}\frac{1}{1-\frac{4}{3}\log2}<\frac{7}{\pi\beta},
\]
thus verifying (2). Finally, we have
\[
\left|p_{\mu_{1}\boxtimes\mu_{2}}'(f(t))\right|p_{\mu_{1}\boxtimes\mu_{2}}(f(t))^{2}\le\frac{1}{(2\pi\beta)^{3}}\frac{\log^{2}R(t)}{1-\frac{2R(t)\log R(t)}{R(t)^{2}-1}}
\]
and the fact that
\[
\frac{\log^{2}r}{1-\frac{2r\log r}{r^{2}-1}}
\]
 is less than $7$ for $r\in(1/2,1)$ yields (1).
\end{proof}
Next, we state an analog of Lemma \ref{lem:trade a free convolution for another}.
The verification is a simple calculation.
\begin{lem}
\label{lem:trade a free convolution for another-1}Let $\mu_{1},\mu_{2}\in\mathcal{P}_{\mathbb{T}}^{*}$
be two nondegenerate measures such that $\mu_{2}$ is $\boxtimes$-infinitely
divisible, and let
\[
\Phi(z)=\gamma z\exp H_{\sigma}(z),\quad z\in\mathbb{D},
\]
be an analytic continuation of $\eta_{\mu_{2}}^{\langle-1\rangle}$.
Assume that $\int_{\mathbb{T}}t\,d\sigma(t)\neq0$. Denote by $\rho_{1}\in\mathcal{P}_{\mathbb{T}}^{*}$
the $\boxtimes$-infinitely divisible measure such that $\eta_{\rho_{1}}^{\langle-1\rangle}$
has the analytic continuation
\[
\Psi(z)=\gamma z\exp\left[\int_{\mathbb{T}}\frac{t+\eta_{\mu_{1}}(z)}{t-\eta_{\mu_{1}}(z)}\,d\sigma(t)\right],\quad z\in\mathbb{D}.
\]
Set $\beta=\sigma(\mathbb{T})$ and denote by $\nu_{1}\in\mathcal{P}_{\mathbb{T}}^{*}$
the measure satisfying
\[
\psi_{\nu_{1}}(z)=\frac{1}{\beta}\int_{\mathbb{T}}\frac{t\eta_{\mu_{1}}(z)}{1-t\eta_{\mu_{1}}(z)}\,d\sigma(\overline{t}),\quad z\in\mathbb{D},
\]
and denote by $\nu_{2}\in\mathcal{P}_{\mathbb{T}}^{*}$ the $\boxtimes$-infinitely
divisible measure such that $\eta_{\nu_{2}}^{\langle-1\rangle}$ has
the analytic continuation
\[
\gamma z\exp\left[\beta\frac{1+z}{1-z}\right],\quad z\in\mathbb{D}.
\]
Then $\eta_{\mu_{1}\boxtimes\mu_{2}}=\eta_{\mu_{1}}\circ\eta_{\rho_{1}}$
and $\eta_{\nu_{1}\boxtimes\nu_{2}}=\eta_{\nu_{1}}\circ\eta_{\rho_{1}}$.
\end{lem}

In preparation for the final proof in this section, we recall some
facts demonstrated in \cite[Theorem 4.5]{huang-wang} whose proof is based on \cite[Proposition 4.5]{B-B-IMRN} and the chain rule for Julia-Carath\'{e}odory derivative. Suppose that
$\mu_{1},\mu_{2}$, and $\rho_{1}$ are as in the preceding lemma,
and that the domain $\Omega_{\rho_{1}}$ is described as
\[
\Omega_{\rho_{1}}=\{rt:t\in\mathbb{T},0\le r<R(t)\}
\]
for some continuous function $R:\mathbb{T}\to(0,1]$. The map $\eta_{\mu_{1}}$
extends continuously to the closure $\overline{\Omega_{\rho_{1}}}$,
and the set 
\[
\partial\Omega_{\rho_{1}}\cap\mathbb{T}=\{t\in\mathbb{T}:R(t)=1\}
\]
can be partitioned into two subsets $A$ and $B$ described as follows.
\begin{enumerate}
\item $A$ consists of those points $t\in\mathbb{T}$ for which $\mu_{1}(\{\overline{t}\})>0$
and
\[
\frac{\mu_{1}(\{\overline{t}\})}{2}\ge\int_{\mathbb{T}}\frac{d\sigma(\xi)}{|1-\xi|^{2}}.
\]
\item $B$ consists of those $t\in\mathbb{T}$ for which $\eta_{\mu_{1}}(t)\in\mathbb{T}\backslash\{1\},$
\[
c=\liminf_{z\to t}\frac{1-|\eta_{\mu_{1}}(z)|}{1-|z|}\in (0,+\infty),
\]
and
\[
c\int_{\mathbb{T}}\frac{d\sigma(\xi)}{|\eta_{\mu_{1}}(t)-\xi|^{2}}\le \frac{1}{2}.
\]
\end{enumerate}
\begin{thm}
\label{thm:cusp on T}Let $\mu_{1},\mu_{2}\in\mathcal{P}_{\mathbb{T}}^{*}$
be two nondegenerate measures such that $\mu_{2}$ is $\boxtimes$-infinitely
divisible and satisfies the hypothesis of \textup{Lemma}\emph{ \ref{lem:trade a free convolution for another-1}}.
Suppose that $\Gamma\subset\mathbb{T}$ is an open arc with an endpoint
$\xi_{0}$, $p_{\mu_{1}\boxtimes\mu_{2}}(\xi_{0})=0<p_{\mu_{1}\boxtimes\mu_{2}}(\xi)$
for every $\xi\in\Gamma$, and---using the notation of \textup{Lemma}\emph{
\ref{lem:trade a free convolution for another-1}}---$1/\eta_{\rho}{}_{_{1}}\left(\overline{\xi_{0}}\right)$
is not an atom of $\mu_{1}.$ Then $p_{\mu_{1}\boxtimes\mu_{2}}(\xi)/|\xi-\xi_{0}|^{1/3}$
is bounded for $\xi\in\Gamma$ close to $\xi_{0}$.
\end{thm}

\begin{proof}
Using the notation of Lemma \ref{lem:trade a free convolution for another-1},
we observe that
\begin{align*}
\{\xi\in\mathbb{T}:p_{\mu_{1}\boxtimes\mu_{2}}(\xi)>0\} & =\{\overline{\Psi(R(t)t)}:R(t)<1\}\\
 & =\{\xi\in\mathbb{T}:p_{\nu_{1}\boxtimes\nu_{2}}(\xi)>0\}.
\end{align*}
By Proposition \ref{prop:boxtimes convo with semisemi-1}, the conclusion
of the theorem is true if $\mu_{1}$ and $\mu_{2}$ are replaced by
$\nu_{1}$ and $\nu_{2}$, respectively. It will therefore suffice
to prove that the ratio $p_{\nu_{1}\boxtimes\nu_{2}}(\xi)/p_{\mu_{1}\boxtimes\mu_{2}}(\xi)$
is bounded away from zero for $\xi$ close to $\xi_{0}$. The hypothesis
implies that the number $\alpha=1/\eta_{\rho}{}_{_{1}}(\overline{\xi_{0}})$
belongs to the set $B$ described before the statement of the theorem.
Using the usual parametrization $z=\eta_{\rho_{1}}(\overline{\xi})$, the relations
$\eta_{\mu_{1}\boxtimes\mu_{2}}=\eta_{\mu_{1}}\circ\eta_{\rho_{1}}$
and $\eta_{\nu_{1}\boxtimes\nu_{2}}=\eta_{\nu_{1}}\circ\eta_{\rho_{1}}$
yield
\[
\gamma z\exp\left[\int_{\mathbb{T}}\frac{t+\eta_{\mu_{1}}(z)}{t-\eta_{\mu_{1}}(z)}\,d\sigma(t)\right]=\Psi(z)=\gamma z\exp\left[\beta\frac{1+\eta_{\nu_{1}}(z)}{1-\eta_{\nu_{1}}(z)}\right].
\]
Equating the absolute values of these quantities yields
\[
\int_{\mathbb{T}}\frac{1-|\eta_{\mu_{1}}(z)|^{2}}{|t-\eta_{\mu_{1}}(z)|^{2}}\,d\sigma(t)=\beta\frac{1-|\eta_{\nu_{1}}(z)|^{2}}{|1-\eta_{\nu_{1}}(z)|^{2}},
\]
or, equivalently
\[
\left[\Re\frac{1+\eta_{\mu_{1}}(z)}{1-\eta_{\mu_{1}}(z)}\right]\int_{\mathbb{T}}\frac{|1-\eta_{\mu_{1}}(z)|^{2}}{|t-\eta_{\mu_{1}}(z)|^{2}}\,d\sigma(t)=\beta\Re\frac{1+\eta_{\nu_{1}}(z)}{1-\eta_{\nu_{1}}(z)}.
\]
Applying (\ref{eq:density boxtimes T}) we rewrite this as
\[
\frac{p_{\nu_{1}\boxtimes\nu_{2}}(\xi)}{p_{\mu_{1}\boxtimes\mu_{2}}(\xi)}=\frac{|1-\eta_{\mu_{1}}(z)|^{2}}{\beta}\int_{\mathbb{T}}\frac{d\sigma(t)}{|t-\eta_{\mu_{1}}(z)|^{2}},\quad\xi\in\Gamma.
\]
 The desired result follows now from the definition of the set $B$
and an application of Fatou's lemma.
\end{proof}
\begin{rem}
With the notation of the preceding proof, we have $|\Psi(z)|=1$ for
the relevant points $z,$ implying further that 
\[
\int_{\mathbb{T}}\frac{|1-\eta_{\mu_{1}}(z)|^{2}}{|t-\eta_{\mu_{1}}(z)|^{2}}\,d\sigma(t)=\frac{|1-\eta_{\mu_{1}}(z)|^{2}\log|z|}{|\eta_{\mu_{1}}(z)|^{2}-1}.
\]
It follows that
\[
\frac{p_{\mu_{1}\boxtimes\mu_{2}}(\xi)}{p_{\nu_{1}\boxtimes\nu_{2}}(\xi)}=\beta\frac{|\eta_{\mu_{1}}(z)|^{2}-1}{|1-\eta_{\mu_{1}}(z)|^{2}\log|z|},
\]
and it is easily seen that this ratio is also bounded away from zero
near $\xi_{0}$.
\end{rem}

\section{Free additive convolution on $\mathcal{P}_{\mathbb{R}}$\label{sec:Free-additive-convolution}}

The free additive convolution $\boxplus$ is a binary operation defined
on $\mathcal{P}_{\mathbb{R}}$, the family of all probability measures
on $\mathbb{R}.$ The Cauchy transform of a measure $\mu\in\mathcal{P}_{\mathbb{R}}$,
already seen in Section \ref{sec:Free-multiplicative-convolution},
is defined by
\[
G_{\mu}(z)=\int_{\mathbb{R}}\frac{d\mu(t)}{z-t},\quad z\in\mathbb{H},
\]
and the density $d\mu/dt$ of $\mu$ is equal almost everywhere to
$(-1/\pi)\Im G_{\mu}(x)$, where the boundary limit
\[
G_{\mu}(x)=\lim_{y\downarrow0}G_{\mu}(x+iy),\quad x\in\mathbb{R},
\]
exists almost everywhere on $\mathbb{R}$. The \emph{reciprocal} Cauchy
transform
\[
F_{\mu}(z)=\frac{1}{G_{\mu}(z)},\quad z\in\mathbb{H},
\]
maps $\mathbb{H}$ to itself, and the collection \{$F_{\mu}:\mu\in\mathcal{P}_{\mathbb{R}}\}$
consists precisely of those analytic functions $F:\mathbb{H}\to\mathbb{H}$
with the property that
\[
\lim_{y\uparrow\infty}\frac{F(iy)}{iy}=1.
\]
As seen, for instance, in \cite{Akhiezer}, these functions have a
Nevanlinna representation of the form
\[
F(z)=\gamma+z-N_{\sigma}(z),\quad z\in\mathbb{H},
\]
where $\gamma\in\mathbb{R}$ and
\[
N_{\sigma}(z)=\int_{\mathbb{R}}\frac{1+tz}{z-t}\,d\sigma(t)
\]
for some finite positive Borel measure $\sigma$ on $\mathbb{R}$.
This integral representation implies that 
\[
\Im F(z)\geq\Im z,\quad z\in\mathbb{H}.
\]

Given a measure $\mu\in\mathcal{P}_{\mathbb{R}}$, the function $F_{\mu}$
is conformal in an open set $U$ containing $\{iy:y\in(\alpha,+\infty)\}$
for some $\alpha>0$, and the restriction $F_{\mu}|U$ has an inverse
$F_{\mu}^{\langle-1\rangle}$ defined in an open set containing another
set of the form $\{iy:y\in(\beta,+\infty)\}$ with $\beta>0$. The
free additive convolution $\mu_{1}\boxplus\mu_{2}$ of two measures
$\mu_{1},\mu_{2}\in\mathcal{P}_{\mathbb{R}}$ is the unique measure
$\mu\in\mathcal{P}_{\mathbb{R}}$ that satisfies the identity
\begin{equation}
z+F_{\mu}^{\langle-1\rangle}(z)=F_{\mu_{1}}^{\langle-1\rangle}(z)+F_{\mu_{2}}^{\langle-1\rangle}(z)\label{eq:defining boxtimes-2}
\end{equation}
for $z$ in some open set containing $iy$ for $y$ large enough (see
\cite{BV-unbounded}). The analog of Theorems \ref{thm:subordination on the line (mult)}
and \ref{thm:subordination by Bi} is as follows.
\begin{thm}\cite{Bi-free inc}
\label{thm:subordination on the line (mult)-2} For every $\mu_{1},\mu_{2}\in\mathcal{P}_{\mathbb{R}},$
there exist unique $\rho_{1},\rho_{2}\in\mathcal{P}_{\mathbb{R}}$
such that
\[
F_{\mu_{1}}(F_{\rho_{1}}(z))=F_{\mu_{2}}(F_{\rho_{2}}(z))=F_{\rho_{1}}(z)+F_{\rho_{2}}(z)-z,\quad z\in\mathbb{H}.
\]
Moreover, we have $F_{\mu_{1}\boxplus\mu_{2}}=F_{\mu_{1}}\circ F_{\rho_{1}}$.
If $\mu_{1}$ and $\mu_{2}$ are nondegenerate, then so are $\rho_{1}$
and $\rho_{2}$. 
\end{thm}

It was shown in \cite{vo-add,BV-unbounded} that a measure $\mu\in\mathcal{P}_{\mathbb{R}}$
is $\boxplus$-infinitely divisible precisely when the inverse $F_{\mu}^{\langle-1\rangle}$
continues analytically to $\mathbb{H}$ and this analytic continuation
has the Nevanlinna form
\begin{equation}
\Phi(z)=\gamma+z+N_{\sigma},\quad z\in\mathbb{H},\label{eq:extension of eta inverse (line)-2}
\end{equation}
 for some $\gamma$ and $\sigma$. The functions described by (\ref{eq:extension of eta inverse (line)-2})
can also be characterized by
\[
\lim_{y\uparrow+\infty}\frac{\Phi(iy)}{iy}=1\text{ and }\Im\Phi(z)\le\Im z,\quad z\in\mathbb{H}.
\]

Suppose now that $\mu\in\mathcal{P}_{\mathbb{R}}$ is $\boxplus$-infinitely
divisible and that $F_{\mu}^{\langle-1\rangle}$ has the analytic
continuation given in (\ref{eq:extension of eta inverse (line)-2}).
The equation $\Phi(F_{\mu}(z))=z$ holds in some open set and therefore
it holds on the entire $\mathbb{\mathbb{H}}$ by analytic continuation.
In particular, $F_{\mu}$ maps $\mathbb{H}$ conformally onto a domain
$\Omega_{\mu}\subset\mathbb{H}$ that can be described as
\[
\Omega_{\mu}=\{z\in\mathbb{H}:\Phi(z)\in\mathbb{H}\}.
\]
As in the multiplicative cases, this domain can also be identified
with $\{x+iy:y>f(x)\}$ for some continuous function $f:\mathbb{R}\to[0,+\infty)$.
The map $F_{\mu}$ extends continuously to the closure $\overline{\mathbb{H}}$,
$\Phi$ extends continuously to $\overline{\Omega_{\mu}}$, and these
two extensions are homeomorphisms, inverse to each other. (See Section 2 of \cite{BWZ-super+} for the details and \cite{huang} for similar results in the context of free semigroups.)

\section{Cusp behavior in $\mathcal{P}_{\mathbb{R}}$ \label{sec:Cusp-behavior-in R}}

We are now ready for the counterpart of Sections \ref{sec:cusps-in R_+}
and \ref{sec:Cusp-behavior-in T} in the context of the free additive
convolution. Thus, we study the density of a measure of the form $\mu_{1}\boxplus\mu_{2}$,
where $\mu_{1},\mu_{2}\in\mathcal{P}_{\mathbb{R}}$ and $\mu_{2}$
is $\boxplus$-infinitely divisible. The following result is essentially
contained in \cite{Bi-cusp} and the brief argument is included here
to establish notation.
\begin{lem}
\label{lem:subordonation is infinitely div, R} Let $\mu_{1},\mu_{2}\in\mathcal{P}_{\mathbb{R}}$
be such that $\mu_{2}$ is $\boxplus$-infinitely divisible, and let
$\rho_{1},\rho_{2}\in\mathcal{P}_{\mathbb{R}}$ be given by \textup{Theorem}
\emph{\ref{thm:subordination on the line (mult)-2}. }Then $\rho_{1}$
is $\boxplus$-infinitely divisible.
\end{lem}

\begin{proof}
Let $\Phi(z)=\gamma+z+N_{\sigma}(z)$ given by (\ref{eq:extension of eta inverse (line)-2})
be the analytic continuation of $F_{\mu_{2}}^{\langle-1\rangle}$
to $\mathbb{H}.$ Then (\ref{eq:defining boxtimes-2}) can be rewritten
as
\[
F_{\mu_{1}}^{\langle-1\rangle}(z)+\gamma+N_{\sigma}(z)=F_{\mu_{1}\boxplus\mu_{2}}^{\langle-1\rangle}(z)
\]
in a neighborhood of infinity. Replacing $z$ by $F_{\mu_{1}}(z)$
yields
\[
\gamma+z+N_{\sigma}(F_{\mu_{1}}(z))=F_{\mu_{1}\boxplus\mu_{2}}^{\langle-1\rangle}(F_{\mu_{1}}(z)),
\]
and therefore the function 
\[
\Psi(z)=\gamma+z+N_{\sigma}(F_{\mu_{1}}(z)),\quad z\in\mathbb{H},
\]
is an analytic continuation of $F_{\rho_{1}}^{\langle-1\rangle},$
thus establishing the conclusion of the lemma.
\end{proof}
The density $p_{\mu_{1}\boxplus\mu_{2}}$ of $\mu_{1}\boxplus\mu_{2}$
relative to Lebesgue measure has already been studied in \cite{Bi-cusp}
for the special case in which $\mu_{2}$ is a semicircular law, that
is, the measure $\sigma$ is a point mass at $0$. The following result
is \cite[Corollary 5]{Bi-cusp}.
\begin{prop}
\label{prop:Biane estimate} With the notation above, suppose that
$\sigma=\beta\delta_{0}$ for some $\beta>0$. If $I\subset\mathbb{R}$
is an open interval with an endpoint $x_{0}$ such that $p_{\mu_{1}\boxplus\mu_{2}}(x_{0})=0<p_{\mu_{1}\boxplus\mu_{2}}(x)$
for every $x\in I$, then
\[
p_{\mu_{1}\boxplus\mu_{2}}(x)\le\left[\frac{3}{4\pi^{3}\beta^{2}}|x-x_{0}|\right]^{1/3},\quad x\in I.
\]
 
\end{prop}

In order to extend this result to general $\boxplus$-infinitely divisible
measures $\mu_{2}$, we proceed as in the multiplicative cases. Thus,
we construct another convolution, this time with a semicircular measure,
with the property that the two convolutions share the same subordination
function.
\begin{lem}
\label{lem:trading convolutions, additive}Let $\mu_{1},\mu_{2}\in\mathcal{P}_{\mathbb{R}}$
be such that $\mu_{2}$ is $\boxplus$-infinitely divisible, and let
\[
\Phi(z)=\gamma+z+N_{\sigma}(z),\quad z\in\mathbb{H},
\]
be the analytic continuation of $F_{\mu_{2}}^{\langle-1\rangle}$.
Denote by $\rho_{1}\in\mathcal{P}_{\mathbb{R}}$ the $\boxplus$-infinitely
divisible measure such that $F_{\rho_{1}}^{\langle-1\rangle}$ has
the analytic continuation
\[
\Psi(z)=\gamma+z+N_{\sigma}(F_{\mu_{1}}(z)),\quad z\in\mathbb{H}.
\]
Suppose that $\beta=\int_{\mathbb{R}}(1+t^{2})\,d\sigma(t)$ is finite,
and set $\gamma'=\gamma+\int_{\mathbb{R}}t\,d\sigma(t)$. Denote by
$\nu_{1}\in\mathcal{P}_{\mathbb{R}}$ the probability measure satisfying
\[
G_{\nu_{1}}(z)=\frac{1}{\beta}\int_{\mathbb{R}}\frac{1+t^{2}}{F_{\mu_{1}}(z)-t}\,d\sigma(t),\quad z\in\mathbb{H},
\]
and let $\nu_{2}\in\mathcal{P}_{\mathbb{R}}$ be the semicircular
measure such that
\[
\gamma'+z+\frac{\beta}{z},\quad z\in\mathbb{H},
\]
is an analytic continuation of $F_{\nu_{2}}^{\langle-1\rangle}$.
Then $F_{\mu_{1}\boxplus\mu_{2}}=F_{\mu_{1}}\circ F_{\rho_{1}}$ and
$F_{\nu_{1}\boxplus\nu_{2}}=F_{\nu_{1}}\circ F_{\rho_{1}}$.
\end{lem}

\begin{proof}
The final assertion of the lemma is an easy verification that $F_{\mu_{1}\boxplus\mu_{2}}^{\langle-1\rangle}\circ F_{\mu_{1}}=F_{\nu_{1}\boxplus\nu_{2}}^{\langle-1\rangle}\circ F_{\nu_{1}}.$
One has to verify however that the measure $\nu_{1}$ actually exists,
and that amounts to showing that the reciprocal 
\[
F(z)=\frac{\beta}{\int_{\mathbb{R}}\frac{1+t^{2}}{F_{\mu_{1}}(z)-t}\,d\sigma(t)}
\]
maps $\mathbb{H}$ to itself and that
\[
\lim_{y\uparrow+\infty}\frac{F(iy)}{iy}=1.
\]
These facts follow from the corresponding properties of the function
$F_{\mu_{1}}$.
\end{proof}
We will use a decomposition, analogous to those for multiplicative
convolutions on $\mathbb{R}_{+}$ and $\mathbb{T}$. With the notation
of the preceding lemma, represent the domain $\Omega_{\rho_{1}}=F_{\rho_{1}}(\mathbb{H})$
as
\[
\Omega_{\rho_{1}}=\{x+iy:x\in\mathbb{R},y>f(x)\},
\]
where $f:\mathbb{R}\to[0,+\infty)$ is a continuous function. We recall
from \cite{huang-wang} that $G_{\mu_{1}}$ extends
continuously to the closure $\overline{\Omega_{\rho_{1}}}$ provided
that $\infty$ is allowed as a possible value. Based on \cite[Proposition 4.7]{B-B-IMRN}, it was shown in \cite[Theorem 3.6]{huang-wang}
that the set $\partial\Omega_{\rho_{1}}\cap\mathbb{R}=\{\alpha\in\mathbb{R}:f(\alpha)=0\}$
can be partitioned into three sets $A,B,$ and $C$ described as follows.
\begin{enumerate}
\item $A$ consists of those points satisfying $\mu_{1}(\{\alpha\})>0$
and
\[
\int_{\mathbb{R}}\left\{ 1+\frac{1}{t^{2}}\right\} \,d\sigma(t)\le\mu_{1}(\{\alpha\}).
\]
\item $B$ is characterized by the conditions $G_{\mu_{1}}(\alpha)\in\mathbb{R}\backslash\{0\}$
and
\[
\left[\int_{\mathbb{R}}\frac{1+t^{2}}{(1-tG_{\mu_{1}}(\alpha))^{2}}\,d\sigma(t)\right]\left[\int_{\mathbb{R}}\frac{d\mu_{1}(t)}{(\alpha-t)^{2}}\right]\le1.
\]
\item $C$ consists of those $\alpha$ satisfying $G_{\mu_{1}}(\alpha)=0$
and
\[
{\rm var}(\mu_{2})\int_{\mathbb{R}}\frac{d\mu_{1}(t)}{(\alpha-t)^{2}}\le1,
\]
where 
\[
{\rm var}(\mu_{2})=\int_{\mathbb{R}}t^{2}\,d\mu_{2}(t)-\left[\int_{\mathbb{R}}t\,d\mu_{2}(t)\right]^{2}
\]
denotes the variance of $\mu_{2}$.
\end{enumerate}
These inequalities provide a quantitative way to determine the zeros of the density $p_{\mu_{1}\boxplus \mu_{2}}$, because $p_{\mu_{1}\boxplus \mu_{2}}=-\pi^{-1}\Im (G_{\mu_1}\circ F_{\rho_1})$ on $\mathbb{R}$ and $F_{\rho_1}\left(\mathbb{R}\right)=\partial\Omega_{\rho_{1}}$. As in the multiplicative cases, they are derived from the chain rule of Julia-Carath\'eodory derivative. In each of the preceding inequalities, the improper integrals converge.
Equality in each case is achieved precisely when $F_{\rho_{1}}$ has
an infinite Julia-Carath\'eo\-dory derivative at the point $\Psi(\alpha)$. The set $A$ is always finite unless $\mu_{2}$ is a degenerate measure.
Moreover, if $t$ is an atom of $\mu_{1}\boxplus\mu_{2}$, then $F_{\rho_{1}}(t)\in A$.

We note for further use an alternative way to write the inequalities
defining the sets $B$ and $C$ \cite[Remark 3.7]{huang-wang}. For
this purpose, we use the Nevanlinna representation 
\[
F_{\mu_{1}}(z)=c+z-N_{\lambda}(z),\quad z\in\mathbb{H},
\]
where $c\in\mathbb{R}$ and $\lambda$ is a finite Borel measure on
$\mathbb{R}$. The inequality in the definition of $B$ can be replaced
by
\[
\left[\int_{\mathbb{R}}\frac{1+t^{2}}{(F_{\mu_{1}}(\alpha)-t)^{2}}\,d\sigma(t)\right]\left[1+\int_{\mathbb{R}}\frac{1+t^{2}}{(\alpha-t)^{2}}\,d\lambda(t)\right]\le1,
\]
and the inequality in the definition of $C$ can be replaced by
\[
(1+\alpha^{2})\lambda(\{\alpha\})\ge{\rm var}(\mu_{2}).
\]
In particular, every point $\alpha\in B$ must satisfy
\begin{equation}
\int_{\mathbb{R}}\frac{1+t^{2}}{(F_{\mu_{1}}(\alpha)-t)^{2}}\,d\sigma(t)<1.\label{eq:property of points in B}
\end{equation}
It is also the case that $C$ is a discrete subset of $\mathbb{C}$.
\begin{thm}
\label{thm:cusps on R}Suppose that $\mu_{1},\mu_{2}\in\mathcal{P}_{\mathbb{R}}$
are nondegenerate measures such that $\mu_{2}$ is $\boxplus$-infinitely
divisible, and let $\rho_{1}\in\mathcal{P}_{\mathbb{R}}$ satisfy
$F_{\mu_{1}\boxplus\mu_{2}}=F_{\mu_{1}}\circ F_{\rho_{1}}.$ Let $I\subset\mathbb{R}$
be an open interval with an endpoint $x_{0}$ such that $F_{\rho_{1}}(x)\in\mathbb{H}$
for every $x\in I$ and $F_{\rho_{1}}(x_{0})$ is real but not an
atom of $\mu_{1}$. Denote by $p_{\mu_{1}\boxplus\mu_{2}}$ the density
of $\mu_{1}\boxplus\mu_{2}$ relative to Lebesgue measure. Then $p_{\mu_{1}\boxplus\mu_{2}}(x)/|x-x_{0}|^{1/3}$
is bounded for $x\in I$ close to $x_{0}$. 
\end{thm}

\begin{proof}
Suppose that 
\[
c+z+N_{\sigma}(z),\quad z\in\mathbb{H},
\]
is the analytic continuation of $F_{\mu_{2}}^{\langle-1\rangle}$
to $\mathbb{H}$, where $c\in\mathbb{R}$ and $\sigma$ is a nonzero
(because $\mu_{2}$ is nondegenerate) finite measure on $\mathbb{R}$.
As in the case of $\mathbb{R}_{+}$, we can always find finite measures
$\sigma'$ and $\sigma''$ on $\mathbb{R}$ such that $\sigma''\ne0$
has compact support and $\sigma=\sigma'+\sigma''$. Define two $\boxplus$-infinitely
divisible measures $\mu'_{2},\mu_{2}''\in\mathcal{P}_{\mathbb{R}}$
by specifying that $F_{\mu_{2}'}^{\langle-1\rangle}$ and $F_{\mu_{2}^{\prime\prime}}^{\langle-1\rangle}$
have analytic continuations
\[
c+z+N_{\sigma^{\prime}}(z)\text{ and }z+N_{\sigma^{\prime\prime}}(z),\quad z\in\mathbb{H},
\]
respectively. Since $\mu_{2}'\boxplus\mu_{2}''=\mu_{2}$, we get $\mu_{1}\boxplus\mu_{2}=\mu_{1}''\boxplus\mu_{2}''$
and $\mu_{1}''=\mu_{1}\boxplus\mu_{2}'$. There exist two $\boxplus$-infinitely
divisible measures $\rho_{1}',\rho_{1}''\in\mathcal{P}_{\mathbb{R}}$
such that $F_{\mu_{1}''}=F_{\mu_{1}}\circ F_{\rho_{1}'}$, $F_{\mu_{1}\boxplus\mu_{2}}=F_{\mu_{1}''}\circ F_{\rho_{1}''}$.
Clearly, $F_{\rho_{1}}=F_{\rho_{1}'}\circ F_{\rho_{1}''}$, and we
argue that $F_{\rho_{1}''}(x_{0})$ is a real number but not an atom
of $\mu_{1}''$. Indeed, letting $z\rightarrow x_{0}$ in the inequality
\[
\Im F_{\rho_{1}}(z)=\Im(F_{\rho_{1}'}(F_{\rho_{1}''}(z))\ge\Im(F_{\rho_{1}''}(z)),\quad z\in\mathbb{H},
\]
the hypothesis $F_{\rho_{1}}(x_{0})\in\mathbb{R}$ shows that $F_{\rho_{1}''}(x_{0})\in\mathbb{R}$.
Suppose, to get a contradiction, that $F_{\rho_{1}''}(x_{0})$ is
an atom of $\mu_{1}''$. Then, as seen in \cite{BV-reg}, $F_{\rho_{1}'}(F_{\rho_{1}''}(x_{0}))$
is necessarily an atom of $\mu_{1}$, contrary to the hypothesis.

The above construction shows that the hypothesis of the theorem also
holds with $\mu_{1}'',\mu_{2}''$, and $\rho_{1}''$ in place of $\mu_{1},\mu_{2},$
and $\rho_{1}$, respectively. Moreover the measure $\sigma''$ has
a finite second moment. Therefore it suffices to prove the theorem
under the additional hypothesis that $\sigma$ has a finite second
moment. Under this hypothesis, Lemma \ref{lem:trading convolutions, additive}
applies and provides measures $\nu_{1}$ and $\nu_{2}$. Since the
set $\{x\in\mathbb{R}:p_{\mu_{1}\boxplus\mu_{2}}(x)>0\}$ is described
in terms of the measure $\rho_{1}$, namely,
\[
\{x\in\mathbb{R}:p_{\mu_{1}\boxplus\mu_{2}}(x)>0\}=\{x:F_{\rho_{1}}(x)\in\mathbb{H}\},
\]
we have
\[
\{x\in\mathbb{R}:p_{\mu_{1}\boxplus\mu_{2}}(x)>0\}=\{x\in\mathbb{R}:p_{\nu_{1}\boxplus\nu_{2}}(x)>0\}.
\]
By Proposition \ref{prop:Biane estimate}, it suffices to show that
the ratio $p_{\nu_{1}\boxplus\nu_{2}}(x)/p_{\mu_{1}\boxplus\mu_{2}}(x)$
is bounded away from zero for $x\in I$ close to $x_{0}$. The two
densities are evaluated in terms of the values of $G_{\nu_{1}}$ and
$G_{\mu_{1}}$ on $\partial\Omega_{\rho_{1}}$:
\begin{align*}
p_{\nu_{1}\boxplus\nu_{2}}(x) & =-\frac{1}{\pi}\Im G_{\nu_{1}}(F_{\rho_{1}}(x))\\
 & =-\frac{1}{\pi\beta}\int_{\mathbb{R}}\Im\left[\frac{G_{\mu_{1}}(F_{\rho_{1}}(x))}{1-tG_{\mu_{1}}(F_{\rho_{1}}(x))}\right](1+t^{2})\,d\sigma(t)\\
 & =-\frac{\Im G_{\mu_{1}}(F_{\rho_{1}}(x))}{\pi\beta}\int_{\mathbb{R}}\frac{1+t^{2}}{|1-tG_{\mu_{1}}(F_{\rho_{1}}(x))|^{2}}\,d\sigma(t)\\
 & =\frac{p_{\mu_{1}\boxplus\mu_{2}}(x)}{\beta}\int_{\mathbb{R}}\frac{1+t^{2}}{|1-tG_{\mu_{1}}(F_{\rho_{1}}(x))|^{2}}\,d\sigma(t).
\end{align*}
The hypotheses that $p_{\mu_{1}\boxplus\mu_{2}}(x_{0})=0$ and $F_{\rho_{1}}(x_{0})$
is not an atom of $\mu_{1}$ imply $F_{\rho_{1}}(x_{0})\in B\cup C$.
Using the Fatou's lemma, we conclude that 
\begin{align*}
\liminf_{x\to x_{0},x\in I}\frac{p_{\nu_{1}\boxplus\nu_{2}}(x)}{p_{\mu_{1}\boxplus\mu_{2}}(x)} & =\liminf_{x\to x_{0},x\in I}\frac{1}{\beta}\int_{\mathbb{R}}\frac{1+t^{2}}{|1-tG_{\mu_{1}}(F_{\rho_{1}}(x))|^{2}}\,d\sigma(t)\\
 & \ge\frac{1}{\beta}\int_{\mathbb{R}}\frac{1+t^{2}}{|1-tG_{\mu_{1}}(F_{\rho_{1}}(x_{0}))|^{2}}\,d\sigma(t)>0,
\end{align*}
thus finishing the proof.
\end{proof}
\begin{rem}
With the notation of the preceding proof, it is also true that
\[
\liminf_{x\to x_{0},x\in I}\frac{p_{\mu_{1}\boxplus\mu_{2}}(x)}{p_{\nu_{1}\boxplus\nu_{2}}(x)}\ge\beta\int_{\mathbb{R}}\frac{d\mu_{1}(t)}{\left(F_{\rho_{1}}(x_{0})-t\right)^{2}},
\]
in which the improper integral converges because $F_{\rho_{1}}(x_{0})\in B\cup C$.
To verify this, we use the parametrization $\partial\Omega_{\rho_{1}}=\{s+if(s):s\in\mathbb{R}\}$
to write 
\[
\{F_{\rho_{1}}(x):x\in I\}=\{s+if(s):s\in J\},
\]
where $J$ is an interval on which $f$ is positive and it has one endpoint
$\alpha=F_{\rho_{1}}(x_{0})\in\mathbb{R}$ such that $f(\alpha)=0$.
The fact that $\Im\Psi(s+if(s))=0$ for $s\in J$ yields the equation
\[
f(s)+\int_{\mathbb{R}}\frac{\Im G_{\mu_{1}}(s+if(s))}{|1-tG_{\mu_{1}}(s+if(s))|^{2}}\,(1+t^{2})d\sigma(t)=0,\quad s\in J.
\]
Using this in the above formula for densities, we obtain
\[
\frac{p_{\mu_{1}\boxplus\mu_{2}}(x)}{p_{\nu_{1}\boxplus\nu_{2}}(x)}=\beta\frac{-\Im G_{\mu_{1}}(s+if(s))}{f(s)}=\beta\int_{\mathbb{R}}\frac{d\mu_{1}(t)}{(t-s)^{2}+f(s)^{2}}.
\]
We can now apply the Fatou's lemma as $s\to\alpha$. 
\end{rem}

\begin{rem}
The two limits inferior above are actual limits precisely when $F_{\rho_{1}}'(x_{0})=+\infty$.
Indeed, as seen above, this condition is equivalent to
\[
\left[\int_{\mathbb{R}}\frac{1+t^{2}}{|1-tG_{\mu_{1}}(F_{\rho_{1}}(x_{0}))|^{2}}\,d\sigma(t)\right]\left[\int_{\mathbb{R}}\frac{d\mu_{1}(t)}{(F_{\rho_{1}}(x_{0})-t)^{2}}\right]=1.
\]
\end{rem}

\begin{rem}
\label{rem:no exceptions}When $x_{0}$ is assumed to be a zero of
the density $p_{\mu_{1}\boxplus\mu_{2}}$, it is easy to see that
$F_{\rho_{1}}(x_{0})$ is an atom of $\mu_{1}$ if and only if $F_{\rho_{1}}(x_{0})\in A$.
In many cases, the collection $\{x_{0}:F_{\rho_{1}}(x_{0})\text{ is an atom of }\mu_{1}\}$
is empty. This happens, of course, when $\mu_{1}$ has no atoms. This
also occurs when 
\[
\int_{\mathbb{R}}\left\{ 1+\frac{1}{t^{2}}\right\} \,d\sigma(t)\in[1,+\infty].
\]
Indeed, in this case the set $A$ is empty (provided, of course, that
$\mu_{1}$ is not degenerate and so its atoms cannot have measure
$1$).
\end{rem}

\begin{example}
Let $\mu_{1}$ be an arbitrary nondegenerate measure in $\mathcal{P}_{\mathbb{R}}$,
and let $\mu_{2}$ be the standard $(0,1)$ normal distribution. It
was shown in \cite{B-B-F-S} that $\mu_{2}$ is $\boxplus$-infinitely
divisible. We denote by $\sigma$ the associated measure that provides
the analytic continuation of $F_{\mu_{2}}^{\langle-1\rangle}$. Since
\[
-\Im G_{\mu_{2}}(x)=\pi p_{\mu_{2}}(x)=\sqrt{\frac{\pi}{2}}e^{-x^{2}/2},\quad x\in\mathbb{R},
\]
the continuous extension of $F_{\mu_{2}}$ to $\mathbb{R}$ has no
zeros and (see \cite[Proposition 5.1]{BWZ-super+})
\[
\int_{\mathbb{R}}\frac{1+t^{2}}{(x-t)^{2}}\,d\sigma(t)>1,\quad x\in\mathbb{R}.
\]
This inequality, along with (\ref{eq:property of points in B}), implies
that $A=B=\varnothing$ for the convolution $\mu_{1}\boxplus\mu_{2}$.
Moreover, since $C$ is a discrete set, the measure $\mu_{1}\boxplus\mu_{2}$
is absolutely continuous with support equal to $\mathbb{R}$. If $C$
is not empty and $\alpha\in C$, there is an open interval $I$ centered
at $x_{0}=F_{\mu_{2}}^{\langle-1\rangle}(\alpha)$ such that $p_{\mu_{1}\boxplus\mu_{2}}(x)/|x-x_{0}|^{1/3}$
is bounded for $x\in I\backslash\{x_{0}\}$. Suppose, for instance,
that $\mu_{1}=\frac{1}{2}(\delta_{1}+\delta_{-1})$ or the absolutely
continuous measure with density
\[
\frac{15}{16}\left[t^{4}\chi_{[-1,1]}(t)+\frac{1}{t^{4}}\chi_{\mathbb{R}\setminus[-1,1]}(t)\right].
\]
 In both cases, $\alpha=0$ is the unique solution of the equation
$G_{\mu_{1}}(\alpha)=0$ under the constraint
\[
\int_{\mathbb{R}}\frac{d\mu_{1}(t)}{(\alpha-t)^{2}}\le1.
\]
Moreover, we have $F_{\rho_{1}}'(F_{\rho_{1}}^{\langle-1\rangle}(0))=+\infty$
because the equality in the above constraint is achieved and thus, by Remark 9.6,
$p_{\mu_{1}\boxplus\mu_{2}}$ is comparable to $p_{\nu_{1}\boxplus\nu_{2}}$
in $I$. To obtain an example in which $F_{\rho_{1}}'(F_{\rho_{1}}^{\langle-1\rangle}(0))$
is finite, one can take $\mu_{1}$ to be the absolutely continuous
measure with density
\[
\frac{3}{14}\left[t^{2}\chi_{[-1,1]}(t)+\left|t\right|^{-3/2}\chi_{\mathbb{R}\setminus[-1,1]}(t)\right].
\]
\end{example}

\appendix

\section{free convolution semigroups}     
We denote by $\nu_{\boxplus}^{\gamma,\sigma}$ the $\boxplus$-infinitely divisible measure determined by $\gamma$
and $\sigma$ through the free L\'{e}vy-Hin\v{c}in formula \eqref{eq:extension of eta inverse (line)-2}. The notation
$\nu_{\boxtimes}^{\gamma,\sigma}$ is used in the multiplicative situation
for the same purpose. Measures in this section are assumed to be nondegenerate. 

Given a measure $\mu=\nu_{\boxplus}^{\gamma,\sigma}$, recall that
the map $F_{\mu}$ is injective on $\mathbb{R}$ and hence it has
at most one zero $t_{\mu}$. The point $t_{\mu}$ exists if and only
if 
\[
c=\int_{\mathbb{R}}\frac{1+t^{2}}{t^{2}}\,d\sigma(t)\leq1,
\]
and if this condition is satisfied, we have $\mu\left(\left\{ t_{\mu}\right\} \right)=1-c$
\cite[Proposition 5.1]{BWZ-super+}. Note that 
\[
\mu=\mu_{1}\boxplus\mu_{2}\quad\text{where}\quad\mu_{1}=\mu_{2}=\nu_{\boxplus}^{\frac{\gamma}{2},\frac{\sigma}{2}}.
\]
If the point $t_{\mu}$ does not exist (or, if the point $t_{\mu_{1}}$
relative to $\mu_{1}$ does not exist), then the set $A$ associated
with $\mu_{1}\boxplus\mu_{2}$ (see Section 9) would be empty and therefore
Theorem \ref{thm:cusps on R} applies to the density $p_{\mu}$ of $\mu$. Thus, we
have proved the following result.
\begin{prop}
\label{prop: inf +}Let $\mu=\nu_{\boxplus}^{\gamma,\sigma}$ be a
freely infinitely divisible law such that 
\[
\int_{\mathbb{R}}\frac{1+t^{2}}{t^{2}}\,d\sigma(t)>1.
\]
Suppose that $p_{\mu}>0$ on an open interval $I$, and that $p_{\mu}(t_{0})=0$
at an endpoint $t_{0}$ of $I$. Then we have $p_{\mu}(t)=O\left(\left|t-t_{0}\right|^{1/3}\right)$
for $t$ sufficiently close to $t_{0}$. 
\end{prop}

Next, we review the construction of the free additive convolution
semigroups from \cite{B-B-IMRN}. Given a measure $\nu\in\mathcal{P}_{\mathbb{R}}$,
we write the map $F_{\nu}$ in its Nevanlinna integral form
\[
F_{\nu}(z)=a+z-N_{\tau}(z),\quad z\in\mathbb{H},
\]
and for each $\beta>1$, we define \[\gamma_\beta=(1-\beta)a \quad \text{and} \quad \sigma_{\beta}=(\beta-1)\tau,\] so that the map
\[
\Phi(z)=\gamma_\beta+z+N_{\sigma_\beta}(z),\quad z\in\mathbb{H},
\]
is the analytic continuation of $F_{\mu}^{\left\langle -1\right\rangle }$
where $\mu=\nu_{\boxplus}^{\gamma_\beta,\sigma_\beta}$. There exists a unique
measure $\nu_{\beta}\in\mathcal{P}_{\mathbb{R}}$ such that 
\[
F_{\nu_{\beta}}=F_{\nu}\circ F_{\mu}\quad\text{in }\mathbb{H}.
\]
The family $\left\{ \nu_{\beta}:\beta>1\right\} $ forms a free additive
convolution semigroup in the sense that $\nu_{\beta_{1}+\beta_{2}}=\nu_{\beta_{1}}\boxplus\nu_{\beta_{2}}$
for all $\beta_{1},\beta_{2}>1$. Fix $\beta>1$. Since 
\[
F_{\nu}(z)=\frac{\beta}{\beta-1}z+\frac{1}{1-\beta}\Phi(z),
\]
the function $F_{\nu}$ has a continuous extension (still denoted
by $F_{\nu}$) to the image set $F_{\mu}\left(\mathbb{R}\right)$
such that 
\begin{equation}
t=\beta F_{\mu}(t)+(1-\beta)F_{\nu}(F_{\mu}(t)),\quad t\in\mathbb{R}.\label{inversion}
\end{equation}
The density $p_{\nu_{\beta}}=-\pi^{-1}\Im G_{\nu_{\beta}}=-\pi^{-1}\Im\left(G_{\nu}\circ F_{\mu}\right)$
of $\nu_{\beta}$ is continuous everywhere on $\mathbb{R}$, except
on the set 
\[
D=\left\{ t\in\mathbb{R}:F_{\nu}(F_{\mu}(t))=0\right\} .
\]
As seen in Section 8, the boundary curve $F_{\mu}\left(\mathbb{R}\right)$
is the graph of a continuous function $f:\mathbb{R}\rightarrow[0,+\infty)$.
For $t\in\mathbb{R}\setminus D$, the parameterization 
\begin{equation}
F_{\mu}(t)=x+if(x)\label{f}
\end{equation}
and (\ref{inversion}) imply the density formula
\begin{equation}
p_{\nu_{\beta}}(t)=\frac{\beta f(x)}{\pi(\beta-1)\left|F_{\nu}\left(F_{\mu}(t)\right)\right|^{2}},\quad\text{where }t\in\mathbb{R},\;F_{\nu}\left(F_{\mu}(t)\right)\neq0.\label{density}
\end{equation}
We have $t\in D$ if and only if $F_{\mu}(t)=t/\beta$; in which case
it leads to $f\left(F_{\mu}(t)\right)=0$. Thus, the set $D$ of discontinuities
of $p_{\nu_{\beta}}$ is related to the zero set of $f$. Since $\Im\Phi(x+iy)=y\left[1-\int_{\mathbb{R}}\frac{1+t^{2}}{(x-t)^{2}+y^{2}}\,d\sigma_\beta(t)\right]$
for $y>0$, it is clear that $f(x)=0$ if and only if $\int_{\mathbb{R}}\frac{1+t^{2}}{(x-t)^{2}}\,d\sigma_\beta(t)\leq1$.
The next result characterizes the zeros of $f$ in terms of $\nu$ and $\beta$,
and it shows that $D$ is a finite set.
\begin{prop}
\label{prop:zero+}The set $\left\{ \alpha\in\mathbb{R}:f(\alpha)=0\right\} $
is partitioned into two sets $A$ and $B$, defined as follows. 
\begin{enumerate}
\item The set $A$ consists of points $\alpha\in\mathbb{R}$ such that $\nu\left(\left\{ \alpha\right\} \right)>0$
and $\nu\left(\left\{ \alpha\right\} \right)\geq1-\beta^{-1}$.
\item The set $B$ consists of those $\alpha\in\mathbb{R}$ such that $F_{\nu}(\alpha)\in\mathbb{R}\setminus\left\{ 0\right\} $
and 
\[
F_{\nu}(\alpha)^{2}\int_{\mathbb{R}}\frac{d\nu(t)}{(\alpha-t)^{2}}\leq\frac{\beta}{\beta-1}.
\]
\end{enumerate}
It follows that $D=\left\{ \Phi(\alpha):\alpha\in A\right\} $.

\end{prop}

\begin{proof}
It was proved in \cite[Proposition 5.1]{B-B-IMRN} that $f(\alpha)=0$
if and only if the Julia-Carath\'{e}odory derivative $\Phi^{\prime}(\alpha)$
exists in $[0, +\infty)$; in which case we have 
\[
\Phi^{\prime}(\alpha)=\beta+(1-\beta)F_{\nu}^{\prime}(\alpha).
\]
It follows that 
\begin{multline*}
\left\{ \alpha\in\mathbb{R}:f(\alpha)=0\right\} =\left\{ \alpha:F_{\nu}^{\prime}(\alpha)\text{ exists in the interval }\left(0,\frac{\beta}{\beta-1}\right]\right\} \\
=\left\{ \alpha:F_{\nu}(\alpha)=0,\;0<F_{\nu}^{\prime}(\alpha)\leq\frac{\beta}{\beta-1}\right\} \cup\left\{ \alpha:F_{\nu}(\alpha)\in\mathbb{R}\setminus\left\{ 0\right\} ,\;0<F_{\nu}^{\prime}(\alpha)\leq\frac{\beta}{\beta-1}\right\} .
\end{multline*}
The two disjoint sets in the last union are precisely $A$ and $B$,
for $F_{\nu}^{\prime}(\alpha)=1/\nu\left(\left\{ \alpha\right\} \right)$
if $F_{\nu}(\alpha)=0$, and 
\[
F_{\nu}^{\prime}(\alpha)=\lim_{\varepsilon\downarrow0}\Re\left[\frac{F_{\nu}(\alpha+i\varepsilon)-F_{\nu}(\alpha)}{i\varepsilon}\right]=F_{\nu}(\alpha)^{2}\int_{\mathbb{R}}\frac{d\nu(t)}{(\alpha-t)^{2}}
\]
when $F_{\nu}(\alpha)\in\mathbb{R}\setminus\left\{ 0\right\} $. 
\end{proof}
\begin{thm}
\label{thm:+}Let $\left\{ \nu_{\beta}:\beta>1\right\} $ be the free
additive convolution semigroup generated by $\nu\in\mathcal{P}_{\mathbb{R}}$
and $\mu=\nu_{\boxplus}^{\gamma_\beta,\sigma_\beta}$. Let $\beta$ be a parameter such that
\[
\int_{\mathbb{R}}\frac{1+t^{2}}{t^{2}}\,d\sigma_\beta(t)>1.
\]
Suppose that $p_{\nu_{\beta}}>0$ on an open interval $I$,
$p_{\nu_{\beta}}(t_{0})=0$ at an endpoint $t_{0}$ of $I$, and $t_{0}\notin D$.
Then we have $p_{\nu_{\beta}}(t)=O\left(\left|t-t_{0}\right|^{1/3}\right)$
for $t$ sufficiently close to $t_{0}$. 
\end{thm}

\begin{proof}
The hypotheses imply that both $F_{\mu}(t_{0})$ and $F_{\nu}\left(F_{\mu}(t_{0})\right)$
are nonzero. Since the density $p_{\mu}=-\pi^{-1}\Im G_{\mu}$ of
$\mu$ is given by
\[
p_{\mu}(t)=\frac{f(x)}{\pi\left|F_{\mu}(t)\right|^{2}},\quad t\in\mathbb{R},
\]
we can rewrite the formula (\ref{density}) as 
\[
p_{\nu_{\beta}}(t)=\frac{\beta}{\beta-1}\frac{\left|F_{\mu}(t)\right|^{2}}{\left|F_{\nu}\left(F_{\mu}(t)\right)\right|^{2}}p_{\mu}(t),\quad t\in I\cup\left\{ t_{0}\right\} .
\]
The result follows from Proposition \ref{prop: inf +}.
\end{proof}
Following the same arguments, one can easily prove the analogs of
Propositions \ref{prop: inf +} and \ref{prop:zero+} and Theorem
\ref{thm:+} for free multiplicative convolution. We shall present
these results below and leave the proofs to the reader. The first
result is a direct consequence of Theorems \ref{thm:cusp on R+} and \ref{thm:cusp on T}.
\begin{prop}
We have the following cusp regularity for infinitely divisible laws. 
\begin{enumerate}
\item Let $\mu=\nu_{\boxtimes}^{\gamma,\sigma}$ be an $\boxtimes$-infinitely
divisible measure on $\mathbb{R}_{+}$ such that 
\[
\sigma\left((0,+\infty)\right)>0\quad\text{and}\quad\int_{[0,+\infty]}\frac{1+t^{2}}{(1-t)^{2}}\,d\sigma(t)>1.
\]
If the density $p_{\mu}$ of $\mu$ vanishes at an endpoint $t_{0}>0$
of an open interval $I\subset(0,+\infty)$ and $p_{\mu}>0$ on $I$,
then $p_{\mu}(t)=O\left(\left|t-t_{0}\right|^{1/3}\right)$ for $t$
close to $t_{0}$. 
\item If $\mu=\nu_{\boxtimes}^{\gamma,\sigma}$ is an $\boxtimes$-infinitely
divisible measure on $\mathbb{T}$ such that 
\[
\int_{\mathbb{T}}t\,d\sigma(t)\neq0,\quad2\int_{\mathbb{T}}\frac{d\sigma(t)}{\left|1-t\right|^{2}}>1,
\]
$p_{\mu}(t_{0})=0$ at an endpoint $t_{0}$ of an open arc $\Gamma\subset\mathbb{T}$,
and $p_{\mu}>0$ on $\Gamma$, then $p_{\mu}(t)=O\left(\left|t-t_{0}\right|^{1/3}\right)$
for $t$ close to $t_{0}$. 
\end{enumerate}
\end{prop}

As shown in \cite{B-B-IMRN}, the construction of free multiplicative
powers in $\mathcal{P}_{\mathbb{R}_{+}}$ actually goes beyond the
integer case discussed at the end of Section 2. Thus, given a measure $\nu\in\mathcal{P}_{\mathbb{R_{+}}}$
and a real number $\beta>1$, there exists unique measures $\nu_{\beta},\mu\in\mathcal{P}_{\mathbb{R}_{+}}$
such that $\mu=\nu_{\boxtimes}^{\gamma_\beta,\sigma_\beta}$ for some $(\gamma_\beta,\sigma_\beta)$ and 
\[
\eta_{\nu_{\beta}}(z)=\eta_{\nu}\left(\eta_{\mu}(z)\right),\quad z\in\mathbb{C}\setminus\mathbb{R}_{+}.
\]
The family $\{\nu_{\beta}:\beta>1\}$ satisfies the semigroup property $\nu_{\beta_{1}+\beta_{2}}=\nu_{\beta_{1}}\boxtimes\nu_{\beta_{2}}$ for all $\beta_1,\beta_2>1$. Fix $\beta>1$. The map $\eta_{\nu}$ extends continuously to $\eta_{\mu}\left((0,+\infty)\right)$,
and so does the map $\eta_{\nu_{\beta}}$ to $(0,+\infty)$. More
importantly, these two extensions never take $\infty$ or $0$ as
a value. The relation \eqref{eq:eta of subordination, halflin} persists when $k$ is replaced by $\beta$
and $z=1/t$, $t\in(0,+\infty)$, that is, 
\[
t\eta_{\mu}(1/t)^{\beta}=\eta_{\nu_{\beta}}(1/t)^{\beta-1},\quad t\in(0,+\infty),
\]
where the principal branch of the power function is used. Following
\eqref{eq:param of eta_mu(pos line)}, we parameterize $\eta_{\mu}(1/t)$ by 
\[
\eta_{\mu}(1/t)=re^{if(r)},
\]
then the relation between $\eta_{\mu}$ and $\eta_{\nu_{\beta}}$
shows that the density $p_{\nu_{\beta}}$ of $\nu_{\beta}$ is given
by 
\[
tp_{\nu_{\beta}}\left(t\right)=\frac{\left(tr\right)^{\frac{1}{\beta-1}}}{\pi}\frac{r\sin\left(\frac{\beta}{\beta-1}f(r)\right)}{\left|1-\eta_{\nu}\left(re^{if(r)}\right)\right|^{2}},\quad\text{where }r>0,\;\eta_{\nu}\left(re^{if(r)}\right)\neq1.
\]
The next result is the analog of Proposition \ref{prop:zero+}. The
key to its proof is that $f(\alpha)=0$ if and only if $\eta_{\nu}(\alpha)\in(0,+\infty)$
and the Julia-Carath\'{e}odory derivative $\eta_{\nu}^{\prime}(\alpha)$
exists and satisfies $(\beta-1)\alpha\eta_{\nu}^{\prime}(\alpha)\leq\beta\eta_{\nu}(\alpha)$.
(This equivalence was already shown implicitly in the proof of Proposition
5.2 of \cite{B-B-IMRN}.)
\begin{prop}
The set $\left\{ \alpha>0:f(\alpha)=0\right\} $ is partitioned into
two sets $A$ and $B$, defined as follows. 
\begin{enumerate}
\item The set $A$ consists of points $\alpha>0$ such that $\nu\left(\left\{ 1/\alpha\right\} \right)>0$
and $\nu\left(\left\{ 1/\alpha\right\} \right)\geq1-\beta^{-1}$.
\item The set $B$ consists of those $\alpha>0$ such that $\eta_{\nu}(\alpha)\in(0,+\infty)\setminus\left\{ 1\right\} $
and 
\[
\left|1-\eta_{\nu}(\alpha)\right|^{2}\int_{\mathbb{R}_{+}}\frac{\alpha t}{(1-\alpha t)^{2}}\,d\nu(t)\leq\frac{\beta\eta_{\nu}(\alpha)}{\beta-1}.
\]
\end{enumerate}
The set $D=\left\{ t>0:\eta_{\nu}\left(\eta_{\mu}(1/t)\right)=1\right\} $
of discontinuities of $p_{\nu_{\beta}}$ is equal to the finite set
$\left\{ 1/\Phi(\alpha):\alpha\in A\right\} $.
\end{prop}

The cusp behavior of $p_{\nu_{\beta}}$ is our next result. The key
ingredient of its proof is the identity 
\[
p_{\nu_{\beta}}\left(t\right)=\left(tr\right)^{\frac{1}{\beta-1}}\frac{\sin\left(\frac{\beta}{\beta-1}f(r)\right)}{\sin(f(r))}\frac{\left|1-re^{if(r)}\right|^{2}}{\left|1-\eta_{\nu}\left(re^{if(r)}\right)\right|^{2}}\,p_{\mu}\left(t\right).
\]

\begin{thm}
Let $\left\{ \nu_{\beta}:\beta>1\right\} $ be the free multiplicative
convolution semigroup generated by $\nu\in\mathcal{P}_{\mathbb{R}_{+}}$
and $\mu=\nu_{\boxtimes}^{\gamma_\beta,\sigma_\beta}$. Let $\beta$ be a parameter such that 
\[
\sigma_\beta\left((0,+\infty)\right)>0\quad\text{and}\quad\int_{[0,+\infty]}\frac{1+t^{2}}{(1-t)^{2}}\,d\sigma_\beta(t)>1.
\]
Suppose that $p_{\nu_{\beta}}>0$ on an open interval $I\subset(0,+\infty)$,
$p_{\nu_{\beta}}(t_{0})=0$ at an endpoint $t_{0}>0$ of $I$, and $t_{0}\notin D$.
Then we have $p_{\nu_{\beta}}(t)=O\left(\left|t-t_{0}\right|^{1/3}\right)$
for $t$ sufficiently close to $t_{0}$. 
\end{thm}

Finally, we turn to the free multiplicative convolution semigroups
in $\mathcal{P}_{\mathbb{T}}^{*}$. Let $\nu_{k}=\nu^{\boxtimes k}$,
$k\in(1,+\infty)$, be the free convolution powers discussed at the
end of Section 5. Let $\mu=\nu_{\boxtimes}^{\gamma_k,\sigma_k}$ be the
$\boxtimes$-infinitely divisible law such that $\eta_{\nu_{k}}=\eta_{\nu}\circ\eta_{\mu}$
and 
\[
\eta_{\nu_{k}}(z)=\eta_{\mu}(z)\left(\frac{\eta_{\mu}(z)}{z}\right)^{\frac{1}{k-1}},\qquad z\in\mathbb{T}.
\]
Fix $k>1$. At a point $\xi\in\mathbb{T}$, we parameterize the value $\eta_{\mu}\left(\overline{\xi}\right)$
by 
\[
\eta_{\mu}(\overline{\xi})=R(t)t
\]
to get the density $p_{\nu_{k}}$ of $\nu_{k}$ as follows: 
\[
p_{\nu_{k}}(\xi)=\frac{1-R(t)^{\frac{2k}{k-1}}}{2\pi\left|1-\eta_{\nu}\left(R(t)t\right)\right|^{2}},\quad\text{where }t\in\mathbb{T},\;\eta_{\nu}\left(R(t)t\right)\neq1.
\]
Below is a characterization of zeros of $p_{\nu_{k}}$, whose proof
is based on \cite[Proposition 5.3]{B-B-IMRN}.
\begin{prop}
The set $\left\{ t\in\mathbb{T}:R(t)=1\right\} $ is partitioned into
two sets $A$ and $B$, defined as follows. 
\begin{enumerate}
\item The set $A$ consists of points $t\in\mathbb{T}$ such that $\nu\left(\left\{ \overline{t}\right\} \right)>0$
and $\nu\left(\left\{ \overline{t}\right\} \right)\geq1-\beta^{-1}$.
\item The set $B$ consists of those $t\in\mathbb{T}$ such that $\eta_{\nu}(t)\in\mathbb{T}\setminus\left\{ 1\right\} $,
the Julia-Carath\'{e}odory derivative $\eta_{\nu}^{\prime}(t)$ exists,
and 
\[
\left|\eta_{\nu}^{\prime}(t)\right|\leq\frac{\beta}{\beta-1}.
\]
\end{enumerate}
The set $D=\left\{ t\in\mathbb{T}:\eta_{\nu}\left(\eta_{\mu}(\overline{t})\right)=1\right\} $
of discontinuities of $p_{\nu_{\beta}}$ is equal to the finite set
$\left\{ \overline{\Phi(\alpha)}:\alpha\in A\right\} $.
\end{prop}

Finally, we use the relation 
\[
p_{\nu_{k}}(\xi)=\frac{\left|1-R(t)t\right|^{2}}{\left|1-\eta_{\nu}\left(R(t)t\right)\right|^{2}}\frac{1-R(t)^{\frac{2k}{k-1}}}{1-R(t)^{2}}\,p_{\mu}(\xi)
\]
to conclude the following result. 
\begin{thm}
Let $\left\{ \nu_{k}:k>1\right\} $ be the free multiplicative convolution
semigroup generated by $\nu\in\mathcal{P}_{\mathbb{T}}^{*}$ and $\mu=\nu_{\boxtimes}^{\gamma_k,\sigma_k}$. Let $k$ be a parameter such that 
\[
\int_{\mathbb{T}}t\,d\sigma_k(t)\neq0\quad\text{and}\quad2\int_{\mathbb{T}}\frac{d\sigma_k(t)}{\left|1-t\right|^{2}}>1.
\]
Suppose that $p_{\nu_{k}}>0$ on an open arc $\Gamma$, $p_{\nu_{k}}(t_{0})=0$
at an endpoint $t_{0}$ of $\Gamma$, and $t_{0}\notin D$. Then we
have $p_{\nu_{k}}(t)=O\left(\left|t-t_{0}\right|^{1/3}\right)$ for
$t$ sufficiently close to $t_{0}$. 
\end{thm}

\section{cusps of $1/3$-exponent}
In view of the cusp regularity results, it is natural to ask whether
the $1/3$-exponent can be achieved or not. Here we provide some quantitative
conditions for constructing such cusps in the interior of the support
of free convolution. The case of free additive convolution with a
semicircle law has been addressed in \cite{Bi-cusp} and \cite{Claeys}.
Inspired by the paper \cite{Claeys}, we extend the results of \cite{Bi-cusp,Claeys}
to additive and multiplicative free convolutions with infinitely divisible
laws. To this end, we first review a well-known fact from \cite{Greenstein}
concerning the analytic continuation of Nevanlinna integral forms.
In the sequel, any continuous or analytic extension of a function $f$
is still denoted by $f$, should such an extension exist. 
\begin{prop}
\label{prop:Greenstein}\cite{Greenstein} For a finite Borel measure
$\sigma$ on $\mathbb{R}$, recall that 
\[
N_{\sigma}(z)=\int_{\mathbb{R}}\frac{1+tz}{z-t}\,d\sigma(t),\quad z\in\mathbb{H}.
\]
Let $I\subset\mathbb{R}$ be an open interval. The function $N_{\sigma}$
extends analytically across the interval $I$ into $-\mathbb{H}$
if and only if the restriction of $\sigma$ on $I$ is absolutely
continuous with a real-analytic density $g$ on $I$; in which case,
the extension in the lower half-plane satisfies $N_{\sigma}(z)=\overline{N_{\sigma}(\overline{z})}-2\pi i(1+z^{2})g(z)$
for $z\in U\cap(-\mathbb{H})$, where $U$ is an open neighborhood
of $I$ and $g(z)$ denotes the complex analytic extension of $g(x)$
for $x\in I$. In particular, $N_{\sigma}$ extends across $I$ by
Schwarz reflection if and only if $\sigma(I)=0$. 
\end{prop}

The preceding result also applies to the Cauchy transform of $\mu\in\mathcal{P}_{\mathbb{R}}$,
because $(1+z^{2})G_{\mu}(z)-z=N_{\mu}(z)$ for $z\in\mathbb{H}$.

The next lemma also establishes the notation used in our next result.
The map $\Psi$ below is exactly the analytic continuation of $F_{\rho_{1}}^{\left\langle -1\right\rangle }$
from Lemma \ref{lem:trading convolutions, additive}.
\begin{lem}
\label{lem:+}Let $\mu_{1}\in\mathcal{P}_{\mathbb{R}}$. Suppose that
the restriction of $\mu_{1}$ on an open interval $I\subset\text{\emph{supp}}(\mu_{1})$
is absolutely continuous with a real-analytic density $g$ on $I$,
and that $\alpha\in I$ is a zero of order $2k$ ($k\geq1$) for $g$.
Denote by $a_{0}$ the real number $G_{\mu_{1}}(\alpha)$. Suppose
that $\mu_{2}=\nu_{\boxplus}^{\gamma,\sigma}$ is an $\boxplus$-infinitely
divisible measure whose $R$-transform $R_{\mu_{2}}(z)=\gamma+N_{\sigma}(1/z)$,
$z\in-\mathbb{H}$, extends analytically to a disk centered at the
point $a_{0}$ and $R_{\mu_{2}}(a_{0})\in\mathbb{R}$. Then: 
\begin{enumerate}
\item The Cauchy transform $G_{\mu_{1}}$ extends analytically to a disk
centered at $\alpha$. If $G_{\mu_{1}}(z)=\sum_{n=0}^{\infty}a_{n}(z-\alpha)^{n}$
is the power series representation of $G_{\mu_{1}}$ at $z=\alpha$,
then 
\[
a_{n}=(-1)^{n}\int_{\mathbb{R}}\frac{d\mu_{1}(t)}{(\alpha-t)^{n+1}},\quad n=0,1,2,\cdots,2k-1,
\]
and $a_{2k}\in-\mathbb{H}$. Here we set the power $(-1)^{0}=1$. 
\item If $R_{\mu_{2}}(z)=\sum_{n=0}^{\infty}b_{n}(z-a_{0})^{n}$ is the
power series representation of $R_{\mu_{2}}$ in a disk centered at
$a_{0}$, then 
\[
b_{1}=\int_{\mathbb{R}}\frac{1+t^{2}}{(1-ta_{0})^{2}}\,d\sigma(t).
\]
\item The function $\Psi(z)=z+R_{\mu_{2}}\left(G_{\mu_{1}}(z)\right)$,
$z\in\mathbb{H}$, extends analytically to a disk centered at $\alpha$,
and the first four Taylor coefficients of $\Psi(z)$ at $z=\alpha$
are $c_{0}=\alpha+b_{0}$, $c_{1}=1+b_{1}a_{1}$, $c_{2}=b_{1}a_{2}+b_{2}a_{1}^{2}$,
and $c_{3}=b_{1}a_{3}+2b_{2}a_{1}a_{2}+b_{3}a_{1}^{3}$.
\end{enumerate}
\end{lem}

\begin{proof}
Shrinking the interval $I$ if necessary, we may assume that 
\begin{equation}
g(t)=A(t-\alpha)^{2k}h(t),\quad t\in I,\label{eq:g}
\end{equation}
where $A>0$ and $h$ is real-analytic in $I$ with $h(\alpha)=1$.
By Proposition \ref{prop:Greenstein}, the Cauchy transform $G_{\mu_{1}}$
extends analytically to a disk $D$ centered at $\alpha$ such that
$G_{\mu_{1}}(z)=\overline{G_{\mu_{1}}(\overline{z})}-2\pi ig(z)$
for $z\in D\cap(-\mathbb{H})$, or, equivalently, $\sum_{n=0}^{\infty}a_{n}(z-\alpha)^{n}=\sum_{n=0}^{\infty}\overline{a_{n}}(z-\alpha)^{n}-2\pi iA(z-\alpha)^{2k}h(z)$
for $z\in D\cap(-\mathbb{H})$. It follows from the uniqueness of power series representation that $a_{0},a_{1},\cdots,a_{2k-1}\in\mathbb{R}$
and $\Im a_{2k}=-\pi A<0$. For $0\leq n\leq2k-1$, we have 
\[
a_{n}=\lim_{\varepsilon\downarrow0}\frac{G_{\mu_{1}}^{(n)}(\alpha+i\varepsilon)}{n!}=\lim_{\varepsilon\downarrow0}(-1)^{n}\int_{\mathbb{R}}\frac{d\mu_{1}(t)}{(\alpha-t+i\varepsilon)^{n+1}}=(-1)^{n}\int_{\mathbb{R}}\frac{d\mu_{1}(t)}{(\alpha-t)^{n+1}},
\]
where (\ref{eq:g}) and the dominated convergence theorem are used
in the last equality. The notation $G_{\mu_{1}}^{(n)}$ means the
$n$-th complex derivative of $G_{\mu_{1}}$ in $\mathbb{H}$. 

Define 
\[
I(z)=\gamma+\int_{\mathbb{R}}\frac{z+t}{1-tz}\,d\sigma(t),\quad z\notin\mathbb{R},
\]
and note that $I\left(\mathbb{H}\right)\subset\mathbb{H}$, $I(\overline{z})=\overline{I(z)}$
for $z\notin\mathbb{R}$, and $I(z)=R_{\mu_{2}}(z)$ for $z\in-\mathbb{H}$.
The analyticity of $R_{\mu_{2}}$ at $a_{0}$ implies that the complex
derivative 
\[
\lim_{\varepsilon\downarrow0}\frac{R_{\mu_{2}}(a_{0}-i\varepsilon)-R_{\mu_{2}}(a_{0})}{-i\varepsilon}
\]
exists and is equal to $b_{1}$. Since $R_{\mu_{2}}(a_{0})=\overline{R_{\mu_{2}}(a_{0})}$
and $R_{\mu_{2}}(a_{0}-i\varepsilon)=\overline{I(a_{0}+i\varepsilon)}$,
the Julia-Carath\'{e}odory derivative $I^{\prime}(a_{0})$
of $I$ at $a_{0}$ exists in $(0,+\infty)$ and is equal to $\overline{b_{1}}$ (cf. \cite{Shapiro}). So, we have $b_{0}\in\mathbb{R}$, $b_{1}>0$, and 
\begin{eqnarray*}
\int_{\mathbb{R}}\frac{1+t^{2}}{(1-ta_{0})^{2}}\,d\sigma(t) & \leq & \liminf_{\varepsilon\downarrow0}\int_{\mathbb{R}}\frac{1+t^{2}}{(1-ta_{0})^{2}+(t\varepsilon)^{2}}\,d\sigma(t)\quad\text{(Fatou's lemma)}\\
 & = & \liminf_{\varepsilon\downarrow0}\frac{\Im I(a_{0}+i\varepsilon)}{\varepsilon}\\
 & = & \liminf_{\varepsilon\downarrow0}\Re\left[\frac{I(a_{0}+i\varepsilon)-b_{0}}{i\varepsilon}\right]=I^{\prime}(a_{0})=b_{1}<+\infty.
\end{eqnarray*}
Then the dominated convergence theorem implies that the inequality
in the preceding calculation is actually an equality, whence the
integral formula of $b_{1}$ is proved. 

The proof of (3) is a simple manipulation of power series, and we
omit the details. 
\end{proof}
\begin{rem}
The regularity assumption of $\mu_{1}$ arises from the random matrix
theory. Indeed, it has been shown in \cite{Deift98,Deift99} that
for some special potentials $V$, the limiting density of the Gibbs
measures associated with $V$ can have a zero of even order in the interior of the
support. 
\begin{rem}
If $a_{0}\neq0$ and $a_{0}\notin\text{supp}(\sigma)$, then 
\[
b_{n}=\int_{\mathbb{R}}\frac{t^{n-1}(1+t^{2})}{(1-ta_{0})^{n+1}}\,d\sigma(t),\quad n\geq2.
\]
If $a_{0}=0$ and $\text{supp}(\sigma)$ is bounded, then $b_{n}=\kappa_{n+1}$,
the $(n+1)$-th free cumulant of $\mu_{2}$, for $n\geq0$ \cite{Benach-Georges}. 
\begin{rem}
An easy induction argument shows that 
\begin{equation}
c_{n}=b_{1}a_{n}+\sum_{j=2}^{n-1}b_{j}\left(\sum_{k_{1}+k_{2}+\cdots+k_{j}=n}a_{k_{1}}a_{k_{2}}\cdots a_{k_{j}}\right)+b_{n}a_{1}^{n},\quad n\geq4,\label{eq:c_n}
\end{equation}
in the power series representation $\Psi(z)=\sum_{n=0}^{\infty}c_{n}(z-\alpha)^{n}$.
\end{rem}

\end{rem}

\end{rem}

\begin{thm}
\label{thm: +}Suppose that $\mu_{1}$ and $\mu_{2}$ satisfy the
hypotheses of \emph{Lemma} \emph{\ref{lem:+}}. Assume in addition
that 
\[
\int_{\mathbb{R}}\frac{d\mu_{1}(t)}{(\alpha-t)^{2}}\int_{\mathbb{R}}\frac{1+t^{2}}{(1-ta_{0})^{2}}\,d\sigma(t)=1\quad\text{{\rm(}that is, \ensuremath{c_{1}=0}{\rm)}}.
\]
Then $c_{0}$ is an isolated zero of the density $p_{\mu_{1}\boxplus\mu_{2}}$
in the interior $[\text{\emph{supp}}(\mu_{1}\boxplus\mu_{2})]^{\circ}$
such that 
\begin{enumerate}
\item If $c_{2}=0$ and $\Re c_{3}<0$, then 
\[
\lim_{x\uparrow c_{0}}\frac{p_{\mu_{1}\boxplus\mu_{2}}(x)}{\left|x-c_{0}\right|^{1/3}}=\frac{-a_{1}}{\pi\sqrt[3]{\left|c_{3}\right|}}\sin\frac{\theta}{3},
\]
and 
\[
\lim_{x\downarrow c_{0}}\frac{p_{\mu_{1}\boxplus\mu_{2}}(x)}{\left|x-c_{0}\right|^{1/3}}=\frac{-a_{1}}{\pi\sqrt[3]{\left|c_{3}\right|}}\cos\left(\frac{\theta}{3}-\frac{\pi}{6}\right),
\]
where $\theta=\arg c_{3}\in(\pi/2,3\pi/2)$. 
\item If $\Im c_{2}<0$, then 
\[
\lim_{x\uparrow c_{0}}\frac{p_{\mu_{1}\boxplus\mu_{2}}(x)}{\left|x-c_{0}\right|^{1/2}}=\frac{a_{1}}{\pi\sqrt{\left|c_{2}\right|}}\cos\frac{\theta}{2},
\]
and 
\[
\lim_{x\downarrow c_{0}}\frac{p_{\mu_{1}\boxplus\mu_{2}}(x)}{\left|x-c_{0}\right|^{1/2}}=\frac{-a_{1}}{\pi\sqrt{\left|c_{2}\right|}}\sin\frac{\theta}{2},
\]
where $\theta=\arg c_{2}\in(\pi,2\pi)$. 
\end{enumerate}
\end{thm}

\begin{proof}
The density formula (\ref{eq:g}) implies that for $x$ close to $\alpha$
and $x\neq\alpha$, one has
\[
\int_{\mathbb{R}}\frac{d\mu_{1}(t)}{(x-t)^{2}}\geq\liminf_{\delta\downarrow0}\int_{(x-\delta,x+\delta)}\frac{A(\alpha-t)^{2k}h(t)}{(x-t)^{2}}\,dt=+\infty.
\]
This shows that $f(x)>0$ and hence $p_{\mu_{1}\boxplus\mu_{2}}(\Psi(x+if(x)))=-\pi^{-1}\Im G_{\mu_1}(x+if(x))>0$. The hypothesis $c_{1}=0$ implies
that $\alpha\in B\cup C$, and thus $c_{0}=\Psi(\alpha)$ belongs to the interior $[\text{supp}(\mu_{1}\boxplus\mu_{2})]^{\circ}$
and is an isolated zero for $p_{\mu_{1}\boxplus\mu_{2}}$.

Given $r>0$, we define the power $(z)^{r}=\left|z\right|^{r}e^{ir\arg z}$ for $z\neq0$ and $\arg z\in[0,2\pi)$.
In Case (1), Lemma \ref{lem:+} shows that in a disk centered at $\alpha$
the function $\Psi$ admits the power series representation 
\[
\Psi(z)=c_{0}+c_{3}(z-\alpha)^{3}+\sum_{n=4}^{\infty}c_{n}(z-\alpha)^{n}.
\]
By the argument principle, $\Psi(z)$ is locally a $3$-to-$1$ function near $z=\alpha$.
An application of the holomorphic inverse function theorem shows that
there exist three functions $F_{0}$, $F_{1}$, and $F_{2}$ defined
by the convergent series 
\[
F_{m}(z)=\alpha+\frac{e^{2m\pi i/3}}{(c_{3})^{1/3}}(z-c_{0})^{1/3}+e^{2m\pi i/3}\sum_{n=2}^{\infty}d_{n,m}(z-c_{0})^{n/3},\quad m=0,1,2,
\]
such that $\Psi\left(F_{m}(z)\right)=z$ for $z$ sufficiently close
to $c_{0}$ and $\arg(z-c_{0})\in[0,2\pi)$. To determine the value
of $m$ for our purpose, we write $(c_{3})^{1/3}=\sqrt[3]{\left|c_{3}\right|}e^{i\theta/3}$
where $\theta=\arg c_{3}\in(\pi/2,3\pi/2)$. Fix $\tau\in(0,\pi)$
and observe that 
\[
F_{m}(c_{0}+\varepsilon e^{i\tau})=\alpha+\sqrt[3]{\frac{\varepsilon}{\left|c_{3}\right|}}\exp\left(\frac{2m\pi i-\theta i+\tau i}{3}\right)+o(1)\quad(\varepsilon\rightarrow0^{+}).
\]
Since $\Im F_{\rho_{1}}(c_{0}+\varepsilon e^{i\tau})>0$ for all $\varepsilon>0$,
we conclude that $F_{1}(z)=F_{\rho_{1}}(z)$ for $z$ sufficiently
close to $c_{0}$ in $\mathbb{H}\cup\left(\mathbb{R}\setminus\left\{ c_{0}\right\} \right)$,
that is, $m=1$. We shall use $F_{1}$ to calculate the asymptotics
of $F_{\rho_{1}}(x)$ as $x\rightarrow c_{0}$. It follows that 
\[
G_{\mu_{1}}\left(F_{\rho_{1}}(x)\right)=a_{0}+a_{1}\alpha+\frac{a_{1}}{\sqrt[3]{\left|c_{3}\right|}}\exp\left(\frac{2\pi i-\theta i}{3}\right)(x-c_{0})^{1/3}+O\left(\left|x-c_{0}\right|^{2/3}\right)
\]
for all $x\in\mathbb{R}\setminus\left\{ c_{0}\right\} $ that are
sufficiently close to $c_{0}$. Note that 
\[
(x-c_{0})^{1/3}=\begin{cases}
e^{i\pi/3}\left|x-c_{0}\right|^{1/3} & \text{if }x<c_{0},\\
\left|x-c_{0}\right|^{1/3} & \text{if }x>c_{0}.
\end{cases}
\]
Therefore, the subordination $G_{\mu_{1}\boxplus\mu_{2}}=G_{\mu_{1}}\circ F_{\rho_{1}}$
shows that
\[
\frac{p_{\mu_{1}\boxplus\mu_{2}}(x)}{\left|x-c_{0}\right|^{1/3}}=-\frac{1}{\pi}\frac{\Im G_{\mu_{1}\boxplus\mu_{2}}(x)}{\left|x-c_{0}\right|^{1/3}}\rightarrow\frac{-a_{1}}{\pi\sqrt[3]{\left|c_{3}\right|}}\cdot\begin{cases}
\sin\frac{\theta}{3} & \text{if }x\rightarrow c_{0}^{-};\\
\frac{\sqrt{3}}{2}\cos\frac{\theta}{3}+\frac{1}{2}\sin\frac{\theta}{3} & \text{if }x\rightarrow c_{0}^{+},
\end{cases}
\]
finishing the proof of Case (1).

The proof of Case (2) goes exactly like that of Case (1), except this
time we have 
\[
\Psi(z)=c_{0}+c_{2}(z-\alpha)^{2}+\sum_{n=3}^{\infty}c_{n}(z-\alpha)^{n},
\]
and hence there are two branches for the choice of the continuation
of $F_{\rho_{1}}$, namely, 
\[
F_{m}(z)=\alpha+\frac{e^{m\pi i}}{(c_{2})^{1/2}}(z-c_{0})^{1/2}+e^{m\pi i}\sum_{n=2}^{\infty}d_{n,m}(z-c_{0})^{n/2},\quad m=0,1.
\]
It is easy to see that $F_{1}$ coincides with $F_{\rho_{1}}$ nearby
$c_{0}$, and the desired asymptotics follows.
\end{proof}
\begin{rem}
\label{rem:+ analytic}If 
\[
\int_{\mathbb{R}}\frac{d\mu_{1}(t)}{(\alpha-t)^{2}}\int_{\mathbb{R}}\frac{1+t^{2}}{(1-ta_{0})^{2}}\,d\sigma(t)<1,
\]
then $c_{1}>0$ and the inverse function theorem shows that $p_{\mu_{1}\boxplus\mu_{2}}$
is real-analytic at the zero $c_{0}$. Consequently, we have $p_{\mu_{1}\boxplus\mu_{2}}(x)=O\left(\left|x-c_{0}\right|^{2}\right)$
for $x$ sufficiently close to its local minimizer $c_{0}$. 
\begin{rem}
\label{rem:1/2 exp}If $c_{1}=0$ and if either (i) $k=1$, $\Im b_{2}=0$
or (ii) $k\geq2$, $\Im b_{2}<0$ holds, then $\Im c_{2}<0$ and $c_{0}$
is a cusp of $1/2$-exponent by Case (2).
\begin{rem}
\label{rem:Asymmetric-vanishing-rates}Asymmetric vanishing rates
can and do occur at a cusp. For example, assume that $k\geq2$, $b_{2},b_{3},\cdots,b_{2k}\in\mathbb{R}$,
$c_{1}=0$, and $c_{2}\in\mathbb{R}\setminus\left\{ 0\right\} $,
say, $c_{2}>0$. The formula (\ref{eq:c_n}) implies that the coefficients
$c_{0},c_{1},\cdots,c_{2k-1}$ are real and $\Im c_{2k}=b_{1}\Im a_{2k}<0$.
Thus, after choosing the appropriate branch for the inverse of $\Psi$,
the map $F_{\rho_{1}}$ can be represented by a convergent series
$F_{\rho_{1}}(z)=\alpha+d_{1}(z-c_{0})^{1/2}+\sum_{n=2}^{\infty}d_{n}(z-c_{0})^{n/2}$
for $z$ sufficiently close to $c_{0}$ and $\arg(z-c_{0})\in[0,2\pi)$,
in which we also have that $d_{1}=1/\sqrt{c_{2}}$, $d_{2},d_{3},\cdots,d_{2k-2}\in\mathbb{R}$,
and $\Im d_{2k-1}=-2^{-1}b_{1}c_{2}^{-k-1/2}\Im a_{2k}>0$. It follows
that the first $2k-1$ Taylor coefficients of $G_{\mu_{1}}\circ F_{\rho_{1}}$
are real and hence they do not contribute to the asymptotics of the
imaginary part $\Im G_{\mu_{1}}\circ F_{\rho_{1}}(x)$ as $x\rightarrow c_{0}^{+}$.
Thus, depending on whether $x>c_{0}$ or not, we get
\[
p_{\mu_{1}\boxplus\mu_{2}}(x)=\frac{-a_{1}}{\pi\sqrt{c_{2}}}\left|x-c_{0}\right|^{1/2}\left[1+o(1)\right]\qquad(x\rightarrow c_{0}^{-}),
\]
and 
\[
p_{\mu_{1}\boxplus\mu_{2}}(x)=\frac{-\Im a_{2k}}{2\pi c_{2}^{k+1/2}}\left|x-c_{0}\right|^{k-1/2}\left[1+o(1)\right]\qquad(x\rightarrow c_{0}^{+}).
\]
The case $c_{2}<0$ also leads to an asymmetric asymptotics with a
similar argument.
\end{rem}

\end{rem}

\end{rem}

\begin{example}
(Free Brownian motion) The results in \cite{Bi-cusp,Claeys} can now
be seen as a special case of Theorem \ref{thm: +}. Indeed, let us
consider $\mu_{\beta}=\nu_{\boxplus}^{\gamma,\sigma}$ where $\gamma=0$
and $\sigma=\beta\delta_{0}$. The family $\{\mu_{1}\boxplus\mu_{\beta}:\beta>0\}$
is the marginal of the free Brownian motion starting at $\mu_{1}\in\mathcal{P}_{\mathbb{R}}$.
Here we have $R_{\mu_{\beta}}(z)=\beta z$ with $b_{1}=\beta$ and
$b_{n}=0$ for all $n\geq2$. Assume that $\mu_{1}$ satisfies the
regularity hypothesis of Lemma \ref{lem:+} at $\alpha\in\mathbb{R}$.
The evolution of $p_{\mu_{1}\boxplus\mu_{\beta}}$ with parameter
$\beta$ is described as follows.
\begin{enumerate}
\item When $\beta<-a_{1}^{-1}$, we have $c_{1}>0$ and the density $p_{\mu_{1}\boxplus\mu_{\beta}}$
is real-analytic at the zero $c_{0}=\alpha+\beta a_{0}$ of order
$\geq2$ by Remark \ref{rem:+ analytic}.
\item When $\beta=-a_{1}^{-1}$, $p_{\mu_{1}\boxplus\mu_{\beta}}$ is no
longer smooth at the zero $c_{0}$. For $k\geq2$, the point $c_{0}$
is a cusp of $1/3$-exponent if $a_{2}=\int_{\mathbb{R}}\frac{d\mu_{1}(t)}{(\alpha-t)^{3}}=0$;
otherwise, it a cusp with asymmetric vanishing rates as illustrated
in Remark \ref{rem:Asymmetric-vanishing-rates}. By Remark \ref{rem:1/2 exp}(i),
the zero $c_{0}$ is a cusp of $1/2$-exponent if $k=1$.
\item When $\beta>-a_{1}^{-1}$, the singular behavior disappears as
the density $p_{\mu_{1}\boxplus\mu_{\beta}}$ is positive and analytic
at the point $c_{0}$. 
\end{enumerate}
This evolution is referred to as the propagation of singularity for
GUE perturbation in \cite{Claeys}. 
\begin{example}
(Free Poisson laws and Wishart perturbation) Given $\lambda\geq1$
and $\beta>0$, denote by $\mu_{\beta,\lambda}$ the free Poisson distribution
with $R$-transform $R_{\mu_{\beta,\lambda}}(z)=\lambda\beta/(1-\beta z)$.
We assume the regularity of $\mu_{1}$ at $\alpha\in\mathbb{R}$ as
in Lemma \ref{lem:+} and consider only those parameters $\beta$ satisfying the condition $\beta a_{0}\neq 1$. It follows that $b_{n}=\lambda\left[\beta/(1-\beta a_{0})\right]^{n+1}$
for $n=0,1,2,\cdots$. The non-dengeneracy of $\mu_{1}$, the Cauchy-Schwarz
inequality, and the hypothesis $\lambda\geq 1$ together yield 
\[
a_{0}^{2}+\lambda a_{1}=\lambda^{2}\left(\left[\int_{\mathbb{R}}\frac{1}{\alpha-t}\frac{d\mu_{1}(t)}{\lambda}\right]^{2}-\int_{\mathbb{R}}\frac{1}{(\alpha-t)^{2}}\frac{d\mu_{1}(t)}{\lambda}\right)<0.
\]
Accordingly, we introduce the positive number 
\[
\beta_{0}=\left(a_{0}+\sqrt{\lambda\left|a_{1}\right|}\right)^{-1}.
\] For $\beta>0$ and $\beta a_{0}\neq1$, note that 
\[
c_{1}=1+b_{1}a_{1}=\frac{(a_{0}^{2}+\lambda a_{1})\beta^{2}-2a_{0}\beta+1}{(1-\beta a_{0})^{2}}.
\]
The evolution of $p_{\mu_{1}\boxplus\mu_{\beta,\lambda}}$ goes as
follows.
\begin{enumerate}
\item ($\beta<\beta_{0}$) We have $c_{1}>0$ and hence $p_{\mu_{1}\boxplus\mu_{\beta,\lambda}}$
is analytic at the zero $c_{0}$ with order $\geq2$.
\item ($\beta=\beta_{0}$) We have $c_{1}=0$ and the following subcases:
\begin{enumerate}
\item ($k=1$) The point $c_{0}$ is a cusp of $1/2$-exponent.
\item ($k\geq2$) If $a_{2}=0$, the point $c_{0}$ is a cusp with asymmetric
vanishing rates. If $a_{2}\neq0$ and $a_{2}=-\lambda^{-1/2}\left|a_{1}\right|^{3/2}$,
then $c_{2}=0$ and we have 
\[
c_{3}=\lambda b_{1}\left(\left[\int_{\mathbb{R}}\frac{1}{(\alpha-t)^{2}}\frac{d\mu_{1}(t)}{\lambda}\right]^{2}-\int_{\mathbb{R}}\frac{1}{(\alpha-t)^{4}}\frac{d\mu_{1}(t)}{\lambda}\right)<0.
\]
It follows that $c_{0}$ is a cusp of $1/3$-exponent.
\end{enumerate}
\item ($\beta>\beta_{0}$) The density $p_{\mu_{1}\boxplus\mu_{\beta,\lambda}}$
is analytic and positive at $c_{0}$. 
\end{enumerate}
Thus, the propagation of singular behavior holds for Wishart matrices
as well. For example, let $B_{M,N}=\left[b_{ij}\right]$ be a non-selfadjoint
Guassian $M\times N$ random matrix whose covariance is specified
by $E\left[b_{ij}\overline{b_{kl}}\right]=N^{-1}\delta_{ik}\delta_{jl}$
and $E\left[b_{ij}b_{kl}\right]=0$ for all $1\leq i,k\leq M$ and
$1\leq j,l\leq N$. Suppose that the ratio $M/N\rightarrow\lambda=20/11$
as $M,N\rightarrow+\infty$, and that $H_{N}$ is a selfadjoint random
matrix which is asymptotically free from the Wishart matrix $B_{M,N}^{*}B_{M,N}$
and has the limiting eigenvalue distribution 
\[
d\mu_{1}(t)=\frac{5}{33}t^{4}\,dt,\quad -2\leq t \leq 1.
\]
Then, as $M,N\rightarrow\infty$, the limiting eigenvalue density
of the sum $H_{N}+\beta\,B_{M,N}^{*}B_{M,N}$ has a cusp of $1/3$-exponent
at the point $2$ in the support, as soon as $\beta$ reaches $44/65$.
\end{example}

\end{example}

We next treat the free multiplicative convolution on $\mathbb{R}_{+}$. 
\begin{lem}
\label{lem:*}Let $\mu_{1}\in\mathcal{P}_{\mathbb{R}_{+}}$ and $\alpha\in(0,+\infty)$.
Suppose that there exists an open interval $I=(\alpha^{-1}-\delta,\alpha^{-1}+\delta)\subset(0,+\infty)$
such that $d\mu_{1}(t)=A(1-\alpha t)^{2k}h(t)\,dt$ on $I$, where $k\in\mathbb{N}$,
$A>0$, and $h$ is real-analytic in $I$ with $h(\alpha^{-1})=1$.
Denote by $a_{0}$ the real number $\psi_{\mu_{1}}(\alpha)$. Suppose
that $\mu_{2}=\nu_{\boxtimes}^{\gamma,\sigma}$ is an $\boxtimes$-infinitely
divisible measure such that the integral form 
\[
u_{\sigma}(z)=\int_{[0,+\infty]}\frac{(1+z)+tz}{z-(1+z)t}\,d\sigma(t),\quad z\in\mathbb{C}\setminus\mathbb{R}_{+},
\]
extends analytically from $\mathbb{H}$ to a disk centered at the
point $a_{0}$ and $u_{\sigma}(a_{0})\in\mathbb{R}$. Then: 
\begin{enumerate}
\item The moment generating function $\psi_{\mu_{1}}$ extends analytically
from $\mathbb{H}$ to a disk centered at $\alpha$. If $\psi_{\mu_{1}}(z)=\sum_{n=0}^{\infty}a_{n}(z-\alpha)^{n}$
denotes the power series representation of $\psi_{\mu_{1}}$ at $z=\alpha$,
then 
\[
a_{0}=\int_{\mathbb{R}_{+}}\frac{\alpha t}{1-\alpha t}\,d\mu_{1}(t),\quad a_{1}=\begin{cases}
\int_{\mathbb{R}_{+}}\frac{t}{(1-\alpha t)^{2}}\,d\mu_{1}(t) & \text{if }a_{0}\neq-1,\\
\frac{1}{\alpha}\int_{\mathbb{R}_{+}}\frac{1}{(1-\alpha t)^{2}}\,d\mu_{1}(t) & \text{if }a_{0}=-1,
\end{cases}
\]
\[
a_{n}=\int_{\mathbb{R}_{+}}\frac{t^{n}}{(1-\alpha t)^{n+1}}\,d\mu_{1}(t),\quad n=2,\cdots,2k-1,
\]
and $a_{2k}\in\mathbb{H}$.
\item If $u_{\sigma}(z)=\sum_{n=0}^{\infty}b_{n}(z-a_{0})^{n}$ is the power
series representation of $u_{\sigma}$ in a disk centered at $a_{0}$,
then 
\[
b_{1}=\begin{cases}
-(1-\eta_{\mu_{1}}(\alpha))^{2}\int_{[0,+\infty]}\frac{1+t^{2}}{(\eta_{\mu_{1}}(\alpha)-t)^{2}}\,d\sigma(t) & \text{if }a_{0}\neq-1,\\
-\int_{[0,+\infty]}(1+t^{2})\,d\sigma(t) & \text{if }a_{0}=-1.
\end{cases}
\]
Here $\eta_{\mu_{1}}(\alpha)=a_{0}(1+a_{0})^{-1}$ if $a_{0}\neq-1$.
\item The function $\Psi(z)=\gamma z\exp\left[(u_{\sigma}\circ\psi_{\mu_{1}})(z)\right]$,
$z\in\mathbb{H}$, extends analytically to a disk centered at $\alpha$,
and the first four Taylor coefficients of $\Psi(z)$ at $z=\alpha$
are $c_{0}=\gamma\alpha e^{b_{0}}$, $c_{1}=c_{0}(\alpha^{-1}+b_{1}a_{1})$,
$c_{2}=2^{-1}c_{1}^{2}c_{0}^{-1}+2^{-1}c_{0}(-\alpha^{-2}+2b_{1}a_{2}+2b_{2}a_{1}^{2})$,
and 
\[
c_{3}=c_{1}c_{2}c_{0}^{-1}-3^{-1}c_{1}^{3}c_{0}^{-2}+3^{-1}c_{0}(\alpha^{-3}+3b_{1}a_{3}+6b_{2}a_{1}a_{2}+3b_{3}a_{1}^{3}).
\]
\end{enumerate}
\end{lem}

\begin{proof}
Let $J=\left\{ t^{-1}:t\in I\right\} $ and recall from Section 2
that $d(\mu_{1})_{*}(t)=d\mu_{1}(t^{-1})$. Since 
\[
\frac{d(\mu_{1})_{*}}{dt}(x)=\frac{1}{x^{2}}\frac{d\mu_{1}}{dt}\left(\frac{1}{x}\right)=A(x-\alpha)^{2k}\frac{h(1/x)}{x^{2k+2}},\quad x\in J,
\]
and 
\[
\psi_{\mu_{1}}(z)=\int_{\mathbb{R}_{+}\setminus I}\frac{tz}{1-tz}\,d\mu_{1}(t)+z\int_{J}\frac{1}{t-z}\,d(\mu_{1})_{*}(t),\quad z\in\mathbb{H},
\]
Proposition \ref{prop:Greenstein} shows that the function $\psi_{\mu_{1}}$
extends analytically from $\mathbb{H}$ to an open disk $D$ centered
at $\alpha$. Moreover, there exists an analytic function $g$ in
$D$ such that $g(x)\in\mathbb{R}$ for all $x\in D\cap\mathbb{R}_{+}$,
$g(\alpha)=1$, and $\psi_{\mu_{1}}(z)=\overline{\psi_{\mu_{1}}(\overline{z})}+i2\pi A\alpha^{-2k-1}(z-\alpha)^{2k}g(z)$
for all $z\in D\cap(-\mathbb{H})$. It follows that $a_{0},a_{1},\cdots,a_{2k-1}\in\mathbb{R}$
and $\Im a_{2k}=\pi A\alpha^{-2k-1}>0$. As seen in the proof of Lemma
\ref{lem:+}, for $0\leq n\leq2k-1$, the regularity assumption of
$\mu_{1}$ at $\alpha$ and the dominated convergence theorem imply
that 
\[
a_{n}=\lim_{\varepsilon\downarrow0}\frac{\psi_{\mu_{1}}^{(n)}(\alpha+i\varepsilon)}{n!}=\lim_{\varepsilon\downarrow0}\int_{\mathbb{R}_{+}}\frac{t^{n}}{(1-t(\alpha+i\varepsilon))^{n+1}}\,d\mu_{1}(t)=\int_{\mathbb{R}_{+}}\frac{t^{n}}{(1-\alpha t)^{n+1}}\,d\mu_{1}(t).
\]
The alternative integral formula of $a_{1}$ follows from the observation:
\begin{eqnarray*}
a_{1}=\int_{\mathbb{R}_{+}}\frac{t}{(1-\alpha t)^{2}}\,d\mu_{1}(t) & = & \frac{1}{\alpha}\left[\int_{\mathbb{R}_{+}}\frac{1}{(1-\alpha t)^{2}}\,d\mu_{1}(t)-\int_{\mathbb{R}_{+}}\frac{1}{1-\alpha t}\,d\mu_{1}(t)\right]\\
 & = & \frac{1}{\alpha}\int_{\mathbb{R}_{+}}\frac{1}{(1-\alpha t)^{2}}\,d\mu_{1}(t)-\frac{1}{\alpha}\left[a_{0}+1\right].
\end{eqnarray*}

We now drive the formula of $b_{1}$, according to whether $a_{0}=-1$
or not. Assume first that $a_{0}\neq-1$, so that $\eta_{\mu_{1}}(\alpha)=a_{0}/(1+a_{0})\in\mathbb{R}\setminus\left\{ 1\right\} $.
Then the analyticity of $u_{\sigma}$ at $a_{0}$ implies that the
function 
\[
N(z)=u_{\sigma}\left(\frac{z}{1-z}\right)=\int_{[0,+\infty]}\frac{1+tz}{z-t}\,d\sigma(t),\quad z\in\mathbb{C}\setminus\mathbb{R}_{+},
\]
extends analytically to a neighborhood of $\eta_{\mu_{1}}(\alpha)$
and $N\left(\eta_{\mu_{1}}(\alpha)\right)=u_{\sigma}(a_{0})\in\mathbb{R}$.
Thus, the Julia-Carath\'{e}odory derivative $N^{\prime}(\eta_{\mu_{1}}(\alpha))$
exists and is equal to the complex first derivative $N^{(1)}(\eta_{\mu_{1}}(\alpha))$.
Argue as in the proof of Lemma \ref{lem:+}, it is easy to see that
\[
N^{\prime}(\eta_{\mu_{1}}(\alpha))=-\int_{[0,+\infty]}\frac{1+t^{2}}{(\eta_{\mu_{1}}(\alpha)-t)^{2}}\,d\sigma(t).
\]
Finally, the chain rule yields 
\[
N^{\prime}(\eta_{\mu_{1}}(\alpha))=N^{(1)}(\eta_{\mu_{1}}(\alpha))=u_{\sigma}^{(1)}(a_{0})\cdot\frac{1}{(1-\eta_{\mu_{1}}(\alpha))^{2}}=\frac{b_{1}}{(1-\eta_{\mu_{1}}(\alpha))^{2}},
\]
as desired. The case of $a_{0}=-1$ is proved in the same way by using
the fact that 
\[
u_{\sigma}(z)=I\left(\frac{1+z}{z}\right),\quad\text{where}\quad I(w)=\int_{[0,+\infty]}\frac{w+t}{1-tw}\,d\sigma(t),\quad w\in\mathbb{C}\setminus\mathbb{R}_{+}.
\]

The formulas in (3) come from a simple calculation and we omit
the details. 
\end{proof}
We remark that the map $\Psi$ is precisely the analytic continuation
of the inverse $\eta_{\rho_{1}}^{\left\langle -1\right\rangle }$
in Lemma \ref{lem:trade a free convolution for another}. The proof of the next result is identical to that of
Theorem \ref{thm: +}, with $F_{\rho_{1}}$ and $G_{\mu_{1}}$ replaced
by $\eta_{\rho_{1}}$ and $\psi_{\mu_{1}}$, respectively. To avoid
repetition, we will not pursue the details. Recall that $q_{\mu_{1}\boxtimes\mu_{2}}$
denotes the density of $\mu_{1}\boxtimes\mu_{2}$ with respect to
the Haar measure $dx/x$, and one has 
\[
q_{\mu_{1}\boxtimes\mu_{2}}(1/x)=\frac{1}{\pi}\Im\psi_{\mu_{1}}\left(\eta_{\rho_{1}}(x)\right),\quad x\in(0,+\infty).
\]

\begin{thm}
Suppose that $\mu_{1}$ and $\mu_{2}$ satisfy the hypotheses of \emph{Lemma}
\emph{\ref{lem:*}}. Assume in addition that $c_{1}=0$. Then $c_{0}^{-1}$
is an isolated zero of $q_{\mu_{1}\boxtimes\mu_{2}}$ in the interior
$[\text{\emph{supp}}(\mu_{1}\boxtimes\mu_{2})]^{\circ}$ such that 
\begin{enumerate}
\item If $c_{2}=0$ and $\Re c_{3}<0$, then 
\[
\lim_{x\uparrow c_{0}^{-1}}\frac{q_{\mu_{1}\boxtimes\mu_{2}}(x)}{\left|x-c_{0}^{-1}\right|^{1/3}}=\frac{a_{1}}{\pi\sqrt[3]{\left|c_{3}\right|}}\cos\left(\frac{\theta}{3}-\frac{\pi}{6}\right),
\]
and 
\[
\lim_{x\downarrow c_{0}^{-1}}\frac{q_{\mu_{1}\boxtimes\mu_{2}}(x)}{\left|x-c_{0}^{-1}\right|^{1/3}}=\frac{a_{1}}{\pi\sqrt[3]{\left|c_{3}\right|}}\sin\frac{\theta}{3},
\]
where $\theta=\arg c_{3}\in(\pi/2,3\pi/2)$. 
\item If $\Im c_{2}<0$, then 
\[
\lim_{x\uparrow c_{0}^{-1}}\frac{q_{\mu_{1}\boxtimes\mu_{2}}(x)}{\left|x-c_{0}^{-1}\right|^{1/2}}=\frac{a_{1}}{\pi\sqrt{\left|c_{2}\right|}}\sin\frac{\theta}{2},
\]
and 
\[
\lim_{x\downarrow c_{0}^{-1}}\frac{q_{\mu_{1}\boxtimes\mu_{2}}(x)}{\left|x-c_{0}^{-1}\right|^{1/2}}=\frac{-a_{1}}{\pi\sqrt{\left|c_{2}\right|}}\cos\frac{\theta}{2},
\]
where $\theta=\arg c_{2}\in(\pi,2\pi)$. 
\end{enumerate}
\end{thm}

Some remarks are in order. First, one can verify that cusps with asymmetric asymptotic behavior, similar to those discussed in
Remark \ref{rem:Asymmetric-vanishing-rates}, can occur for $\mu_{1}\boxtimes\mu_{2}$.
Second, it is now fairly easy to construct a cusp of $1/3$-exponent in the
support of $\mu_{1}\boxtimes\mu_{2}$. For example, we take 
\[
d\mu_{1}(t)=\frac{5}{29\sqrt{2}-40}(1-t)^{4}\,dt,\quad 0\leq t \leq \sqrt{2},
\]
and $\mu_{2}=\nu_{\boxtimes}^{\gamma,\sigma}$ where $\gamma=1$ and
$\sigma=\frac{87\sqrt{2}-120}{60-40\sqrt{2}}\,\delta_{1}$, that is,
$\mu_{2}$ is a multiplicative analog of the semicircle law. Then
the free convolution density $q_{\mu_{1}\boxtimes\mu_{2}}$ has a
cusp of $1/3$-exponent at the point $c_{0}^{-1}\doteqdot3.44$. 

We move on to the results on $\mathbb{T}$. Recall that $d(\mu_{1})_{*}(t)=d\mu_{1}\left(\overline{t}\right)$
and $m=d\theta/2\pi$. The function $\Psi$ below is the analytic
continuation of the inverse $\eta_{\rho_{1}}^{\left\langle -1\right\rangle }$
in Lemma \ref{lem:trade a free convolution for another-1}.
\begin{lem}
\label{lem:T}Let $\mu_{1}\in\mathcal{P}_{\mathbb{T}}$ and $\alpha\in\mathbb{T}$.
Assume that $(\mu_{1})_{*}$ is absolutely continuous with an analytic
density $g$ (with respect to $m$) on an open arc $\Gamma\subset\mathbb{T}$,
and that $\alpha\in\Gamma$ is a zero of even order for $g$. Denote
by $a_{0}$ the complex number $\psi_{\mu_{1}}(\alpha)$. Suppose
that $\mu_{2}=\nu_{\boxtimes}^{\gamma,\sigma}$ is an $\boxtimes$-infinitely
divisible measure such that the integral form 
\[
u_{\sigma}(z)=\int_{\mathbb{T}}\frac{t(1+z)+z}{t(1+z)-z}\,d\sigma(t),\quad\Re z>-\frac{1}{2},
\]
extends analytically from the half-plane $\left\{ z\in\mathbb{C}:\Re z>-1/2\right\} $
to a disk centered at the point $a_{0}$ and $\Re u_{\sigma}(a_{0})=0$.
Let $u_{\sigma}(z)=\sum_{n=0}^{\infty}b_{n}(z-a_{0})^{n}$ denote
the power series representation of $u_{\sigma}$ nearby $z=\alpha$.
Then: 
\begin{enumerate}
\item The functions $\psi_{\mu_{1}}$ and $\eta_{\mu_{1}}$ extend analytically
to a disk centered at $\alpha$. If $\psi_{\mu_{1}}(z)=\sum_{n=0}^{\infty}a_{n}(z-\alpha)^{n}$
denotes the power series representation of $\psi_{\mu_{1}}$ at $z=\alpha$,
then 
\[
a_{0}=-\frac{1}{2}+i\int_{\mathbb{T}}\frac{\Im(\alpha t)}{\left|1-\alpha t\right|^{2}}\,d\mu_{1}(t)\quad\text{and}\quad a_{1}=-\overline{\alpha}\int_{\mathbb{T}}\frac{1}{\left|1-\alpha t\right|^{2}}\,d\mu_{1}(t).
\]
It follows that $\eta_{\mu_{1}}(\alpha)\in\mathbb{T}\setminus\left\{ 1\right\} $,
$\eta_{\mu_{1}}(\xi)\in\mathbb{D}$ for all $\xi\neq\alpha$ sufficiently
close to $\alpha$, and 
\[
c=\liminf_{r\uparrow1}\frac{1-\left|\eta_{\mu_{1}}(r\alpha)\right|}{1-r}\in(0,+\infty).
\]
\item The function $\Psi(z)=\gamma z\exp\left[(u_{\sigma}\circ\psi_{\mu_{1}})(z)\right]$,
$z\in\mathbb{D}$, extends analytically to a disk centered at $\alpha$,
and the first four Taylor coefficients of $\Psi(z)$ at $z=\alpha$
are $c_{0}=\gamma\alpha e^{b_{0}}$, 
\[
c_{1}=c_{0}\overline{\alpha}\left(1-2c\int_{\mathbb{T}}\frac{d\sigma(t)}{\left|\eta_{\mu_{1}}(\alpha)-t\right|^{2}}\right),
\]
$c_{2}=2^{-1}c_{1}^{2}\overline{c_{0}}+2^{-1}c_{0}(-\overline{\alpha}^{2}+2b_{1}a_{2}+2b_{2}a_{1}^{2})$,
and 
\[
c_{3}=c_{1}c_{2}\overline{c_{0}}-3^{-1}c_{1}^{3}\overline{c_{0}}^{2}+3^{-1}c_{0}(\overline{\alpha}^{3}+3b_{1}a_{3}+6b_{2}a_{1}a_{2}+3b_{3}a_{1}^{3}).
\]
\end{enumerate}
\end{lem}

\begin{proof}
Assume $\alpha\neq1$. It is known from \cite{Akhiezer} that the
transformations 
\[
w=i\frac{1+z}{1-z}\quad\text{and}\quad a+bw-N_{\sigma^{\prime}}(w)=iH_{(\mu_{1})_{*}}(z),\quad z\in\mathbb{D},
\]
map the upper half-plane $\mathbb{H}$ to the disk $\mathbb{D}$ and
the Herglotz integral form 
\[
H_{(\mu_{1})_{*}}(z)=\int_{\mathbb{T}}\frac{t+z}{t-z}\,d(\mu_{1})_{*}(t)
\]
to an analytic self-map $a+bw-N_{\sigma^{\prime}}(w)$ on $\mathbb{C}^{+}$,
where $b=\mu_{1}\left(\left\{ 1\right\} \right)$ and $\sigma^{\prime}$
is the push-forward measure of $d(\mu_{1})_{*}(e^{i\theta})-\mu_{1}(\left\{ 1\right\} )\delta_{1}$
via the map $T(\theta)=-\cot(\theta/2)$, $\theta\in(0,2\pi)$. Proposition
\ref{prop:Greenstein} applies to $\sigma^{\prime}$ at the image
$T(\theta_{0})$ where $\alpha=e^{i\theta_{0}}$, whence $H_{(\mu_{1})_{*}}$
extends analytically to a neighborhood of $\alpha$. It follows that
$\psi_{\mu_{1}}$ extends analytically to a neighborhood of $\alpha$
by the identity $1+2\psi_{\mu_{1}}=H_{(\mu_{1})_{*}}$, and so does
$\eta_{\mu_{1}}$. If $\alpha=1$, the preceding argument shows that
$\psi_{\mu_{1}\boxtimes\delta_{-i}}$ extends analytically to a neighborhood
of $i$, and hence the function $\psi_{\mu_{1}}(z)=\psi_{\mu_{1}\boxtimes\delta_{-i}}(iz)$
extends analytically to a neighborhood of $1$. Since 
\[
\psi_{\mu_{1}}(z)=z\int_{\Gamma}\frac{g(t)}{t-z}\,dm(t)+z\int_{\mathbb{T}\setminus\Gamma}\frac{1}{t-z}\,d(\mu_{1})_{*}(t),
\]
the formulas of the coefficients $a_{0}$ and $a_{1}$ follow from
differentiation under the integral sign and the dominated convergence
as in the proof of Lemma \ref{lem:+}. The integral formula of $a_{0}$
shows that $\eta_{\mu_{1}}(\alpha)\in\mathbb{T}\setminus\left\{ 1\right\} $.
The linear approximation $\psi_{\mu_{1}}(\xi)=a_{0}+a_{1}(\xi-\alpha)+O(\left|\xi-\alpha\right|^{2})$
implies that $1+2\Re\psi_{\mu_{1}}(\xi)=\left|a_{1}\right|\left|\xi-\alpha\right|+O(\left|\xi-\alpha\right|^{2})>0$
for all $\xi\in\mathbb{T}\setminus\left\{ \alpha\right\} $ sufficiently
close to $\alpha$. This shows that $\eta_{\mu_{1}}(\xi)\in\mathbb{D}$
for $\xi\neq\alpha$. On the other hand,
since $\eta_{\mu_{1}}(\alpha)\in\mathbb{T}$ and $\eta_{\mu_{1}}$
extends analytically to a neighborhood of $\alpha$, the Julia-Carath\'{e}odory
derivative $\eta_{\mu_{1}}^{\prime}(\alpha)$ exists and is equal
to the complex first derivative $\eta_{\mu_{1}}^{(1)}(\alpha)$. Finally,
the identity 
\[
\frac{1-\left|\eta_{\mu_{1}}(z)\right|}{1-\left|z\right|}=\frac{1+2\Re\psi_{\mu_{1}}(z)}{\left(1+\left|\eta_{\mu_{1}}(z)\right|\right)\left|1+\psi_{\mu_{1}}(z)\right|^{2}\left(1-\left|z\right|\right)},\quad z\in\mathbb{D},
\]
and the approximation $\psi_{\mu_{1}}(r\alpha)=a_{0}+\left|a_{1}\right|(1-r)+O(\left|1-r\right|^{2})$
together imply that $c=\left|a_{1}\right|\left|1+a_{0}\right|^{-2}\in(0,+\infty)$.
The statement (1) is proved.

We next derive the integral formula of $c_{1}$. By the analyticity
assumption of $u_{\sigma}$, the Herglotz form $H_{\sigma}(z)=u_{\sigma}(z/(1-z))$
extends analytically to a neighborhood of the point $\xi_{0}=\eta_{\mu_{1}}(\alpha)\in\mathbb{T}\setminus\left\{ 1\right\} $,
and so does the analytic self-map 
\[
f(z)=z\exp\left(-H_{\sigma}(z)\right),\qquad z\in\mathbb{D}.
\]
Since $\Re H_{\sigma}(\xi_{0})=\Re u_{\sigma}(a_{0})=0$, we have
$f(\xi_{0})\in\mathbb{T}$ and hence the Julia-Carath\'{e}odory derivative
$f^{\prime}(\xi_{0})$ of $f$ at $\xi_{0}$ exists and is equal to
its complex first derivative 
\[
f^{(1)}(\xi_{0})=\overline{\xi_{0}}f(\xi_{0})\left[1-\xi_{0}H^{(1)}(\xi_{0})\right];
\]
in which case the Ahern and Clark theorem (cf. \cite{Cima}) shows
further that 
\[
f^{\prime}(\xi_{0})=\overline{\xi_{0}}f(\xi_{0})\left[1+2\int_{\mathbb{T}}\frac{d\sigma(t)}{\left|\xi_{0}-t\right|^{2}}\right].
\]
Therefore we obtain that 
\[
H^{(1)}(\xi_{0})=-2\overline{\xi_{0}}\int_{\mathbb{T}}\frac{d\sigma(t)}{\left|\xi_{0}-t\right|^{2}}.
\]
Finally, we use the facts that $u_{\sigma}\circ\psi_{\mu_{1}}=H_{\sigma}\circ\eta_{\mu_{1}}$
in $\mathbb{D}$ and $\eta_{\mu_{1}}^{(1)}(\alpha)=\eta_{\mu_{1}}^{\prime}(\alpha)=c\overline{\alpha}\xi_{0}$
to get
\begin{eqnarray*}
c_{1}=\lim_{r\uparrow1}\Psi^{(1)}(r\alpha) & = & \Psi(\alpha)\overline{\alpha}\left[1+\alpha H_{\sigma}^{(1)}(\xi_{0})\eta_{\mu_{1}}^{(1)}(\alpha)\right]\\
 & = & c_{0}\overline{\alpha}\left[1-2c\int_{\mathbb{T}}\frac{d\sigma(t)}{\left|\xi_{0}-t\right|^{2}}\right].
\end{eqnarray*}
The formulas for $c_{2}$ and $c_{3}$ follow from a straightforward
computation and we skip the details.
\end{proof}
In the next result, the notation $\xi\rightarrow c^{+}$ means $\xi\rightarrow c$
in the counterclockwise orientation on $\mathbb{T}$, while $\xi\rightarrow c^{-}$
denotes $\xi\rightarrow c$ in the clockwise orientation. Also, we
define the power $(z)^{1/3}=\left|z\right|^{1/3}e^{i(\arg z)/3}$
for $z\neq0$ and $\arg z\in[0,2\pi)$. 
\begin{thm}
Suppose that $\mu_{1}$ and $\mu_{2}$ satisfy the hypotheses of \emph{Lemma}
\emph{\ref{lem:T}}. Assume in addition that $c_{1}=0$. Then $\overline{c_{0}}$
is an isolated zero of $p_{\mu_{1}\boxtimes\mu_{2}}$ in the interior
$[\text{\emph{supp}}(\mu_{1}\boxtimes\mu_{2})]^{\circ}$ such that 
\begin{enumerate}
\item If $c_{2}=0$ and $\Re c_{3}>0$, then 
\[
\lim_{\xi\rightarrow\overline{c_{0}}^{+}}\frac{p_{\mu_{1}\boxtimes\mu_{2}}(\xi)}{\left|\xi-\overline{c_{0}}\right|^{1/3}}=\frac{-2\left|a_{1}\right|}{\sqrt[3]{\left|c_{3}\right|}}\cos\left(\frac{\theta}{3}-\frac{5\pi}{6}\right),
\]
and 
\[
\lim_{\xi\rightarrow\overline{c_{0}}^{-}}\frac{p_{\mu_{1}\boxtimes\mu_{2}}(\xi)}{\left|\xi-\overline{c_{0}}\right|^{1/3}}=\frac{-2\left|a_{1}\right|}{\sqrt[3]{\left|c_{3}\right|}}\cos\left(\frac{\theta}{3}-\frac{7\pi}{6}\right),
\]
where $\theta\in(-\pi/2,\pi/2)$. 
\item If $\Re c_{2}<0$, then 
\[
\lim_{\xi\rightarrow\overline{c_{0}}^{+}}\frac{p_{\mu_{1}\boxtimes\mu_{2}}(\xi)}{\left|\xi-\overline{c_{0}}\right|^{1/2}}=\frac{2\left|a_{1}\right|}{\sqrt{\left|c_{2}\right|}}\cos\left(\frac{\theta}{2}-\frac{\pi}{4}\right),
\]
and 
\[
\lim_{\xi\rightarrow\overline{c_{0}}^{-}}\frac{p_{\mu_{1}\boxtimes\mu_{2}}(\xi)}{\left|\xi-\overline{c_{0}}\right|^{1/2}}=\frac{2\left|a_{1}\right|}{\sqrt{\left|c_{2}\right|}}\cos\left(\frac{\theta}{2}-\frac{3\pi}{4}\right),
\]
where $\theta\in(\pi/2,3\pi/2)$. 
\end{enumerate}
\end{thm}

\begin{proof}
By Lemma \ref{lem:T} and the characterization of the set $\left\{ t\in\mathbb{T}:R(t)=1\right\} $
(see notes before Theorem \ref{thm:cusp on T}), there exists an open arc $\Gamma\subset\mathbb{T}$ such that $\alpha\in\Gamma$, $R(\alpha)=1$, and $R(t)<1$ for all
$t\in\Gamma\setminus\left\{ \alpha\right\} $. Thus, the density formula \eqref{eq:density boxtimes T} shows that $p_{\mu_{1}\boxtimes\mu_{2}}\left(\overline{\Psi(R(t)t)}\right)>0=p_{\mu_{1}\boxtimes\mu_{2}}\left(\overline{\Psi(\alpha)}\right)$
for all $t\in\Gamma\setminus\left\{ \alpha\right\} $, which means
that the point $\overline{c_{0}}=\overline{\Psi(\alpha)}$ is an isolated
zero of $p_{\mu_{1}\boxtimes\mu_{2}}$ in the interior $[\text{supp}(\mu_{1}\boxtimes\mu_{2})]^{\circ}$. 

We first prove Case (1) under an additional assumption that $\alpha=1$
and $c_{0}=1$. In this case the map $\Psi$ has a power series representation
\[
\Psi(z)=1+\left|c_{3}\right|e^{i\theta}(z-1)^{3}+\sum_{n=4}^{\infty}c_{n}(z-1)^{n}
\]
for $z$ nearby $1$, where $\theta\in(-\pi/2,\pi/2)$. So, as seen
in the proof of Theorem \ref{thm: +}, there exist functions $\eta_{0}$,
$\eta_{1}$, and $\eta_{2}$ defined by the convergent series 
\[
\eta_{m}(z)=1+\frac{e^{i2m\pi/3}}{\sqrt[3]{\left|c_{3}\right|}e^{i\theta/3}}(z-1)^{1/3}+e^{i2m\pi/3}\sum_{n=2}^{\infty}d_{n,m}(z-1)^{n/3},\quad m=0,1,2,
\]
such that $\Psi\left(\eta_{m}(z)\right)=z$ for $z$ sufficiently
close to $1$ and $\arg(z-1)\in[0,2\pi)$. Since 
\[
\Re\eta_{m}(r)=1+\frac{\sqrt[3]{1-r}}{\sqrt[3]{\left|c_{3}\right|}}\cos\left(\frac{2m\pi}{3}-\frac{\theta}{3}+\frac{\pi}{3}\right)+o(1)\quad\text{as }r\rightarrow1^{-},
\]
and $\Re\eta_{\rho_{1}}(r)<1$ for $r\in(0,1)$, we conclude that
$\eta_{1}$ coincides with $\eta_{\rho_{1}}$ near the point $1$.
Note that 
\[
z-1=\begin{cases}
\left|z-1\right|e^{i\pi/2}e^{i\tau/2} & \text{if }z=e^{i\tau},\;0<\tau<\pi/4,\\
\left|z-1\right|e^{i3\pi/2}e^{i\tau/2} & \text{if }z=e^{i\tau},\;7\pi/4<\tau<2\pi.
\end{cases}
\]
It follows, as $z\rightarrow1$ on $\mathbb{T}$, that 
\begin{eqnarray*}
p_{\mu_{1}\boxtimes\mu_{2}}\left(\overline{z}\right) & = & 1+2\Re\psi_{\mu_{1}}\left(\eta_{\rho_{1}}(z)\right)=1+2\Re\psi_{\mu_{1}}\left(\eta_{1}(z)\right)\\
 & = & \frac{-2\left|a_{1}\right|}{\sqrt[3]{\left|c_{3}\right|}}\Re\left[e^{(2\pi-\theta)i/3}\left(z-1\right)^{1/3}\right]+O\left(\left|z-1\right|^{2/3}\right)\\
 & = & \begin{cases}
\frac{-2\left|a_{1}\right|}{\sqrt[3]{\left|c_{3}\right|}}\cos\left(\frac{\theta}{3}-\frac{5\pi}{6}\right)\left|z-1\right|^{1/3}(1+o(1)) & \text{if }z\rightarrow1^{-},\\
\frac{-2\left|a_{1}\right|}{\sqrt[3]{\left|c_{3}\right|}}\cos\left(\frac{\theta}{3}-\frac{7\pi}{6}\right)\left|z-1\right|^{1/3}(1+o(1)) & \text{if }z\rightarrow1^{+}.
\end{cases}
\end{eqnarray*}
In the general case, we define $\mu_{1}^{\prime}=\mu_{1}\boxtimes\delta_{\alpha}$
and $\mu_{2}^{\prime}=\mu_{2}\boxtimes\delta_{\gamma e^{b_{0}}}$,
so that $\mu_{1}^{\prime}\boxtimes\mu_{2}^{\prime}=\mu_{1}\boxtimes\mu_{2}\boxtimes\delta_{c_{0}}$
and the preceding asymptotics applies to the density $p_{\mu_{1}^{\prime}\boxtimes\mu_{2}^{\prime}}$
of $\mu_{1}^{\prime}\boxtimes\mu_{2}^{\prime}$. Since 
\[
p_{\mu_{1}^{\prime}\boxtimes\mu_{2}^{\prime}}\left(\overline{z}\right)=1+2\Re\psi_{\mu_{1}^{\prime}\boxtimes\mu_{2}^{\prime}}(z)=1+2\Re\psi_{\mu_{1}\boxtimes\mu_{2}}\left(c_{0}z\right)=p_{\mu_{1}\boxtimes\mu_{2}}\left(\overline{c_{0}z}\right),
\]
the proof of Case (1) is finished after we make the substitution $\xi=\overline{c_{0}z}$.

Case (2) is proved in the same way, and we skip the details.
\end{proof}

\end{document}